\documentclass{amsart}
\usepackage{amssymb,amsmath}
\usepackage[american]{babel}
\usepackage{amsopn}
\usepackage{mathrsfs}
\usepackage{color}
\usepackage{wasysym}
\usepackage{vmargin}
\usepackage{amscd}
\usepackage{pdfsync}
\usepackage{verbatim}
\usepackage{CJKutf8}
\usepackage[pdfstartview=FitH,colorlinks=true]{hyperref}
\hypersetup{urlcolor=red, linkcolor=blue, citecolor=red, CJKbookmarks=true}

\theoremstyle{plain}
\newtheorem{theorem}{Theorem}[section]
\newtheorem{remark}[theorem]{Remark}
\newtheorem{lemma}[theorem]{Lemma}
\newtheorem{proposition}[theorem]{Proposition}
\newtheorem{definition}[theorem]{Definition}
\newtheorem{corollary}[theorem]{Corollary}

\numberwithin{equation}{section}

\begin{document}

\title[Non-cutoff Kac equation
in critical Besov space]
{The Cauchy problem for the inhomogeneous\\ non-cutoff Kac equation
in critical Besov space}

\author[H.-M. Cao, H.-G. Li, C.-J. Xu and J. Xu]
{Hongmei Cao,\, Hao-Guang Li,\,  Chao-Jiang Xu\, and Jiang Xu}

\address{Hong-Mei Cao,
\newline\indent
Department of Mathematics, Nanjing University of Aeronautics and Astronautics, Nanjing 211106,
P. R. China
\newline\indent
Universit\'e de Rouen, CNRS UMR 6085, Laboratoire de Math\'ematiques
\newline\indent
76801 Saint-Etienne du Rouvray, France
}
\email{hmcao\underline{ }91@nuaa.edu.cn}

\address{Hao-Guang Li,
\newline\indent
School of Mathematics and Statistics, South-Central University for Nationalities
\newline\indent
430074, Wuhan, P. R. China}
\email{lihaoguang@mail.scuec.edu.cn}

\address{Chao-Jiang Xu,
\newline\indent
School of Mathematics and statistics, Wuhan University 430072,
Wuhan, P. R. China
\newline\indent
Universit\'e de Rouen, CNRS UMR 6085, Laboratoire de Math\'ematiques
\newline\indent
76801 Saint-Etienne du Rouvray, France
}
\email{chao-jiang.xu@univ-rouen.fr}

\address{Jiang Xu,
\newline\indent
Department of Mathematics, Nanjing University of Aeronautics and Astronautics, Nanjing 211106,
P. R. China
}
\email{jiangxu\underline{ }79@nuaa.edu.cn}

\date{\today}



\begin{abstract}
In this work, we investigate the Cauchy problem for the spatially inhomogeneous non-cutoff Kac equation. If the initial datum belongs to the spatially critical Besov space, we can prove the well-posedness of weak solution under a perturbation framework. Furthermore, it is shown that the solution enjoys Gelfand-Shilov regularizing properties with respect to the velocity variable and Gevrey regularizing properties with respect to the position variable. In comparison with the recent result in \cite{LMPX2},  the Gelfand-Shilov regularity index is improved to be optimal. To the best of our knowledge, our work is
the first one that exhibits smoothing effect for the kinetic equation in Besov spaces.
\end{abstract}

\maketitle

\tableofcontents

\section{Introduction}
In this work, we consider the spatially inhomogeneous non-cutoff Kac equation
\begin{equation}\label{eq1.10}
\left\{
\begin{array}{ll}
   \partial_t f+v\partial_x f=K(f,f),\\
   f|_{t=0}=f_0,
\end{array}
\right.
\end{equation}
where $f=f(t, x, v)\ge0$ is the density distribution function depending on the position $x\in\mathbb{R}$, velocity $v\in\mathbb{R}$ and time $t\geq0$. The Kac collision operator is given by
\begin{equation*}
K(f,g)(v)
=\int_{|\theta|\leq\frac{\pi}{4}}\beta(\theta)\left(\int_{\mathbb{R}}\left(f_{*}'g'-f_{*}g\right)dv_{*}\right)d\theta,
\end{equation*}
with the standard shorthand $f_{*}'=f(t, x, v_{*}'), f'=f(t, x, v'), f_{*}=f(t, x, v_{*}), f=f(t, x, v)$, where the pre and post collision velocities can be defined by
$$
v'+iv_{*}'=e^{i\theta}(v+iv_{*}), \ i.e., \ \ v'=v\cos\theta-v_{*}\sin\theta,\ \ v_{*}'=v\sin\theta+v_{*}\cos\theta, \ v,v_{*}\in\mathbb{R}.
$$
Indeed, the relation follows from the conversation of the kinetic energy in the binary collision
$$
v^{2}+v_{*}^{2}=v'^{2}+v_{*}'^{2}.
$$
We consider a cross section with a non-integrable singularity of the type
\begin{equation}\label{A-0}
\beta(\theta)\underset{\theta\rightarrow0}\approx|\theta|^{-1-2s}
\end{equation}
for some given parameter $0<s<1$. For more details on the physics background, the reader is referred to \cite{Cerci,V4} and references therein.

We study the Kac equation \eqref{eq1.10} around the normalized Maxwellian distribution
$$
\mu(v)=(2\pi)^{-\frac 12}e^{-\frac{|v|^{2}}{2}}, \ \ v\in\mathbb{R}.
$$
In a close to equilibrium framework, considering the fluctuation of density distribution function
$$
f(t, x, v)=\mu(v)+\sqrt{\mu}(v)g(t, x, v).
$$
Note that $K(\mu,\mu)=0$ by conservation of the kinetic energy, we turn to the following Cauchy problem
\begin{equation} \label{eq-1}
\left\{
\begin{aligned}
&\partial_t g+v\partial_xg+\mathcal{K}g=\Gamma(g, g),\,\,\, t>0,\, v\in\mathbb{R},\\
&g|_{t=0}=g_{0},
\end{aligned} \right.
\end{equation}
where
$$
\mathcal{K}(g)\triangleq\mathcal{K}_1(g)+\mathcal{K}_2(g),
$$
with
$$
\mathcal{K}_1(g)=-\mu^{-1/2}K(\mu,\mu^{1/2}g),\quad \mathcal{K}_2(g)=-\mu^{-1/2}K(\mu^{1/2}g,\mu)
$$
and
$$
\Gamma(g,g)=\mu^{-1/2}K(\mu^{1/2}g,\mu^{1/2}g).
$$
The linearized Kac operator $\mathcal{K}$ has been investigated by Lerner, Morimoto, Pravda-Starov and Xu in \cite{LMPX1}. It is shown that $\mathcal{K}$ is a non-negative unbounded operator on $L^2(\mathbb{R}_{v})$ with a kernel given by
$$
\mathrm{Ker}\ \mathcal{K}=\mathrm{Span}\{e_0,e_2\}.
$$
Here the Hermite basis $(e_n)_{n\geq0}$ is an orthonormal basis of $L^2(\mathbb{R})$, which is presented in Section \ref{S5-1}.
The harmonic oscillator
$$
\mathcal{H}\triangleq-\Delta_{v}+\frac{v^2}{4}=\sum_{n\geq0}(n+\frac{1}{2})\mathbb{P}_n,
$$
where $\mathbb{P}_n$ stands for the orthogonal projection
$$
\mathbb{P}_nf=(f,e_n)_{L^2(\mathbb{R})}e_n.
$$
Furthermore, the fractional harmonic oscillator
$$
\mathcal{H}^s=\sum_{k\geq0}(k+\frac{1}{2})^{s}\mathbb{P}_{k}
$$
can be defined by the functional calculus. The linearized Kac operator is diagonal in the Hermite basis
\begin{equation}\label{A-2}
\mathcal{K}=\sum_{k\geq1}\lambda_k\mathbb{P}_{k}
\end{equation}
with a spectrum only composed by the non-negative eigenvalues
\begin{align*}
&\lambda_{2k+1}=\int_{-\frac{\pi}{4}}^{\frac{\pi}{4}}\beta(\theta)(1-(\cos\theta)^{2k+1})d\theta\geq0, \ \ k\geq0,
\\& \lambda_{2k}=\int_{-\frac{\pi}{4}}^{\frac{\pi}{4}}\beta(\theta)(1-(\cos\theta)^{2k}-(\sin\theta)^{2k})d\theta\geq0, \ \ k\geq1,
\end{align*}
satisfying the asymptotic estimates
\begin{equation}\label{A-3}
\lambda_{k}\approx k^{s}.
\end{equation}

Now we take the following choice for the cross section
\begin{equation}\label{A-5}
\beta(\theta)=\frac{|\cos\frac{\theta}{2}|}{|\sin\frac{\theta}{2}|^{1+2s}}, \ \ |\theta|\leq\frac{\pi}{4},
\end{equation}
in part because of the usage of those results in \cite{LMPX1} directly. In that case, the eigenvalues satisfy the asymptotic equivalence
\begin{equation}\label{A-55}
\lambda_k\underset{k\rightarrow+\infty}\sim\frac{2^{1+s}}{s}\mathbf{\Gamma}(1-s)k^s,
\end{equation}
where $\mathbf{\Gamma}$ denotes the Gamma function.

It is known that the solution of Boltzmann equation without angular cutoff can enjoy smoothing effects. The non-integrability of the cross section is essential for the smoothing effect, see for example \cite{D}. Alexandre-Morimoto-Ukai-Xu-Yang \cite{AMUXY3} highlighted the importance of regularization effects for Boltzmann equation (see also \cite{AMUXY,AM,DFT,DS}). They studied $C^\infty$ smoothing properties of the spatially inhomogeneous non-cutoff Boltzmann equation in  \cite{AMUXY,AMUXY1,AMUXY3}. In \cite{LMPX3}, Lerner-Morimoto-Starov-Xu studied the linearized Landau and Boltzmann equation and proved that the linearized non-cutoff Boltzmann operator with Maxwellian is exactly equal to a fractional power of the linearized Landau operator. In addition, Lekrine-Xu \cite{LX}
investigated the Gevrey regularizing effect of the Cauchy problem for non-cutoff homogeneous Kac equation. Later, Lerner-Morimoto-Starov-Xu
\cite{LMPX1} considered the linearized non-cutoff Kac collision operator around the Maxwellian distribution and found that it behaved like a fractional power of the harmonic oscillator and was diagonal in the hermite basis. Moreover, it was shown in \cite{LMPX4} by Lerner-Morimoto-Starov-Xu that the Cauchy problem to the homogeneous non-cutoff Kac equation
\begin{equation*}
\left\{
\begin{array}{ll}
   \partial_t g+\mathcal{K}g=\Gamma(g,g), \\
   g|_{t=0}=g_0\in L^2(\mathbb{R}_v),
\end{array}
\right.
\end{equation*}
enjoys the following Gefand-Shilov regularizing properties
$$
\forall t>0,\quad g(t)\in S^{\frac 1{2s}}_{\frac{1}{2s}}(\mathbb{R}),
$$
where the Gefand-Shilov space $S^{\mu}_{\nu}(\mathbb{R}^d)$ with $\mu, \nu>0, \mu+\nu\ge 1$, are defined as the set of smooth functions $f\in C^\infty(\mathbb{R}^d)$ satisfying
$$
\exists A, C>0, \ \forall \alpha ,\beta\in\mathbb{N}^d, \ \ \sup_{v\in\mathbb{R}^d}|v^\beta\partial_v^\alpha f(v)|\leq C A^{|\alpha|+|\beta|}(\alpha!)^{\mu}(\beta !)^{\nu}.
$$
The Gevrey class $G^{\mu}(\mathbb{R}^d)$ is the set of smooth function fulfilling
$$
\exists A, C>0, \ \forall \alpha\in\mathbb{N}^d, \ \ \sup_{v\in\mathbb{R}^d}|\partial_v^\alpha f(v)|\leq C A^{|\alpha|}(\alpha!)^{\mu}.
$$
The analysis of the Gevrey regularizing properties of spatially inhomogeneous kinetic equations with respect to both position and velocity variables is more complicated. Up to now, there are few results expect for a very simplified model of the linearized inhomogeneous non-cutoff Boltzmann equation, say the generalized Kolomogorov equation
\begin{equation}\label{K}
\left\{
\begin{array}{ll}
   \partial_t g+v\cdot\nabla_{x}g+(-\Delta_v)^{s}g=0,\\
   g|_{t=0}=g_0\in L^{2}(\mathbb{R}^{2d}_{x,v}),
\end{array}
\right.
\end{equation}
with $0<s<1$. Morimoto and Xu \cite{MX} found that the solution to \eqref{K} satisfied
$$
\exists c>0, \forall t>0, \ \ e^{c(t^{2s+1}(-\Delta_x)^s+t(-\Delta_v)^s)}g(t)\in L^{2}(\mathbb{R}^{2d}_{x,v}),
$$
which implies that the generalized Kolomogorov equation enjoys a $G^{\frac{1}{2s}}(\mathbb{R}^{2d}_{x,v})$ Gevrey smoothing effect with respect to both position and velocity variables. The phenomenon of hypoellipticity arises from non-commutation and non-trivial interactions between the transport part $v\cdot\nabla_{x}$ and the diffusion part $(-\Delta_v)^{s}$ in this evolution equation. On the other hand, for the Cauchy problem of the linear model of spatially inhomogeneous Landau equation,
\begin{equation}\label{Landau}
\left\{
\begin{array}{ll}
   \partial_t g+v\cdot\nabla_{x}g=\nabla_{v}\left(\bar{a}(\mu)\cdot\nabla_vg-\bar{b}(\mu)g\right),\\
   g|_{t=0}=g_0\,
\end{array}
\right.
\end{equation}
with
\begin{align*}
&\bar{a}_{ij}(\mu)=\delta_{ij}(|v|^2+1)-v_iv_j, \\
&\bar{b}_j(\mu)=-v_j,\ \  \ i,j=1,\cdots,d,
\end{align*}
they showed in \cite{MX} that the solution to \eqref{Landau} enjoyed a $G^1(\mathbb{R}^{2d}_{x,v})$ Gevrey smoothing effect with respect to both position and velocity variables with the estimate
$$
\exists c>0, \forall t>0, \ \ e^{c(t^2(-\Delta_x)^{\frac{1}{2}}+t(-\Delta_v)^{\frac{1}{2}})}g(t)\in L^{2}(\mathbb{R}^{2d}_{x,v}),
$$
which coincides with the fact that the Landau equation can be regarded as the limit $s=1$ of the Boltzmann equation.

Recently, Lerner-Morimoto-Starov-Xu \cite{LMPX2} considered the spatially inhomogeneous non-cutoff Kac equation in the Sobolev space and showed that the Cauchy problem for the fluctuation around the Maxwellian distribution admitted $S^{1+\frac{1}{2s}}_{1+\frac{1}{2s}}$ Gelfand-Shilov regularity properties with respect to the velocity variable and $G^{1+\frac{1}{2s}}$ Gevrey regularizing properties with respect to the position variable. In \cite{LMPX2}, the authors conjectured that it remained still open to determine whether the regularity indices $1+\frac{1}{2s}$ is sharp or not. On the other hand, Duan-Liu-Xu \cite{DLX-2016} and Morimoto-Sakamoto \cite{MS} studied the Cauchy problem for the Boltzmann equation with the initial datum belonging to critical Besov space. Motivated by those works, we intend to study the inhomogeneous non-cutoff Kac equation in critical Besov space and then improve the Gelfand-Shilov regularizing properties and Gevrey regularizing properties.

Now, our main results are stated as follows (see Section \ref{S2} for the definition of Besov spaces ).

\begin{theorem}[{\bf Main Theorem}]\label{Main}
Let $0<T<+\infty$. We suppose that the collision cross section satisfies \eqref{A-5} with $0<s<1$. There exists a constant $\varepsilon_{0}>0$ such that for all $g_0\in\widetilde{L}^{2}_{v}(B^{1/2}_{2,1})$ satisfying
$$
\|g_0\|_{\widetilde{L}_v^2(B^{1/2}_{2,1})}\leq\,\varepsilon_{0},
$$
then the Cauchy problem \eqref{eq-1} admits a weak solution $g\in \widetilde{L}^{\infty}_T\widetilde{L}^2_{v}(B^{1/2}_{2,1})$ satisfying
\begin{equation*}
\begin{split}	&\|g\|_{\widetilde{L}^{\infty}_T\widetilde{L}^2_{v}(B^{1/2}_{2,1})}+\|\mathcal{H}^{\frac{s}{2}}g\|_{\widetilde{L}^{2}_T\widetilde{L}^2_{v}(B^{1/2}_{2,1})}
\leq c_{0}e^{T}\|g_{0}\|_{\widetilde{L}^2_{v}(B^{1/2}_{2,1})},
\end{split}
\end{equation*}	
for some constant $c_{0}>1$. Furthermore, this solution is smooth for all positive time $0<t\leq T$, which satisfies the following Gelfand-Shilov and Gevrey type estimates: For $\delta>0$, there exists $C>1$ such that for all $ 0<t\leq T$ and for all $k,l,q\geq0$,
\begin{align*}
\|v^k\partial_v^l\partial_x^qg(t)\|_{\widetilde{L}^2_{v}(B^{1/2}_{2,1})}
\leq\frac{C^{k+l+q+1}}{t^{\frac{3s+1}{2s(s+1)}(k+l+2)+\frac{2s+1}{2s}q+\delta}}
(k!)^{\frac{3s+1}{2s(s+1)}}(l!)^{\frac{3s+1}{2s(s+1)}}(q!)^{\frac{2s+1}{2s}}\|g_{0}\|_{\widetilde{L}^2_{v}(B^{1/2}_{2,1})}.
\end{align*}
\end{theorem}

Our result deserves some comments in contrast to the result of \cite{LMPX2}.
\begin{remark}
\begin{itemize}
  \item[(1)] We show the well-posedness of Cauchy problem with the initial datum belonging to the spatially critical Besov space $\widetilde{L}_v^2(B^{1/2}_{2,1})$, rather than in the Sobolev space $L^2_v(H^1_x)$.
  \item[(2)] For the regularizing effect, our result indicates that
  $$
  \forall t>0, \ \forall x\in\mathbb{R}, \ \ g(t,x,\cdot)\in S^{\frac{3s+1}{2s(s+1)}}_{\frac{3s+1}{2s(s+1)}}(\mathbb{R}); \quad
  \forall t>0, \ \forall v\in\mathbb{R}, \ \ g(t,\cdot, v)\in G^{1+\frac{1}{2s}}(\mathbb{R}).
  $$
  Actually, the Gelfand-Shilov index for the velocity variable is sharp for $0<s<1$, if noticing that
  $$
  \quad\frac{3s+1}{2s(s+1)}=\frac{2s+1}{2s}\frac{3s+1}{(2s+1)(s+1)}<1+\frac{1}{2s}.
  $$
  \item[(3)] If $s$ is close to $1$, the solution is almost analytic in the velocity variable, since
  $$
  \quad\frac{3s+1}{2s(s+1)}\rightarrow1.
  $$
 Therefore, our Gelfand-Shilov index for the velocity variable should be optimal. 
\end{itemize}	
\end{remark}

The paper is arranged as follows. In Section \ref{S2}, we recall the definitions of Besov spaces and Chemin-Lerner spaces as well as some
key estimates for the Kac collision operator. Section \ref{S3} is devoted to establish the local existence for \eqref{eq-1} in critical Besov space. In Section \ref{S4}, we establish Gelfand-Shilov regularizing properties with respect to the velocity variable and Gevrey smoothing effects with respect to position variable. Section \ref{S5} is Appendix, where instrumental estimates in terms of Hermite functions, the definition of the Kac collision operator as a finite part integral, some estimates of Kac collision operator, the equivalent definitions of Gelfand-Shilov regularity and analysis properties in Besov spaces are presented.

\section{Analysis of Kac collision operator}\label{S2}

In this section, we present the trilinear estimates of the Kac collision operator which will be used in the subsequent analysis. Firstly, we recall the Littlewood-Paley decomposition. The reader is also referred to \cite{BCD} for more details. Let $(\varphi,\chi)$ be a couple of smooth functions valued in the closed interval $[0,1]$ such that $\varphi$ is supported in the shell $\mathbb{C}(0,\frac{3}{4},\frac{8}{3})=\left\{\xi\in\mathbb{R}:\frac{3}{4}\leq\left|\xi\right|\leq\frac{8}{3}\right\}$ and $\chi$ is supported in the ball $\mathbb{B}(0,\frac{4}{3})=\left\{\xi\in\mathbb{R}:\left|\xi\right|\leq\frac{4}{3}\right\}$. In terms of the two functions, one has the unit decomposition
$$
\chi(\xi)+\sum_{q\geq0}\varphi\left(2^{-q}\xi\right)=1, \ \ \forall\xi\in\mathbb{R}.
$$
The inhomogeneous dyadic blocks are defined by
$$
\Delta_{-1}u\triangleq\chi(D)u, \quad\quad \Delta_{q}u\triangleq\varphi(2^{-q}D)u, \ \ \ q\geq0
$$
for $u=u(x)\in\mathcal{S}'(\mathbb{R}_{x})$. Hence, the Littlewood-Paley decomposition for any tempered distribution $u$ reads
$$
u=\sum_{q\geq-1}\Delta_qu.
$$
It is also convenient to introduce the low-frequency cut-off:
$$
S_{q}u=\sum_{p\leq q-1}\Delta_{p}u.
$$

Now, we give the definition of main functional spaces in the present paper.

\begin{definition}
Let $\sigma\in\mathbb{R}$ and $1\leq p,r\leq\infty$. The nonhomogeneous Besov space $B^{\sigma}_{p,r}$ is defined by
\begin{align*}
B^{\sigma}_{p,r}:=\Big\{u\in\mathcal{S}'(\mathbb{R}_{x}):u=\sum_{q\geq-1}\Delta_qu \ \text{in} \ \mathcal{S}' \Big| \
\|u\|_{B^{\sigma}_{p,r}}<\infty\Big\}
\end{align*}
where
\begin{align*}
\|u\|_{B^{\sigma}_{p,r}}=\left(\sum_{q\ge-1}\left(2^{q\sigma}\|\Delta_qf\|_{L^{p}_{x}}\right)^{r}\right)^{1/r}
\end{align*}
with the usual convention for $p, r=\infty$.
\end{definition}

For the distribution $u=u(t,x,v)$, we define the mixed Banach space

\begin{align*}
L^{p_{1}}_{T}L^{p_{2}}_{v}L^{p_{3}}_{x}\triangleq L^{p_{1}}([0,T]; L^{p_{2}}(\mathbb{R}_{v};L^{p_{3}}(\mathbb{R}_{x})))
\end{align*}
for $0<T\leq\infty$ and $ 1\leq p_{1},p_{2},p_{3}\leq\infty$, whose norm is given by
\begin{align*}
\|u\|_{L^{p_{1}}_{T}L^{p_{2}}_{v}L^{p_{3}}_{x}}=
\left(\int_{0}^{T}\left(\int_{\mathbb{R}}\left(\int_{\mathbb{R}}\left|u(t,x,v)\right|^{p_{3}}dx\right)^{p_{2}/p_{3}}dv\right)^{p_{1}/p_{2}}dt\right)^{1/p_{1}}
\end{align*}
with the usual convention if $p_{1},p_{2},p_{3}=\infty$.

In addition, we give another mixed Banach space, which was initialed by Chemin and Lerner in \cite{CL}.

\begin{definition}
Let $\sigma\in\mathbb{R}$ and $1\leq\varrho_{1}, \varrho_{2},p, r\leq\infty$. For $0<T\leq\infty$, the space $\widetilde{L}^{\varrho_{1}}_{T}\widetilde{L}^{\varrho_{2}}_{v}(B^{\sigma}_{p,r})$ is defined by
\begin{align*}
\widetilde{L}^{\varrho_{1}}_{T}\widetilde{L}^{\varrho_{2}}_{v}(B^{\sigma}_{p,r})=\left\{u(t,\cdot,v)\in\mathcal{S}': \|u\|_{\widetilde{L}^{\varrho_{1}}_{T}\widetilde{L}^{\varrho_{2}}_{v}(B^{\sigma}_{p,r})}<\infty\right\},
\end{align*}
where
\begin{align*}
\|u\|_{\widetilde{L}^{\varrho_{1}}_{T}\widetilde{L}^{\varrho_{2}}_{v}(B^{\sigma}_{p,r})}=\left(\sum_{q\geq-1}\left(2^{q\sigma}\|\Delta_{q}u\|
_{L^{\varrho_{1}}_{T}L^{\varrho_{2}}_{v}L^{p}_{x}}\right)^{r}\right)^{1/r}
\end{align*}
with the usual convention for $\varrho_{1},\varrho_{2},p,r=\infty$.
\end{definition}

Due to the coercivity of the linearized Kac collision operator $\mathcal{K}$, we have the following result.

\begin{lemma}\label{JJ}
For the linear term $\mathcal{K}$, there exists a constant $C>0$ such that for the suitable function $f,g$
\begin{align*}
\frac{1}{C}\|\Delta_p\mathcal{H}^{\frac{s}{2}}f\|^2_{L^2(\mathbb{R}_{v})}\leq(\Delta_p\mathcal{K}f,\Delta_pf)_{L^{2}(\mathbb{R}_{v})}+\|\Delta_pf\|^2_{L^2(\mathbb{R}_{v})}
\leq C\|\Delta_p\mathcal{H}^{\frac{s}{2}}f\|^2_{L^2(\mathbb{R}_{v})}
\end{align*}
for each $p\geq-1$. Moreover, for $\sigma>0$ and $ T>0$, it holds that
\begin{align*}
\sum_{p\geq-1}2^{p\sigma}\left(\int^T_0\left(\Delta_p\mathcal{K}g,\Delta_pg\right)_{L^{2}(\mathbb{R}^2_{x,v})}dt\right)^{1/2}
\geq\frac{1}{C}\|\mathcal{H}^{\frac{s}{2}}g\|_{\widetilde{L}_T^{2}\widetilde{L}^2_{v}(B^\sigma_{2,1})}
-\|g\|_{\widetilde{L}_T^{2}\widetilde{L}^2_{v}(B^\sigma_{2,1})}.
\end{align*}
\end{lemma}

\begin{proof}
Observe that the inner product $(\cdot,\cdot)_{L^{2}(\mathbb{R}_{v})}$ is with respect to $v$, $\Delta_p$ acts on $x$ and the linearized non-cutoff Kac operator $\mathcal{K}$ is independent of $x$. Thus, the first inequality can be obtained by using the spectral estimate for $\mathcal{K}$ in Section \ref{S5-2}.  The second inequality just follows from the first inequality and the definition of Chemin-Lerner spaces.
\end{proof}

For the nonlinear term $\Gamma(f,g)$,  the authors \cite{LMPX2} showed some trilinear estimates in the Sobolev space. Here, we establish the trilinear estimates with minor changes, which will be used in the proof of local existence of  \eqref{eq-1}.

\begin{lemma}\label{trilinear-estimate-S}
Let $f,g,h\in\mathcal{S}(\mathbb{R}^2_{x,v})$. Then there exists a constant $C_0>0$ such that for all $0\leq\delta\leq1, j_1,j_2\geq0$ with $j_1+j_2\leq1$, it holds that
\begin{equation}\label{trilinear-A}
\begin{split}
&\left|\left((1+\delta\sqrt{\mathcal{H}}+\delta\langle D_{x}\rangle)^{-1}\Gamma((1+\delta\sqrt{\mathcal{H}})^{j_{1}}f,(1+\delta\sqrt{\mathcal{H}})^{j_{2}}g),h\right)_{L^2_{x,v}}\right|
\\&\qquad
\leq
C_{0}\left\|f\right\|_{L^2_{v}L^{\infty}_x}
\|\mathcal{H}^{\frac{s}{2}}g\|_{L^2_{x,v}}\|\mathcal{H}^{\frac{s}{2}}h\|_{L^2_{x,v}},
\\&
\Big|\left((1+\delta\sqrt{\mathcal{H}}+\delta\langle D_{x}\rangle)^{-1}\Gamma(f,\delta\langle D_{x}\rangle(1+\delta\sqrt{\mathcal{H}}+\delta\langle D_{x}\rangle)^{-1}g),h\right)_{L^2_{x,v}}\Big|
\\&\qquad
\leq C_{0}\|f\|_{L^{2}_{v}L^{\infty}_{x}}\|\mathcal{H}^{\frac{s}{2}}g\|_{L^{2}_{x,v}}
\|\mathcal{H}^{\frac{s}{2}}h\|_{L^{2}_{x,v}}.
\end{split}
\end{equation}
\end{lemma}

\begin{proof}
For $f, g, h\in\mathcal{S}(\mathbb{R}^2_{x,v})$, we decompose these functions into the Hermite basis in the velocity variable
$$
f=\sum_{n=0}^{+\infty}f_{n}(x)e_{n}(v),\quad\,f_{n}=\langle f(x,\cdot),e_{n}\rangle_{L^2(\mathbb{R}_{v})}
$$
and similar decomposition for $g,h$. Notice that
\begin{equation}\label{non-1}
\begin{split}
\|f\|_{L^2_vL^p_x}=\Big(\sum_{n=0}^{+\infty}\|f_{n}\|^{2}_{L^p(\mathbb{R}_{x})}\Big)^{\frac{1}{2}}, \ \
\|\mathcal{H}^{m}f\|_{L^2_vL^p_x}=\Big(\sum_{n=0}^{+\infty}\Big(n+\frac{1}{2}\Big)^{2m}\|f_{n}\|^{2}_{L^p(\mathbb{R}_{x})}\Big)^{\frac{1}{2}}
\end{split}
\end{equation}
with $1\leq p\leq\infty$ and $m\in\mathbb{R}$. Following from Lemma \ref{J10-1} and Cauchy-Schwarz inequality, we obtain
\begin{align*}
&\left|\left((1+\delta\sqrt{\mathcal{H}}+\delta\langle D_{x}\rangle)^{-1}\Gamma((1+\delta\sqrt{\mathcal{H}})^{j_{1}}f,\Big(1+\delta\sqrt{\mathcal{H}}\Big)^{j_{2}}g),h\right)_{L^2_{x,v}}\right|
\\&
=\Big|\sum_{n=0}^{+\infty}\sum_{\substack{k+l=n\\k,l\geq0}}\alpha_{k,l}\Big(\Big(1+\delta\sqrt{n+\frac{1}{2}}+\delta\langle D_{x}\rangle\Big)^{-1}\Big(1+\delta\sqrt{k+\frac{1}{2}}\Big)^{j_{1}}
\\&\qquad
\times\Big(1+\delta\sqrt{l+\frac{1}{2}}\Big)^{j_{2}}f_{k}g_{l},h_{n}\Big)_{L^2_x}\Big|
\\&
\leq\sum_{n=0}^{+\infty}|\alpha_{0,n}|\|f_{0}\|_{L^\infty_x}\|g_{n}\|_{L^2_x}\|h_{n}\|_{L^2_x}
+\sum_{n=0}^{+\infty}\sum_{\substack{2k+l=n\\k\geq1,l\geq0}}|\alpha_{2k,l}|\|f_{2k}\|_{L^\infty_x}\|g_{l}\|_{L^2_x}\|h_{n}\|_{L^2_x},
\end{align*}
where we used that
$$
\Big\|\Big(1+\delta\sqrt{n+\frac{1}{2}}+\delta\langle D_{x}\rangle\Big)^{-1}\Big(1+\delta\sqrt{k+\frac{1}{2}}\Big)^{j_{1}}\Big(1+\delta\sqrt{l+\frac{1}{2}}\Big)^{j_{2}}\Big\|_{\mathcal{L}(L^2(\mathbb{R}_x))}\leq1,
$$
since $k+l=n$ and $j_1,j_2\geq0$ with $j_1+j_2\leq1$. Under the assumption \eqref{A-5}, it follows from the formula \eqref{A-55} and Lemma \ref{J10-2} that for $f,g,h\in\mathcal{S}(\mathbb{R}^2_{x,v})$,
\begin{align*}
&\left|\left((1+\delta\sqrt{\mathcal{H}}+\delta\langle D_{x}\rangle)^{-1}\Gamma((1+\delta\sqrt{\mathcal{H}})^{j_{1}}f,\Big(1+\delta\sqrt{\mathcal{H}}\Big)^{j_{2}}g),h\right)_{L^2_{x,v}}\right|
\\&
\lesssim\|f_{0}\|_{L^\infty_x}\sum_{n=0}^{+\infty}(n+\frac{1}{2})^{s}\|g_{n}\|_{L^2_x}\|h_{n}\|_{L^2_x}
+\sum_{n=0}^{+\infty}\|h_{n}\|_{L^2_x}\Big(\sum_{\substack{2k+l=n\\k\geq1,l\geq0}}\frac{\tilde{\mu}_{k,l}}{k^{\frac{3}{4}}}
\|f_{2k}\|_{L^\infty_x}\|g_{l}\|_{L^2_x}\Big)
\\&\quad
:=I_{1}+I_2.
\end{align*}
By using \eqref{non-1}, we have
\begin{align*}
I_{1}
&\leq\|f_{0}\|_{L^\infty_x}\Big(\sum_{n=0}^{+\infty}(n+\frac{1}{2})^{s}\|g_{n}\|^{2}_{L^2_x}\Big)^{\frac{1}{2}}
\Big(\sum_{n=0}^{+\infty}(n+\frac{1}{2})^{s}\|h_{n}\|^{2}_{L^2_x}\Big)^{\frac{1}{2}}
\\&
\leq\|f_{0}\|_{L^\infty_x}\|\mathcal{H}^{\frac{s}{2}}g\|_{L^2_{x,v}}\|\mathcal{H}^{\frac{s}{2}}h\|_{L^2_{x,v}}.
\end{align*}
On the other hand, we obtain
\begin{align*}
I_{2}&=\sum_{k\geq1,l\geq0}\frac{\tilde{\mu}_{k,l}}{k^{\frac{3}{4}}}\|f_{2k}\|_{L^\infty_x}\|g_{l}\|_{L^2_x}\|h_{2k+l}\|_{L^2_x}
\\&
=\sum_{l=0}^{+\infty}(l+\frac{1}{2})^{\frac{s}{2}}\|g_{l}\|_{L^2_x}\Big(\sum_{k=1}^{+\infty}
\frac{\tilde{\mu}_{k,l}}{k^{\frac{3}{4}}(l+\frac{1}{2})^{\frac{s}{2}}}\|f_{2k}\|_{L^\infty_x}\|h_{2k+l}\|_{L^2_x}\Big)
\\&
\leq\left\|f\right\|_{L^2_{v}(L^{\infty}_x)}\|\mathcal{H}^{\frac{s}{2}}g\|_{L^2_{x,v}}
\Big(\sum_{l=0}^{+\infty}\sum_{k=1}^{+\infty}\frac{\tilde{\mu}^{2}_{k,l}}{k^{\frac{3}{2}}(l+\frac{1}{2})^{s}}\|h_{2k+l}\|^{2}_{L^2_x}\Big)^{\frac{1}{2}}.
\end{align*}
Here, we may calculate that
\begin{align*}
\Big(\sum_{l=0}^{+\infty}\sum_{k=1}^{+\infty}\frac{\tilde{\mu}^{2}_{k,l}}{k^{\frac{3}{2}}(l+\frac{1}{2})^{s}}\|h_{2k+l}\|^{2}_{L^2_x}\Big)^{\frac{1}{2}}
=\Big[\sum_{n=0}^{+\infty}\|h_{n}\|^{2}_{L^2_x}\Big(\sum_{\substack{2k+l=n\\k\geq1,l\geq0}}
\frac{\tilde{\mu}^{2}_{k,l}}{k^{\frac{3}{2}}(l+\frac{1}{2})^{s}}\Big)\Big]^{\frac{1}{2}}.
\end{align*}
Since
$$
\tilde{\mu}_{k,l}\lesssim k^{\frac{1}{4}} \ \text{when} \ k\geq l, k\geq1, l\geq0; \ \ \
\tilde{\mu}_{k,l}\lesssim(l+\frac{1}{2})^{s} \ \text{when} \ 1\leq k\leq l,
$$
it follows from Lemma \ref{J10-2} that
\begin{equation*}
\sum_{\substack{2k+l=n\\k\geq1,l\geq0}}
\frac{\tilde{\mu}^{2}_{k,l}}{k^{\frac{3}{2}}(l+\frac{1}{2})^{s}}
\lesssim\sum_{\substack{2k+l=n\\k\geq1,l\geq0\\k\geq l}}\frac{k^{\frac{1}{2}}}{k^{\frac{3}{2}}(l+\frac{1}{2})^{s}}
+\sum_{\substack{2k+l=n\\k\leq1,l\geq0\\k\leq l}}\frac{(l+\frac{1}{2})^{s}}{k^{\frac{3}{2}}}
\lesssim(n+\frac{1}{2})^{s}.
\end{equation*}
Thus, we are led to
\begin{equation}\label{HHH}
\Big(\sum_{l=0}^{+\infty}\sum_{k=1}^{+\infty}\frac{\tilde{\mu}^{2}_{k,l}}{k^{\frac{3}{2}}(l+\frac{1}{2})^{s}}\|h_{2k+l}\|^{2}_{L^2_x}\Big)^{\frac{1}{2}}
\lesssim\|\mathcal{H}^{\frac{s}{2}}h\|_{L^2_{x,v}}.
\end{equation}
We can conclude that there a positive constant $C_0>0$ such that
\begin{align*}
&\left|\left((1+\delta\sqrt{\mathcal{H}}+\delta\langle D_{x}\rangle)^{-1}\Gamma((1+\delta\sqrt{\mathcal{H}})^{j_{1}}f,(1+\delta\sqrt{\mathcal{H}})^{j_{2}}g),h\right)_{L^2_{x,v}}\right|
\\&\qquad
\leq
C_{0}\left\|f\right\|_{L^2_{v}L^{\infty}_x}
\|\mathcal{H}^{\frac{s}{2}}g\|_{L^2_{x,v}}\|\mathcal{H}^{\frac{s}{2}}h\|_{L^2_{x,v}},
\end{align*}
which leads to the first inequality in \eqref{trilinear-A}. Similarly, we have
\begin{align*}
&\Big|\left((1+\delta\sqrt{\mathcal{H}}+\delta\langle D_{x}\rangle)^{-1}\Gamma(f,\delta\langle D_{x}\rangle(1+\delta\sqrt{\mathcal{H}}+\delta\langle D_{x}\rangle)^{-1}g),h\right)_{x,v}\Big|
\\&
=\Big|\sum_{n=0}^{+\infty}\sum_{\substack{k+l=n\\k,l\geq0}}\alpha_{k,l}\int_{\mathbb{R}}\frac{1}{(1+\delta\sqrt{n+\frac{1}{2}}+\delta\langle D_{x}\rangle)}(f_{k}\frac{\delta\langle D_{x}\rangle}{(1+\delta\sqrt{l+\frac{1}{2}}+\delta\langle D_{x}\rangle)}g_{l})\overline{h_{n}(x)}dx\Big|
\\&
\leq\sum_{n=0}^{+\infty}|\alpha_{0,n}|\|f_{0}\|_{L^\infty_x}\|g_{n}\|_{L^2_x}\|h_{\delta,n}\|_{L^2_x}
+\sum_{n=0}^{+\infty}\sum_{\substack{2k+l=n\\k\geq1,l\geq0}}|\alpha_{2k,l}|\|f_{2k}\|_{L^\infty_x}\|g_{l}\|_{L^2_x}\|h_{n}\|_{L^2_x},
\end{align*}
where we used
\begin{align*}
&\Big\|\delta\langle D_{x}\rangle\Big(1+\delta\sqrt{n+\frac{1}{2}}+\delta\langle D_{x}\rangle\Big)^{-1}\Big\|_{\mathcal{L}(L^2(\mathbb{R}_x))}\leq1,
\\&
\Big\|\Big(1+\delta\sqrt{n+\frac{1}{2}}+\delta\langle D_{x}\rangle\Big)^{-1}\Big\|_{\mathcal{L}(L^2(\mathbb{R}_x))}\leq1.
\end{align*}
Hence, by proceeding the similar procedure, we can obtain the second inequality in \eqref{trilinear-A}.
\end{proof}

Putting $\delta=0$ in Lemma \ref{trilinear-estimate-S}, which coincides with Lemma 3.5 in \cite{LMPX4}.

\begin{remark}
Let $f,g,h\in\mathcal{S}(\mathbb{R}^2_{x,v})$, then there exists a constant $C_0>0$ such that
\begin{equation*}
\begin{split}
&\left|(\Gamma(f,g),h)_{L^2_{x,v}}\right|
\leq C_{0}\left\|f\right\|_{L^2_{v}L^{\infty}_x}
\|\mathcal{H}^{\frac{s}{2}}g\|_{L^2_{x,v}}\|\mathcal{H}^{\frac{s}{2}}h\|_{L^2_{x,v}}.
\end{split}
\end{equation*}
\end{remark}

By the similar proof as in \cite{LMPX2}, we also have

\begin{lemma}\label{trilinear-estimate-SSS}
Let $f,g,h\in\mathcal{S}(\mathbb{R}^2_{x,v})$. Then there exists a constant $C_0>0$ such that
\begin{equation*}
\begin{split}
\left\|\mathcal{H}^{-s}\Gamma(f,g)\right\|_{L^2_{x,v}}
\leq
C_{0}\left\|f\right\|_{L^2_{v}L^{\infty}_x}\|g\|_{L^2_{x,v}}.
\end{split}
\end{equation*}
\end{lemma}

We prove the following result in order to estimate the nonlinear collision operator in the framework of Besov spaces.
\begin{lemma}\label{trilinear-estimate-M-G}
There exists a constant $C_1>0$ such that for all $f,g\in\mathcal{S}(\mathbb{R}_x), t\geq0, 0<\kappa\leq1, m,n\geq0$,
\begin{equation}\label{trilinear-est-M-G}
\begin{split}
\|\mathcal{G}_{\kappa,m+n}(t)([(\mathcal{G}_{\kappa,m}(t))^{-1}f][(\mathcal{G}_{\kappa,n}(t))^{-1}g])\|_{L^2_x}
\leq C_1\|f\|_{B^{1/2}_{2,1}}\|g\|_{L^2_x},
\end{split}
\end{equation}
with the Fourier multiplier
\begin{equation}\label{G-N}
\mathcal{G}_{\kappa,n}(t)=\frac{\exp\Big(t\Big((n+\frac{1}{2})^{\frac{s+1}{2}}+\langle D_x\rangle^{\frac{3s+1}{2s+1}}\Big)^{\frac{2s}{3s+1}}\Big)}{1+\kappa\exp\Big(t\Big((n+\frac{1}{2})^{\frac{s+1}{2}}+\langle D_x\rangle^{\frac{3s+1}{2s+1}}\Big)^{\frac{2s}{3s+1}}\Big)}.
\end{equation}
\end{lemma}

\begin{proof}
Notice that the operator $\mathcal{G}_{\kappa,n}(t)$ is a bounded isomorphism of $L^2(\mathbb{R}_x)$ such that
$$
(\mathcal{G}_{\kappa,n}(t))^{-1}=\frac{1+\kappa\exp\Big(t\Big((n+\frac{1}{2})^{\frac{s+1}{2}}+\langle D_x\rangle^{\frac{3s+1}{2s+1}}\Big)^{\frac{2s}{3s+1}}\Big)}{\exp\Big(t(n+\frac{1}{2})^{\frac{s+1}{2}}+\langle D_x\rangle^{\frac{3s+1}{2s+1}}\Big)^{\frac{2s}{3s+1}}\Big)}.
$$
Set
$$
h=\mathcal{G}_{\kappa,m+n}(t)([(\mathcal{G}_{\kappa,m}(t))^{-1}f][(\mathcal{G}_{\kappa,n}(t))^{-1}g]),
$$
then we have
\begin{equation}\label{trilinear-est-MM}
\begin{split}
&\hat{h}(\xi)=\frac{\exp\Big(t\Big(\left(m+n+\frac{1}{2}\right)^{\frac{s+1}{2}}+\langle\xi\rangle^{\frac{3s+1}{2s+1}}\Big)^{\frac{2s}{3s+1}}\Big)}
{1+\kappa\exp\Big(t\Big(\left(m+n+\frac{1}{2}\right)^{\frac{s+1}{2}}+\langle\xi\rangle^{\frac{3s+1}{2s+1}}\Big)^{\frac{2s}{3s+1}}\Big)}
\mathcal{F}(([(\mathcal{G}_{\kappa,m}(t))^{-1}f][(\mathcal{G}_{\kappa,n}(t))^{-1}g]))
\\&=\frac{1}{2\pi}\frac{\exp\Big(t\Big(\left(m+n+\frac{1}{2}\right)^{\frac{s+1}{2}}+\langle\xi\rangle^{\frac{3s+1}{2s+1}}\Big)^{\frac{2s}{3s+1}}\Big)}
{1+\kappa\exp\Big(t\Big(\left(m+n+\frac{1}{2}\right)^{\frac{s+1}{2}}+\langle\xi\rangle^{\frac{3s+1}{2s+1}}\Big)^{\frac{2s}{3s+1}}\Big)}
\mathcal{F}((\mathcal{G}_{\kappa,m}(t))^{-1}f)\ast\mathcal{F}((\mathcal{G}_{\kappa,n}(t))^{-1}g),
\end{split}
\end{equation}
where $\mathcal{F}$ denotes the Fourier transform. Consider the increasing function
$$
Z(x)=\frac{e^x}{1+\kappa e^x},
$$
we can calculate and obtain
\begin{align*}
\forall x, y\geq0, \quad  \frac{Z(x+y)}{Z(x)Z(y)}=\kappa+\frac{1-\kappa}{1+\kappa e^{x+y}}+\frac{\kappa(e^x+e^y)}{1+\kappa e^{x+y}}
\leq1+\frac{1}{e^x}+\frac{1}{e^y}\leq3,
\end{align*}
which implies that the function $Z(x)$ satisfies the inequality
\begin{equation}\label{Z}
\begin{split}
\forall x, y\geq0, \quad Z(x+y)\leq3Z(x)Z(y).
\end{split}
\end{equation}
Since for all $m,n\geq0, \xi,\eta\in\mathbb{R}$,
\begin{equation*}
\begin{split}
&\Big(\Big(m+n+\frac{1}{2}\Big)^{\frac{s+1}{2}}+\langle\xi\rangle^{\frac{3s+1}{2s+1}}\Big)^{\frac{2s}{3s+1}}
\\&
\leq\Big(\Big(m+\frac{1}{2}\Big)^{\frac{s+1}{2}}+\langle\eta\rangle^{\frac{3s+1}{2s+1}}\Big)^{\frac{2s}{3s+1}}+
\Big(\Big(n+\frac{1}{2}\Big)^{\frac{s+1}{2}}+\langle\xi-\eta\rangle^{\frac{3s+1}{2s+1}}\Big)^{\frac{2s}{3s+1}},
\end{split}
\end{equation*}
by using \eqref{Z}, we obtain
\begin{equation}\label{FF}
\begin{split}
&\frac{\exp\Big(t\Big(\left(m+n+\frac{1}{2}\right)^{\frac{s+1}{2}}+\langle\xi\rangle^{\frac{3s+1}{2s+1}}\Big)^{\frac{2s}{3s+1}}\Big)}
{1+\kappa\exp\Big(t\Big(\left(m+n+\frac{1}{2}\right)^{\frac{s+1}{2}}+\langle\xi\rangle^{\frac{3s+1}{2s+1}}\Big)^{\frac{2s}{3s+1}}\Big)}
\\&
\leq\frac{3\exp\Big(t\Big(\left(m+\frac{1}{2}\right)^{\frac{s+1}{2}}+\langle\eta\rangle^{\frac{3s+1}{2s+1}}\Big)^{\frac{2s}{3s+1}}\Big)}
{1+\kappa\exp\Big(t\Big(\left(m+\frac{1}{2}\right)^{\frac{s+1}{2}}+\langle\eta\rangle^{\frac{3s+1}{2s+1}}\Big)^{\frac{2s}{3s+1}}\Big)}
\\&\qquad
\times
\frac{\exp\Big(t\Big(\left(n+\frac{1}{2}\right)^{\frac{s+1}{2}}+\langle\xi-\eta\rangle^{\frac{3s+1}{2s+1}}\Big)^{\frac{2s}{3s+1}}\Big)}
{1+\kappa\exp\Big(t\Big(\left(n+\frac{1}{2}\right)^{\frac{s+1}{2}}+\langle\xi-\eta\rangle^{\frac{3s+1}{2s+1}}\Big)^{\frac{2s}{3s+1}}\Big)}.
\end{split}
\end{equation}
Then it follows from \eqref{trilinear-est-MM} and \eqref{FF} that
\begin{equation*}
\begin{split}
\|h\|_{L^2_x}&\leq\frac{3}{(2\pi)^{\frac{3}{2}}}\||\hat{f}|\ast|\hat{g}|\|_{L^2_x}
=\frac{3}{(2\pi)^{\frac{3}{2}}}\|\mathcal{F}\mathcal{F}^{-1}(|\hat{f}|\ast|\hat{g})|\|_{L^2_x}
\\&
=\frac{3}{2\pi}\|\mathcal{F}^{-1}(|\hat{f}|\ast|\hat{g})|\|_{L^2_x}
=3\|\mathcal{F}^{-1}(|\hat{f}|)\mathcal{F}^{-1}(|\hat{g}|)\|_{L^2_x}
\\&
\leq3\|\mathcal{F}^{-1}(|\hat{f}|)\|_{L^\infty_x}\|\mathcal{F}^{-1}(|\hat{g}|)\|_{L^2_x}
\leq C_1\|f\|_{B^{1/2}_{2,1}}\|g\|_{L^2_x},
\end{split}
\end{equation*}
which leads to the desired \eqref{trilinear-est-M-G}. Here we used $\|\mathcal{F}^{-1}(|\hat{u}|)\|_{L^\infty_x}\leq C_d\|u\|_{B^{d/2}_{2,1}}$ on $\mathbb{R}^{d} (d\geq1)$ for $u\in\mathcal{S}(\mathbb{R}^{d})$. Indeed,
\begin{align*}
\|\mathcal{F}^{-1}(|\hat{u}|)\|_{L^\infty_x}&\leq\|\hat{u}\|_{L^1_x}
\leq\sum_{p\geq-1}\|\widehat{\Delta_pu}\|_{L^1_x}
\\&
=\int_{|\xi|\leq\frac{4}{3}}|\widehat{\Delta_{-1}u}|d\xi+\sum_{p\geq0}\int_{\frac{3}{4}2^p\leq|\xi|\leq\frac{8}{3}2^p}|\widehat{\Delta_pu}|d\xi
\\&
\leq\sum_{p\geq-1}C_d2^{\frac{d}{2}p}\|\Delta_pu\|_{L^2_x}
= C_d\|u\|_{B^{d/2}_{2,1}},
\end{align*}
where $C_d>0$ is a positive constant depending on the dimension $d$. Hence, the proof of Lemma \ref{trilinear-estimate-M-G} is finished.
\end{proof}

Now, we establish the key trilinear estimates for $\Gamma(f,g)$.

\begin{lemma}\label{trilinear-estimate-G}
Let $f,g,h\in\mathcal{S}(\mathbb{R}^2_{x,v})$. Then it holds that for all $t\geq0$ and $ 0<\kappa\leq1$
\begin{equation}\label{trilinear-est1-G}
\begin{split}
&\left|\left(G_{\kappa}(t)\Delta_p\Gamma((G_{\kappa}(t))^{-1}f,(G_{\kappa}(t))^{-1}g),\Delta_p h\right)_{L^2_{x,v}}\right|
\\&
\lesssim\sum_{|j-p|\leq4}\left\|\Delta_jf\right\|_{L^2_{x,v}}
\|\mathcal{H}^{\frac{s}{2}}g\|_{L^2_v(B^{1/2}_{2,1})}\|\mathcal{H}^{\frac{s}{2}}\Delta_ph\|_{L^2_{x,v}}
\\&\quad
+\sum_{j\geq\,p-4}\left\|f\right\|_{L^2_{v}(B^{1/2}_{2,1})}
\|\mathcal{H}^{\frac{s}{2}}\Delta_jg\|_{L^2_{x,v}}\|\mathcal{H}^{\frac{s}{2}}\Delta_ph\|_{L^2_{x,v}},
\end{split}
\end{equation}
with
\begin{equation*}
\begin{split}
&G_{\kappa}(t)=\frac{\exp\Big(t\Big(\mathcal{H}^{\frac{s+1}{2}}+\langle D_x\rangle^{\frac{3s+1}{2s+1}}\Big)^{\frac{2s}{3s+1}}\Big)}{1+\kappa\exp\Big(t\Big(\mathcal{H}^{\frac{s+1}{2}}+\langle D_x\rangle^{\frac{3s+1}{2s+1}}\Big)^{\frac{2s}{3s+1}}\Big)}.
\end{split}
\end{equation*}
\end{lemma}

\begin{proof}
Firstly, recalling Bony's decomposition, one can write $\Delta_p\left(uv\right)$ as follows
$$
\Delta_p\left(uv\right)=\Delta_p\left(\mathbf{T}_uv+\mathbf{T}_vu+\mathbf{R}(u,v)\right),
$$
where $\mathbf{T}$ and $\mathbf{R}$ are called as ``paraproduct" and ``remainder".  They are defined formally by
$$
\mathbf{T}_uv=\sum_{j}S_{j-1}u\Delta_jv,\quad \mathbf{R}(u,v)=\sum_{j}\sum_{|j-j'|\leq1}\Delta_{j'}u\Delta_{j}v, \quad\text{for}\,\, u,v\in\mathcal{S}'(\mathbb{R}).
$$
Notice that
\begin{align*}
&\Delta_p\mathbf{T}_uv=\sum_{j}\Delta_p(S_{j-1}u\Delta_jv)=\sum_{|j-p|\leq4}\Delta_p(S_{j-1}u\Delta_jv),\\
&\Delta_p\mathbf{T}_vu=\sum_{j}\Delta_p(\Delta_juS_{j-1}v)=\sum_{|j-p|\leq4}\Delta_p(\Delta_juS_{j-1}v),\\
&\Delta_p\mathbf{R}(u,v)=\sum_{j}\sum_{|j-j'|\leq1}\Delta_p(\Delta_{j}u\Delta_{j'}v)
=\sum_{\max(j,j')\geq\,p-2}\sum_{|j-j'|\leq1}\Delta_p(\Delta_{j}u\Delta_{j'}v),
\end{align*}
so it holds that
\begin{align*}
\Delta_p\left(uv\right)=\sum_{|j-p|\leq4}\Delta_p(S_{j-1}u\Delta_jv)+\sum_{|j-p|\leq4}\Delta_p(\Delta_juS_{j-1}v)
+\sum_{\max(j,j')\geq\,p-2}\sum_{|j-j'|\leq1}\Delta_p(\Delta_{j}u\Delta_{j'}v).
\end{align*}
Since
\begin{equation*}
\begin{split}
&G_{\kappa}(t)=\frac{\exp\Big(t\Big(\mathcal{H}^{\frac{s+1}{2}}+\langle D_x\rangle^{\frac{3s+1}{2s+1}}\Big)^{\frac{2s}{3s+1}}\Big)}{1+\kappa\exp\Big(t\Big(\mathcal{H}^{\frac{s+1}{2}}+\langle D_x\rangle^{\frac{3s+1}{2s+1}}\Big)^{\frac{2s}{3s+1}}\Big)}=\sum_{n=0}^{+\infty}\mathcal{G}_{\kappa,n}\mathbb{P}_n,
\end{split}
\end{equation*}
where $\mathcal{G}_{\kappa,n}$ is given by \eqref{G-N} and $\mathbb{P}_n$ denotes the orthogonal projections onto the Hermite basis described in Section \ref{S5}. Taking $u=(\mathcal{G}_{\kappa,k}(t))^{-1}f_{k}, v=(\mathcal{G}_{\kappa,l}(t))^{-1}g_{l}$ for $f,g,h\in\mathcal{S}(\mathbb{R}^2_{x,v})$, it follows from Lemma \ref{J10-1} that
\begin{align*}
&\left(G_{\kappa}(t)\Delta_p\Gamma((G_{\kappa}(t))^{-1}f,(G_{\kappa}(t))^{-1}g),\Delta_ph\right)_{L^2_{x,v}}
\\&
=\sum_{n=0}^{+\infty}\sum_{\substack{k+l=n\\k,l\geq0}}\alpha_{k,l}
\left(\mathcal{G}_{\kappa,n}(t)\Delta_p\left([(\mathcal{G}_{\kappa,k}(t))^{-1}f_{k}][(\mathcal{G}_{\kappa,l}(t))^{-1}g_{l}]\right),\Delta_ph_{n}\right)_{L^2_x}
\\&
=\sum_{|j-p|\leq4}\sum_{n=0}^{+\infty}\sum_{\substack{k+l=n\\k,l\geq0}}\alpha_{k,l}\left(
\mathcal{G}_{\kappa,n}(t)\Delta_p\left(S_{j-1}[(\mathcal{G}_{\kappa,k}(t))^{-1}f_{k}]\Delta_{j}[(\mathcal{G}_{\kappa,l}(t))^{-1}g_{l}]\right),\Delta_ph_{n}\right)_{L^2_x}
\\& \hspace{5mm}
+\sum_{|j-p|\leq4}\sum_{n=0}^{+\infty}\sum_{\substack{k+l=n\\k,l\geq0}}\alpha_{k,l}\left(
\mathcal{G}_{\kappa,n}(t)\Delta_p\left(\Delta_{j}[(\mathcal{G}_{\kappa,k}(t))^{-1}f_{k}]S_{j-1}[(\mathcal{G}_{\kappa,l}(t))^{-1}g_{l}]\right),\Delta_ph_{n}\right)_{L^2_x}
\\& \hspace{5mm}
+\sum_{\max(j,j')\geq p-2}\sum_{|j-j'|\leq1}\sum_{n=0}^{+\infty}\sum_{\substack{k+l=n\\k,l\geq0}}\alpha_{k,l}\left(
\mathcal{G}_{\kappa,n}(t)\Delta_p\left(\Delta_{j}[(\mathcal{G}_{\kappa,k}(t))^{-1}f_{k}]\Delta_{j'}[(\mathcal{G}_{\kappa,l}(t))^{-1}g_{l}]\right),\Delta_ph_{n}\right)_{L^2_x}
\\& \triangleq A_1+A_2+A_3.
\end{align*}
For $A_1$, since $[S_{j-1}, (\mathcal{G}_{\kappa,k}(t))^{-1}]=0$ and $[\Delta_{j}, (\mathcal{G}_{\kappa,l}(t))^{-1}]=0$ with $j\geq1, 0<\kappa\leq1, k\geq0$, we obtain
\begin{equation*}
\begin{split}
&|A_1|
\\&
\leq\sum_{|j-p|\leq4}\sum_{n=0}^{+\infty}\sum_{\substack{k+l=n\\k,l\geq0}}|\alpha_{k,l}|
\|\mathcal{G}_{\kappa,n}(t)\left(S_{j-1}[(\mathcal{G}_{\kappa,k}(t))^{-1}f_{k}]\Delta_{j}[(\mathcal{G}_{\kappa,l}(t))^{-1}g_{l}]\right)\|_{L^2_x}
\|\Delta_p^2h_{n}\|_{L^2_x}
\\&
=\sum_{|j-p|\leq4}\sum_{n=0}^{+\infty}\sum_{\substack{k+l=n\\k,l\geq0}}|\alpha_{k,l}|
\|\mathcal{G}_{\kappa,n}(t)\left((\mathcal{G}_{\kappa,k}(t))^{-1}[S_{j-1}f_{k}](\mathcal{G}_{\kappa,l}(t))^{-1}[\Delta_{j}g_{l}]\right)\|_{L^2_x}
\|\Delta_p^2h_{n}\|_{L^2_x}
\\&
\leq\sum_{|j-p|\leq4}\sum_{n=0}^{+\infty}\sum_{\substack{k+l=n\\k,l\geq0}}|\alpha_{k,l}|
\|S_{j-1}f_{k}\|_{B^{1/2}_{2,1}}\|\Delta_{j}g_{l}\|_{L^2_x}\|\Delta_ph_{n}\|_{L^2_x}
\\&
\leq\sum_{|j-p|\leq4}\sum_{n=0}^{+\infty}|\alpha_{0,n}|\|f_{0}\|_{B^{1/2}_{2,1}}\|\Delta_{j}g_{n}\|_{L^2_x}\|\Delta_ph_{n}\|_{L^2_x}
\\&
\hspace{5mm}+\sum_{|j-p|\leq4}\sum_{n=0}^{+\infty}\sum_{\substack{2k+l=n\\k\geq1,l\geq0}}|\alpha_{2k,l}|
\|f_{2k}\|_{B^{1/2}_{2,1}}\|\Delta_{j}g_{l}\|_{L^2_x}\|\Delta_ph_{n}\|_{L^2_x},
\end{split}
\end{equation*}
where we used Lemma \ref{trilinear-estimate-M-G} in the forth line, and Lemma \ref{J10-1} and Lemma \ref{J7-3} in the last two line. Bounding
$A_2, A_3$ are similar, we have
\begin{equation*}
\begin{split}
&|A_2|
\\&
\leq\sum_{|j-p|\leq4}\sum_{n=0}^{+\infty}\sum_{\substack{k+l=n\\k,l\geq0}}|\alpha_{k,l}|
\|\mathcal{G}_{\kappa,n}(t)\left(\Delta_{j}[(\mathcal{G}_{\kappa,k}(t))^{-1}f_{k}]S_{j-1}[(\mathcal{G}_{\kappa,l}(t))^{-1}g_{l}]\right)\|_{L^2_x}
\|\Delta_p^2h_{n}\|_{L^2_x}
\\
&\hspace{5mm} +\sum_{|j-p|\leq4}\sum_{n=0}^{+\infty}\sum_{\substack{2k+l=n\\k\geq1,l\geq0}}|\alpha_{2k,l}|
\|\Delta_{j}f_{2k}\|_{L^2_x}\|g_{l}\|_{B^{1/2}_{2,1}}\|\Delta_ph_{n}\|_{L^2_x},
\end{split}
\end{equation*}
\begin{equation*}
\begin{split}
&|A_3|
\\&
\leq\sum_{\max(j,j')\geq p-2}\sum_{|j-j'|\leq1}\sum_{n=0}^{+\infty}\sum_{\substack{k+l=n\\k,l\geq0}}|\alpha_{k,l}|
\|\mathcal{G}_{\kappa,n}(t)\left(\Delta_{j}[(\mathcal{G}_{\kappa,k}(t))^{-1}f_{k}]\Delta_{j'}[(\mathcal{G}_{\kappa,l}(t))^{-1}g_{l}]\right)\|_{L^2_x}
\|\Delta_p^2h_{n}\|_{L^2_x}
\\&
\hspace{5mm} +\sum_{j\geq p-3}\sum_{n=0}^{+\infty}\sum_{\substack{2k+l=n\\k\geq1,l\geq0}}|\alpha_{2k,l}|
\|\Delta_{j}f_{2k}\|_{L^2_x}\|g_{l}\|_{B^{1/2}_{2,1}}\|\Delta_ph_{n}\|_{L^2_x}.
\end{split}
\end{equation*}
Combining the three estimates for $A_1, A_2, A_3$ implies that
\begin{align*}
&\left|\left(G_{\kappa}(t)\Delta_p\Gamma((G_{\kappa}(t))^{-1}f,(G_{\kappa}(t))^{-1}g),\Delta_ph\right)_{L^2_{x,v}}\right|
\\&
\leq\sum_{|j-p|\leq4}\sum_{n=0}^{+\infty}|\alpha_{0,n}|\|f_{0}\|_{B^{1/2}_{2,1}}\|\Delta_{j}g_{n}\|_{L^2_x}\|\Delta_ph_{n}\|_{L^2_x}
\\&
+\sum_{j\geq p-4}\sum_{n=0}^{+\infty}|\alpha_{0,n}|\|\Delta_{j}f_{0}\|_{L^2_x}\|g_{n}\|_{B^{1/2}_{2,1}}\|\Delta_ph_{n}\|_{L^2_x}
\\&
+\sum_{|j-p|\leq4}\sum_{n=0}^{+\infty}\sum_{\substack{2k+l=n\\k\geq1,l\geq0}}|\alpha_{2k,l}|
\|f_{2k}\|_{B^{1/2}_{2,1}}\|\Delta_{j}g_{l}\|_{L^2_x}\|\Delta_ph_{n}\|_{L^2_x}
\\&
+\sum_{j\geq p-4}\sum_{n=0}^{+\infty}\sum_{\substack{2k+l=n\\k\geq1,l\geq0}}|\alpha_{2k,l}|
\|\Delta_{j}f_{2k}\|_{L^2_x}\|g_{l}\|_{B^{1/2}_{2,1}}\|\Delta_ph_{n}\|_{L^2_x}
\\&
\triangleq J_{1}+J_{2}+J_{3}+J_{4}.
\end{align*}
By using the formula \eqref{A-55} and \eqref{non-1}, we arrive at
\begin{equation}\label{J1-2}
\begin{split}
J_{1}+J_{2}
&\leq\sum_{|j-p|\leq4}\sum_{n=0}^{+\infty}(n+\frac{1}{2})^{s}\|f_{0}\|_{B^{1/2}_{2,1}}\|\Delta_{j}g_{n}\|_{L^2_x}\|\Delta_ph_{n}\|_{L^2_x}
\\&\quad
+\sum_{j\geq\,p-4}\sum_{n=0}^{+\infty}(n+\frac{1}{2})^{s}\|\Delta_{j}f_{0}\|_{L^2_x}\|g_{n}\|_{B^{1/2}_{2,1}}\|\Delta_ph_{n}\|_{L^2_x}
\\&
\leq\sum_{|j-p|\leq4}\|f_{0}\|_{B^{1/2}_{2,1}}\|\Delta_j\mathcal{H}^{\frac{s}{2}}g\|_{L^2_{v}L^2_x}\|\mathcal{H}^{\frac{s}{2}}\Delta_ph\|_{L^2_{x,v}}
\\&\quad
+\sum_{j\geq\,p-4}\|\Delta_jf_{0}\|_{L^2_x}\|\mathcal{H}^{\frac{s}{2}}g\|_{L^2_{v}(B^{1/2}_{2,1})}\|\mathcal{H}^{\frac{s}{2}}\Delta_ph\|_{L^2_{x,v}}.
\end{split}
\end{equation}
On the other hand, for $J_{3}$ and $J_{4}$, by using the Lemma \ref{J10-2} once again, we obtain
\begin{equation}\label{J3-4}
\begin{split}
J_{3}+J_{4}
&\leq\sum_{|j-p|\leq4}\sum_{k\geq1,l\geq0}\frac{\tilde{\mu}_{k,l}}{k^{\frac{3}{4}}}\|f_{2k}\|_{B^{1/2}_{2,1}}\|\Delta_{j}g_{l}\|_{L^2_x}
\|\Delta_ph_{2k+l}\|_{L^2_x}
\\&
\quad+\sum_{j\geq\,p-4}\sum_{k\geq1,l\geq0}\frac{\tilde{\mu}_{k,l}}{k^{\frac{3}{4}}}\|\Delta_{j}f_{2k}\|_{L^2_x}\|g_{l}\|_{B^{1/2}_{2,1}}\|\Delta_ph_{2k+l}\|_{L^2_x}
\\&
\leq\sum_{|j-p|\leq4}\left\|f\right\|_{L^2_{v}(B^{1/2}_{2,1})}\|\mathcal{H}^{\frac{s}{2}}\Delta_jg\|_{L^2_{x,v}}
\Big(\sum_{l=0}^{+\infty}\sum_{k=1}^{+\infty}\frac{\tilde{\mu}^{2}_{k,l}}{k^{\frac{3}{2}}(l+\frac{1}{2})^{s}}\|\Delta_{p}h_{2k+l}\|^{2}_{L^2_x}\Big)^{\frac{1}{2}}
\\&
\quad+\sum_{j\geq\,p-4}\left\|\Delta_jf\right\|_{L^2_{x,v}}\|\mathcal{H}^{\frac{s}{2}}g\|_{L^2_{v}(B^{1/2}_{2,1})}
\Big(\sum_{l=0}^{+\infty}\sum_{k=1}^{+\infty}\frac{\tilde{\mu}^{2}_{k,l}}{k^{\frac{3}{2}}(l+\frac{1}{2})^{s}}\|\Delta_{p}h_{2k+l}\|^{2}_{L^2_x}\Big)^{\frac{1}{2}}.
\end{split}
\end{equation}
Therefore, we conclude that from \eqref{HHH}, \eqref{J1-2} and \eqref{J3-4}
\begin{align*}
&\left|(G_{\kappa}(t)\Delta_{p}\Gamma((G_{\kappa}(t))^{-1}f,(G_{\kappa}(t))^{-1}g),\Delta_{p}h)_{L^2_{x,v}}\right|
\\&
\lesssim \sum_{|j-p|\leq4}\left\|f\right\|_{L^2_{v}(B^{1/2}_{2,1})}\|\mathcal{H}^{\frac{s}{2}}\Delta_jg\|_{L^2_{x,v}}\|\mathcal{H}^{\frac{s}{2}}\Delta_{p}h\|_{L^2_{x,v}}
\\&\quad
+\sum_{j\geq\,p-4}\left\|\Delta_jf\right\|_{L^2_{x,v}}\|\mathcal{H}^{\frac{s}{2}}g\|_{L^2_{v}(B^{1/2}_{2,1})}\|\mathcal{H}^{\frac{s}{2}}\Delta_{p}h\|_{L^2_{x,v}},
\end{align*}
which is \eqref{trilinear-est1-G} exactly. The proof of Lemma \ref{trilinear-estimate-G} is completed.
\end{proof}

Take $\kappa=t=0$ in  Lemma \ref{trilinear-estimate-G}, we have the following consequence.
\begin{remark}\label{Tri}
Let $f,g,h\in\mathcal{S}(\mathbb{R}^2_{x,v})$. Then it holds that
\begin{equation*}
\begin{split}
&\left|\left(\Delta_p\Gamma(f,(g),\Delta_p h\right)_{L^2_{x,v}}\right|
\\&
\lesssim\sum_{|j-p|\leq4}\left\|\Delta_jf\right\|_{L^2_{x,v}}
\|\mathcal{H}^{\frac{s}{2}}g\|_{L^2_vL^\infty_x}\|\mathcal{H}^{\frac{s}{2}}\Delta_ph\|_{L^2_{x,v}}
\\&\quad
+\sum_{j\geq\,p-4}\left\|f\right\|_{L^2_{v}L^\infty_x}
\|\mathcal{H}^{\frac{s}{2}}\Delta_jg\|_{L^2_{x,v}}\|\mathcal{H}^{\frac{s}{2}}\Delta_ph\|_{L^2_{x,v}}.
\end{split}
\end{equation*}
\end{remark}

\begin{lemma}\label{trilinear-G}
Assume $\sigma>0, T>0$ and $ 0<\kappa\leq1$. Let $f=f(t,x,v)$, $g=g(t,x,v)$ and $h=h(t,x,v)$ be three suitably functions such that all norms on the right of the following inequalities are well defined. Then there exists a constant $C_1>0$ such that
\begin{equation}\label{trilinear-est2-G}
\begin{split}
&\sum_{p\geq-1}2^{p\sigma}\left[\int^T_0\left|\left(G_{\kappa}(t)\Delta_p\Gamma((G_{\kappa}(t))^{-1}f,(G_{\kappa}(t))^{-1}g),\Delta_ph\right)
_{L^2_{x,v}}\right|dt\right]^{1/2}
\\&
\leq C_{1}\|f\|^{1/2}_{\widetilde{L}_T^{\infty}\widetilde{L}^2_{v}(B^\sigma_{2,1})}
\|\mathcal{H}^{\frac{s}{2}}g\|^{1/2}_{L_T^2L^2_v(B^{1/2}_{2,1})}
\|\mathcal{H}^{\frac{s}{2}}h\|^{1/2}_{\widetilde{L}_T^2\widetilde{L}^2_{v}(B^\sigma_{2,1})}
\\&\quad
+C_{1}\left\|f\right\|^{1/2}_{L_T^{\infty}L^2_v(B^{1/2}_{2,1})}
\|\mathcal{H}^{\frac{s}{2}}g\|^{1/2}_{\widetilde{L}_T^2\widetilde{L}^2_{v}(B^\sigma_{2,1})}
\|\mathcal{H}^{\frac{s}{2}}h\|^{1/2}_{\widetilde{L}_T^2\widetilde{L}^2_{v}(B^\sigma_{2,1})}.
\end{split}
\end{equation}
\end{lemma}

\begin{proof}
Based on Lemma \ref{trilinear-estimate-G}, it follows from Cauchy-Schwarz inequality that
\begin{align*}
&\sum_{p\geq-1}2^{p\sigma}\left[\int^T_0\left|\left(G_{\kappa}(t)\Delta_p\Gamma((G_{\kappa}(t))^{-1}f,(G_{\kappa}(t))^{-1}g),
\Delta_ph\right)_{L^2_{x,v}}\right|dt\right]^{1/2}
\\&
\lesssim\sum_{p\geq-1}2^{p\sigma}\left[\sum_{|j-p|\leq4}\int^T_0\left\|f\right\|_{L^2_v(B^{1/2}_{2,1})}
\left\|\mathcal{H}^{\frac{s}{2}}\Delta_jg\right\|_{L^2_{x,v}}
\left\|\mathcal{H}^{\frac{s}{2}}\Delta_ph\right\|_{L^2_{x,v}}dt\right]^{1/2}
\\&\quad
+\sum_{p\geq-1}2^{p\sigma}\left[\sum_{j\geq\,p-4}\int^T_0\left\|\Delta_jf\right\|_{L^2_{x,v}}
\left\|\mathcal{H}^{\frac{s}{2}}g\right\|_{L^2_{v}(B^{1/2}_{2,1})}
\left\|\mathcal{H}^{\frac{s}{2}}\Delta_ph\right\|_{L^2_{x,v}}dt\right]^{1/2}
\\&
\lesssim\left\|f\right\|_{L^\infty_{T}L^2_v(B^{1/2}_{2,1})}^{1/2}\sum_{p\geq-1}2^{p\sigma}
\Big(\sum_{|j-p|\leq4}\left\|\mathcal{H}^{\frac{s}{2}}\Delta_jg\right\|_{L^{2}_{T}L^2_{x,v}}\Big)^{1/2}
\left\|\mathcal{H}^{\frac{s}{2}}\Delta_ph\right\|_{L^2_{T}L^2_{x,v}}^{1/2}
\\&\quad
+\left\|\mathcal{H}^{\frac{s}{2}}g\right\|_{L^2_{T}L^2_{v}(B^{1/2}_{2,1})}^{1/2}
\sum_{p\geq-1}2^{p\sigma}\Big(\sum_{j\geq\,p-4}\left\|\Delta_jf\right\|_{L^{\infty}_{T}L^2_{x,v}}\Big)^{1/2}
\left\|\mathcal{H}^{\frac{s}{2}}\Delta_ph\right\|_{L^2_{T}L^2_{x,v}}^{1/2}
\\&
\lesssim\left\|f\right\|_{L^{\infty}_{T}L^2_{v}(B^{1/2}_{2,1})}^{1/2}
\|\mathcal{H}^{\frac{s}{2}}g\|^{1/2}_{\widetilde{L}_T^2\widetilde{L}^2_{v}(B^\sigma_{2,1})}
\|\mathcal{H}^{\frac{s}{2}}h\|^{1/2}_{\widetilde{L}_T^2\widetilde{L}^2_{v}(B^\sigma_{2,1})}
\left(\sum_{p\geq-1}\sum_{|j-p|\leq4}2^{(p-j)\sigma}c(j)\right)^{1/2}
\\&\quad
+\|\mathcal{H}^{\frac{s}{2}}g\|^{1/2}_{L_T^2L^2_v(B^{1/2}_{2,1})}
\|\mathcal{H}^{\frac{s}{2}}h\|^{1/2}_{\widetilde{L}_T^2\widetilde{L}^2_{v}(B^\sigma_{2,1})}
\left(\sum_{p\geq-1}\sum_{j\geq\,p-4}2^{p\sigma}\left\|\Delta_jf\right\|_{L^{\infty}_{T}L^2_{x,v}}\right)^{1/2},
\end{align*}
where $c(j)=2^{j\sigma}\left\|\mathcal{H}^{\frac{s}{2}}\Delta_jg\right\|_{L_T^{2}L^2_{x,v}}/
\left\|\mathcal{H}^{\frac{s}{2}}g\right\|_{\widetilde{L}_T^{2}\widetilde{L}^2_{v}(B^\sigma_{2,1})}$
and $\|c(j)\|_{\ell^{1}}\leq1$. Hence, with Fubini's theorem and Young's inequality, we have
\begin{align*}
\sum_{p\geq-1}\sum_{|j-p|\leq4}2^{(p-j)\sigma}c(j)&=\sum_{p\geq-1}[(\mathbf{1}_{|j|\leq4}2^{j\sigma})\ast c(j)](p)
\\&\leq\|\mathbf{1}_{|j|\leq4}2^{j\sigma}\|_{\ell^{1}}\|c(j)\|_{\ell^{1}}<+\infty.
\end{align*}
Since
$$
\sum_{-1\leq\,p\leq\,j+4}2^{(p-j)\sigma}=2^{-(j+1)\sigma}+2^{-j\sigma}+\cdots+2^{4\sigma}=\frac{2^{4\sigma}(1-2^{-j\sigma})}{1-2^{-\sigma}}<+\infty,
$$
it follows that
\begin{align*}
\sum_{p\geq-1}\sum_{j\geq\,p-4}2^{p\sigma}\left\|\Delta_jf\right\|_{L^{\infty}_{T}L^2_{x,v}}
&=\sum_{j\geq-1}2^{j\sigma}\Big(\sum_{-1\leq\,p\leq\,j+4}2^{(p-j)\sigma}\Big)\left\|\Delta_jf\right\|_{L^{2}_{T}L^2_{x,v}}
\\&\lesssim\sum_{j\geq-1}2^{j\sigma}\left\|\Delta_jf\right\|_{L^{\infty}_{T}L^2_{x,v}}
=\|f\|_{\widetilde{L}_T^\infty\widetilde{L}^2_{v}(B^\sigma_{2,1})}.
\end{align*}
Consequently, we conclude that
\begin{align*}
&\sum_{p\geq-1}2^{p\sigma}\left[\int^T_0\left|\left(G_{\kappa}(t)\Delta_p\Gamma((G_{\kappa}(t))^{-1}f,(G_{\kappa}(t))^{-1}g),\Delta_ph\right)\right|dt\right]^{1/2}
\\&
\leq C_{1}\left\|f\right\|^{1/2}_{L_T^{\infty}L^2_v(B^{1/2}_{2,1})}
\|\mathcal{H}^{\frac{s}{2}}g\|^{1/2}_{\widetilde{L}_T^2\widetilde{L}^2_{v}(B^\sigma_{2,1})}
\|\mathcal{H}^{\frac{s}{2}}h\|^{1/2}_{\widetilde{L}_T^2\widetilde{L}^2_{v}(B^\sigma_{2,1})}
\\&\quad
+C_{1}\|f\|^{1/2}_{\widetilde{L}_T^{\infty}\widetilde{L}^2_{v}(B^\sigma_{2,1})}
\|\mathcal{H}^{\frac{s}{2}}g\|^{1/2}_{L_T^2L^2_v(B^{1/2}_{2,1})}
\|\mathcal{H}^{\frac{s}{2}}h\|^{1/2}_{\widetilde{L}_T^2\widetilde{L}^2_{v}(B^\sigma_{2,1})}.
\end{align*}
Hence, the proof of Lemma \ref{trilinear-G} is complete.
\end{proof}

Similarly, it follows from Remark \ref{Tri}, Lemma \ref{J7-1} and Lemma \ref{J7-5} that

\begin{corollary}\label{trilinear-G-R}
Set $T>0$. Let $f=f(t,x,v)$, $g=g(t,x,v)$ and $h=h(t,x,v)$ be three suitably functions. Then it holds that
\begin{equation*}
\sum_{p\geq-1}2^{\frac{p}{2}}\left(\int_{0}^{T}\left|\left(\Delta_{p}\Gamma(f,g),\Delta_{p}g\right)_{L^{2}_{x,v}}\right|dt\right)^{1/2}
\lesssim\left\|f\right\|^{1/2}_{\widetilde{L}_T^{\infty}\widetilde{L}^2_v(B^{1/2}_{2,1})}
\|\mathcal{H}^{\frac{s}{2}}g\|^{1/2}_{\widetilde{L}_T^2\widetilde{L}^2_{v}(B^{1/2}_{2,1})}
\|\mathcal{H}^{\frac{s}{2}}h\|^{1/2}_{\widetilde{L}_T^2\widetilde{L}^2_{v}(B^{1/2}_{2,1})}.
\end{equation*}
\end{corollary}

\section{The local existence of weak solution}\label{S3}
This section is devoted to proving the local existence of weak solution to the Cauchy problem \eqref{eq-1}.

\subsection{The local existence of weak solution}
We first state the local-in-time existence of weak solution to \eqref{eq-1}.
\begin{theorem}[{\bf Local existence}]\label{local existence}
Let $0<T<+\infty$. We assume that the collision cross section satisfies \eqref{A-5} with $0<s<1$. There exists a constant $\varepsilon_{0}>0$ such that for all $g_0\in\widetilde{L}^{2}_{v}(B^{1/2}_{2,1})$ fulfilling
$$
\|g_0\|_{\widetilde{L}_v^2(B^{1/2}_{2,1})}\leq\,\varepsilon_{0},
$$
then \eqref{eq-1} admits a weak solution $g\in L^{\infty}([0,T];L^{2}(\mathbb{R}^{2}_{x,v}))$ satisfying
\begin{equation}\label{local-A}
\begin{split}	&\|g\|_{\widetilde{L}^{\infty}_T\widetilde{L}^2_{v}(B^{1/2}_{2,1})}+\|\mathcal{H}^{\frac{s}{2}}g\|_{\widetilde{L}^{2}_T\widetilde{L}^2_{v}(B^{1/2}_{2,1})}
\leq c_{0}e^{T}\|g_{0}\|_{\widetilde{L}^2_{v}(B^{1/2}_{2,1})},
\end{split}
\end{equation}	
for some constant $c_{0}>1$.
\end{theorem}

\begin{remark}
Furthermore, we can obtain the uniqueness of solutions to the Cauchy problem \eqref{eq-1} among the small solutions satisfying \eqref{local-A}. The uniqueness of solutions will be used in establishing the Gelfand-Shilov and Gevrey regularizing effects as in Section \ref{S4}.
\end{remark}

Actually, in order to prove the above theorem, the following local existence of linearized Kac equation is necessary.
\begin{theorem}[{\bf Local existence for linearized equation}]\label{local-regularity}
There exists a constant $\varepsilon_{0}>0$ such that for all $T>0, g_{0}\in\widetilde{L}^2_{v}(B^{1/2}_{2,1}), f\in\widetilde{L}^\infty_{T}\widetilde{L}^2_{v}(B^{1/2}_{2,1}),$ satisfying
$$
\|f\|_{\widetilde{L}^{\infty}_T\widetilde{L}^2_{v}(B^{1/2}_{2,1})}\leq\varepsilon_{0},
$$
then the Cauchy problem
\begin{equation}\label{local-B1}
\left\{\begin{aligned}
&\partial_tg+v\partial_xg+\mathcal{K}g=\Gamma(f,g),\\
&g(t,x,v)|_{t=0}=g_{0}(x,v),
\end{aligned} \right.
\end{equation}
admits a weak solution $g\in L^{\infty}([0,T];L^{2}(\mathbb{R}^{2}_{x,v}))$ satisfying
\begin{equation}\label{local-B2}
\begin{split}	&\|g\|_{\widetilde{L}^{\infty}_T\widetilde{L}^2_{v}(B^{1/2}_{2,1})}+\|\mathcal{H}^{\frac{s}{2}}g\|_{\widetilde{L}^{2}_T\widetilde{L}^2_{v}(B^{1/2}_{2,1})}
\leq c_{0}e^{T}\|g_{0}\|_{\widetilde{L}^2_{v}(B^{1/2}_{2,1})},
\end{split}
\end{equation}
for some constant $c_{0}>1$.	
\end{theorem}

Once having this result, Theorem \ref{local existence} follows from the standard procedure.
\begin{proof}
Let $0<\lambda<1, T>0$ and $g_{0}\in\widetilde{L}^2_{v}(B^{1/2}_{2,1})$ be the initial fluctuation satisfying
\begin{align}\label{g}
\|g_{0}\|_{\widetilde{L}^2_{v}(B^{1/2}_{2,1})}\leq\tilde{\varepsilon}_{0}, \ \text{with} \ \ 0<\tilde{\varepsilon}_{0}
=\inf\Big(\frac{\varepsilon_0}{c_{0}e^{T}},\frac{1}{4c_{0}CC_{1}^{2}e^{3T}},\frac{\lambda}{\sqrt{2C}c_{0}C_{1}^{2}e^{3T}}\Big)\leq\varepsilon_0,
\end{align}
where $C ,C_{1},c_{0},\varepsilon_0$ are those constants defined in Lemma \ref{JJ}, Lemma \ref{trilinear-G} and Theorem \ref{local existence}. We define
\begin{align*}
\tilde{g}_{0}=\exp\left(-\delta t(\sqrt{\mathcal{H}}+\langle D_x\rangle)^{\frac{2s}{2s+1}}\right)g_0, \ \ 0\leq t\leq T
\end{align*}
with $0\leq\delta\leq1$. With aid of Theorem \ref{local-regularity}, we prove the local existence of solutions to the nonlinear Kac equation by constructing a local solution to the Cauchy problem \eqref{eq-1} for the nonlinear Kac equation as the limit of the following sequence of iterating approximate solutions:
\begin{equation}\label{equationA}
\left\{
\begin{aligned}
&\partial_t \tilde{g}_{n+1}+v\partial_x\tilde{g}_{n+1}+\mathcal{K}\tilde{g}_{n+1}=\Gamma(\tilde{g}_n, \tilde{g}_{n+1}),\,\,\,\, n\geq0,
\\&\tilde{g}_{n+1}(t,x,v)|_{t=0}=g_{0}(x,v).
\end{aligned} \right.
\end{equation}
The procedure is standard, which is similar to that of Theorem $4.2$ in \cite{LMPX2}. Here, we omit details for simplicity.
\end{proof}

In order to show Theorem \ref{local-regularity}, we need to develop the regularization method in \cite{LMPX2}. The proof can be divided into several steps for clarity.
\subsection{The local weak solution of linearized Kac equation}
In the first step, we give the existence of local weak solution with the rough initial datum, the interested reader is referred to Lemma 4.1 in \cite{LMPX2} for similar details.
\begin{proposition}\label{local-1A}
There exists a constant $\varepsilon_{0}>0$ such that for all $T>0, g_{0}\in L^{2}(\mathbb{R}^2_{x,v}), f\in L^{\infty}([0,T]\times\mathbb{R}_{x};L^{2}(\mathbb{R}_{v}))$ satisfying
$$
\|f\|_{L^{\infty}([0,T]\times\mathbb{R}_{x};L^{2}_{v})}\leq\varepsilon_{0},
$$
then the Cauchy problem \eqref{local-B1} admits a weak solution
$$
g\in L^{\infty}([0,T];L^{2}(\mathbb{R}^{2}_{x,v})).
$$
\end{proposition}

Next, we turn to prove the regularity with respect to $x$ and $v$, which is shown by the following two subsections.

\subsection{Regularity of weak solution in velocity variable}
A rigorous proof of Theorem \ref{local-regularity} is to mollifier the weak solution $g\in L^{\infty}([0,T];L^{2}(\mathbb{R}^{2}_{x,v}))$ in velocity and position variables. To do this, we mollifier the function $f$, that is, setting $f_{N}=\sum^{N-1}_{p\geq-1}\Delta_{p}f$ for $N\in\mathbb{N}$, then we have $f_N\in \widetilde{L}^\infty_{T}\widetilde{L}^2_{v}(H^{+\infty}_{x})$. For each $f_N(N\in\mathbb{N})$, we consider a weak solution $g_{N}\in L^{\infty}([0,T]; L^{2}(\mathbb{R}^{2}_{x,v}))$ to the following Cauchy problem
\begin{equation}\label{local-111A}
\left\{\begin{aligned}
&\partial_tg_N+v\partial_xg_N+\mathcal{K}g_N=\Gamma(f_N,g_N),\\
&g_N(t,x,v)|_{t=0}=g_{0}(x,v).
\end{aligned} \right.
\end{equation}
Some simple calculations enable us to obtain the following proposition for $f_N$.

\begin{proposition}\label{Cauchy-f}	
If $f\in\widetilde{L}^{\infty}_{T}\widetilde{L}^{2}_{v}(B^{1/2}_{2,1})$. For $N\in\mathbb{N}$, put $f_N=S_Nf=\sum^{N-1}_{p\geq-1}\Delta_{p}f$. Then we get

i) If $f_N\in\widetilde{L}^{\infty}_{T}\widetilde{L}^{2}_{v}(B^{1/2}_{2,1})$, then
$\{f_N\}$ is a Cauchy sequence in $\widetilde{L}^{\infty}_{T}\widetilde{L}^{2}_{v}(B^{1/2}_{2,1})$.

ii) For $0<\sigma\leq1/2$, $f_N$ satisfies $\|f_N\|_{L^{\infty}_TL^2_{v}L^2_{x}}\leq C\|f\|_{L^{\infty}_TL^2_{v}L^2_{x}}$ and
\begin{equation*}
\|f_N\|_{L^{\infty}_TL^2_{v}L^2_{x}}\leq C_{2}\|f_N\|_{L^{\infty}_TL^2_{v}(B^{\sigma}_{2,1})}
\leq C_{3}\|f_N\|_{\widetilde{L}^{\infty}_T\widetilde{L}^2_{v}(B^{\sigma}_{2,1})}
\leq C_{4}\|f\|_{\widetilde{L}^{\infty}_T\widetilde{L}^2_{v}(B^{\sigma}_{2,1})},
\end{equation*}
where $C_{2},C_{3},C_{4}>0$ are constants independent of $N$.
\end{proposition}

Then, we can establish the following proposition for the weak solution $g_N$.

\begin{proposition}\label{local existence-gA}	
For $N\in\mathbb{N}$, put $f_N=\sum^{N-1}_{p\geq-1}\Delta_{p}f$. There exists a constant $\varepsilon_{0}>0$ such that for all $T>0, g_{0} \in L^2(\mathbb{R}^{2}_{x,v}), f\in L^{\infty}([0,T]\times\mathbb{R}_{x};L^{2}(\mathbb{R}_{v}))$ satisfying
$$
\|f\|_{L^{\infty}([0,T]\times\mathbb{R}_{x};L^{2}_{v})}\leq\varepsilon_{0},
$$
then the Cauchy problem \eqref{local-111A}
admits a weak solution $g_N(t,x,v)\in L^{\infty}([0,T];L^{2}(\mathbb{R}^{2}_{x,v}))$ such that
\begin{equation}\label{local-112A}
\begin{split}
\|g_N\|_{L^{\infty}_TL^2_{v}L^2_{x}}+\|\mathcal{H}^{\frac{s}{2}}g_N\|_{L^{2}_TL^2_{v}L^2_{x}}
\leq c_{0}e^{3T}\|g_{0}\|_{L^{2}_{v}L^2_{x}},
\end{split}
\end{equation}
for some constant $c_{0}>1$.
\end{proposition}

\begin{proof}
By applying Proposition \ref{local-1A}, we see that the Cauchy problem \eqref{local-111A} admits a weak solution $g_N(t,x,v)\in L^{\infty}([0,T];L^{2}(\mathbb{R}^{2}_{x,v}))$. It only need to show \eqref{local-112A} for a weak solution $g_N\in L^{\infty}([0,T];L^{2}(\mathbb{R}^{2}_{x,v}))$ under the assumption that $\|f_N\|_{L^{\infty}([0,T]\times\mathbb{R}_{x};L^{2}_{v})}$ is sufficiently small, independent of $N$.

It follows from \eqref{A-2}, \eqref{A-3} and Lemma \ref{trilinear-estimate-SSS} that
\begin{align*}
\mathcal{H}^{-s}\mathcal{K}g\in L^{\infty}([0,T];L^{2}(\mathbb{R}^{2}_{x,v})), \ \ \mathcal{H}^{-s}\Gamma(f,g)\in L^{\infty}([0,T];L^{2}(\mathbb{R}^{2}_{x,v})),
\end{align*}
for $f\in L^{\infty}([0,T]\times\mathbb{R}_{x};L^{2}(\mathbb{R}_{v})),g\in L^{\infty}([0,T];L^{2}(\mathbb{R}^{2}_{x,v}))$. Define
\begin{align}\label{KK1}
g_{\delta}=(1+\delta\sqrt{\mathcal{H}}+\delta\langle D_{x}\rangle)^{-1}g_{N}, \ \ 0<\delta\leq1.
\end{align}
Notice that
$$
(1+\delta\sqrt{\mathcal{H}}+\delta\langle D_{x}\rangle)g_{\delta}\in L^{\infty}([0,T];L^{2}(\mathbb{R}^{2}_{x,v}))\subset L^{2}([0,T];L^{2}(\mathbb{R}^{2}_{x,v})).
$$
According to Theorem 3 in \cite{E}, we deduce that the mapping
$$
t\mapsto\|g_{\delta}(t)\|^{2}_{L^{2}_{x,v}},
$$
is absolutely continuous with
\begin{align}\label{KK-1}
\frac{d}{dt}\Big(\|g_{\delta}\|^{2}_{L^{2}_{x,v}}\Big)=2\mathrm{Re}(\partial_{t}g_{\delta},g_{\delta})_{L^{2}_{x,v}}.
\end{align}
Taking the inner product of \eqref{local-111A} with $(1+\delta\sqrt{\mathcal{H}}+\delta\langle D_{x}\rangle)^{-2}g$ and integrating the resulting inequality with respect to $(x,v)\in\mathbb{R}^{2}$. It follows from \eqref{KK1} and \eqref{KK-1} that
\begin{equation*}
\begin{split}
&\frac{1}{2}\frac{d}{dt}\left(\left\|g_{\delta}\right\|^{2}_{L^{2}_{x,v}}\right)+\mathrm{Re}\left(\mathcal{K}g_{\delta},g_{\delta}\right)_{L^{2}_{x,v}}
+\mathrm{Re}\left(v\partial_{x}g_{\delta},g_{\delta}\right)_{L^{2}_{x,v}}
\\&\quad
+\mathrm{Re}([(1+\delta\sqrt{\mathcal{H}}+\delta\langle D_{x}\rangle)^{-1},v](1+\delta\sqrt{\mathcal{H}}+\delta\langle D_{x}\rangle)\partial_{x}g_{\delta},g_{\delta})_{L^{2}_{x,v}}
\\&
=\mathrm{Re}((1+\delta\sqrt{\mathcal{H}}+\delta\langle D_{x}\rangle)^{-1}\Gamma(f_{N},(1+\delta\sqrt{\mathcal{H}}+\delta\langle D_{x}\rangle)g_{\delta}),g_{\delta})_{L^{2}_{x,v}},
\end{split}
\end{equation*}
since $[(1+\delta\sqrt{\mathcal{H}}+\delta\langle D_{x}\rangle)^{-1},\mathcal{K}]=0$. Due to the coercivity estimate of the linearized Kac collision operator $\mathcal{K}$, we obtain, for all $0\leq t\leq T$,
\begin{equation}\label{local-S1}
\begin{split}
&\frac{1}{2}\frac{d}{dt}(\|g_{\delta}\|^{2}_{L^{2}_{x,v}})+\frac{1}{C}\|\mathcal{H}^{\frac{s}{2}}g_{\delta}\|^{2}_{L^{2}_{x,v}}-\|g_{\delta}\|^{2}_{L^{2}_{x,v}}
\\&
\leq|((1+\delta\sqrt{\mathcal{H}}+\delta\langle D_{x}\rangle)^{-1},v](1+\delta\sqrt{\mathcal{H}}+\delta\langle D_{x}\rangle)\partial_{x}g_{\delta},g_{\delta})_{L^{2}_{x,v}}|
\\&\quad
+|((1+\delta\sqrt{\mathcal{H}}+\delta\langle D_{x}\rangle)^{-1}\Gamma(f_{N},(1+\delta\sqrt{\mathcal{H}})g_{\delta}),g_{\delta})_{L^{2}_{x,v}}|
\\&\quad
+|((1+\delta\sqrt{\mathcal{H}}+\delta\langle D_{x}\rangle)^{-1}\Gamma(f_{N},\delta\langle D_{x}\rangle g_{\delta}),g_{\delta})_{L^{2}_{x,v}}|,
\end{split}
\end{equation}
since $\mathcal{K}$ is a selfadjoint operator and $\mathrm{Re}(v\partial_{x}g_{\delta},g_{\delta})_{L^{2}_{x,v}}=0$. Furthermore, it follows from Lemma \ref{trilinear-estimate-S} with $j_1=0,j_2=1$ that for all $0\leq\delta\leq1$,
\begin{equation}\label{local-S2}
\begin{split}
&\left|\left(\left(1+\delta\sqrt{\mathcal{H}}+\delta\langle D_{x}\rangle\right)^{-1}\Gamma\left(f_{N},(1+\delta\sqrt{\mathcal{H}})g_{\delta}\right),g_{\delta}\right)_{L^{2}_{x,v}}\right|
\\&\qquad
\leq C_{0}\|f_{N}\|_{L^{2}_{v}L^{\infty}_{x}}\|\mathcal{H}^{\frac{s}{2}}g_{\delta}\|^{2}_{L^{2}_{x,v}}.
\end{split}
\end{equation}
Similarly,
\begin{equation}\label{local-S3}
\begin{split}
&\left|\left((1+\delta\sqrt{\mathcal{H}}+\delta\langle D_{x}\rangle)^{-1}\Gamma(f_{N},\delta\langle D_{x}\rangle g_{\delta}),g_{\delta}\right)_{L^{2}_{x,v}}\right|
\\&
=\left|\left(\Gamma(f_{N},\delta\langle D_{x}\rangle(1+\delta\sqrt{\mathcal{H}}+\delta\langle D_{x}\rangle)^{-1}g_{N}),M_{\delta}g_{\delta}\right)_{L^{2}_{x,v}}\right|
\\&
\leq C_{0}\|f_{N}\|_{L^{2}_{v}L^{\infty}_{x}}\|\mathcal{H}^{\frac{s}{2}}g_{N}\|_{L^{2}_{x,v}}\|\mathcal{H}^{\frac{s}{2}}g_{\delta}\|_{L^{2}_{x,v}}.
\end{split}
\end{equation}
Thanks to the commutator estimate in (4.10) of \cite{LMPX2}, we have
$$
\left\|\left[(1+\delta\sqrt{\mathcal{H}}+\delta\langle D_{x}\rangle)^{-1},v\right]\left(1+\delta\sqrt{\mathcal{H}}+\delta\langle D_{x}\rangle\right)\partial_{x}f\right\|^{2}_{L^{2}_{x,v}}\leq4\|f\|^{2}_{L^{2}_{x,v}},
$$
which leads to
\begin{equation}\label{local-S4}
\begin{split}
\left|\left([(1+\delta\sqrt{\mathcal{H}}+\delta\langle D_{x}\rangle)^{-1},v](1+\delta\sqrt{\mathcal{H}}+\delta\langle D_{x}\rangle)\partial_{x}g_{\delta},g_{\delta}\right)_{L^{2}_{x,v}}\right|\leq2\|g_{\delta}\|^{2}_{L^{2}_{x,v}}.
\end{split}
\end{equation}
Consequently, we can deduce from \eqref{local-S1}, \eqref{local-S2}, \eqref{local-S3} and \eqref{local-S4} that for all $0\leq t\leq T, 0<\delta\leq1$
\begin{equation*}
\begin{split}
&\frac{1}{2}\frac{d}{dt}\Big(\|g_{\delta}\|^{2}_{L^{2}_{x,v}}\Big)+\frac{1}{C}\|\mathcal{H}^{\frac{s}{2}}g_{\delta}\|^{2}_{L^{2}_{x,v}}
\\&
\leq3\|g_{\delta}\|^{2}_{L^{2}_{x,v}}+C_{0}\|f_{N}\|_{L^{2}_{v}L^{\infty}_{x}}
(\|\mathcal{H}^{\frac{s}{2}}g_{N}\|_{L^{2}_{x,v}}+\|\mathcal{H}^{\frac{s}{2}}g_{\delta}\|_{L^{2}_{x,v}})
\|\mathcal{H}^{\frac{s}{2}}g_{\delta}\|_{L^{2}_{x,v}}.
\end{split}
\end{equation*}
Furthermore, if taking
\begin{align*}
\left\|f\right\|_{L^{\infty}([0,T]\times\mathbb{R}_{x};L^{2}_{v})}\leq\frac{1}{4CC_{0}}, \ \ \ \ T>0,
\end{align*} then we obtain
\begin{equation*}
\begin{split}
&\frac{d}{dt}\Big(\|g_{\delta}\|^{2}_{L^{2}_{x,v}}\Big)+\frac{1}{C}\|\mathcal{H}^{\frac{s}{2}}g_{\delta}\|^{2}_{L^{2}_{x,v}}
\leq6\|g_{\delta}\|^{2}_{L^{2}_{x,v}}+2C_{0}\|f_{N}\|_{L^{2}_{v}L^{\infty}_{x}}\|\mathcal{H}^{\frac{s}{2}}g_{N}\|^{2}_{L^{2}_{x,v}},
\end{split}
\end{equation*}
which leads to
\begin{equation*}
\begin{split}
&\|g_{\delta}\|^{2}_{L^{2}_{x,v}}+\frac{1}{C}\int_{0}^{t}e^{6(t-\tau)}\|\mathcal{H}^{\frac{s}{2}}g_{\delta}(\tau)\|^{2}_{L^{2}_{x,v}}d\tau
\\&
\leq e^{6t}\|g_{0}\|^{2}_{L^{2}_{x,v}}+2C_{0}\|f_{N}(\tau)\|_{L^{\infty}([0,T]\times\mathbb{R}_{x};L^{2}_{v})}\int_{0}^{t}e^{6(t-\tau)}
\|\mathcal{H}^{\frac{s}{2}}g_{N}(\tau)\|^{2}_{L^{2}_{x,v}}d\tau
\end{split}
\end{equation*}
for $0\leq t\leq T$ and $ 0<\delta\leq1$. Consequently, we obtain
\begin{equation*}
\begin{split}
&\|g_{\delta}\|^{2}_{L^\infty_TL^{2}_{x,v}}+\frac{1}{C}\|\mathcal{H}^{\frac{s}{2}}g_{\delta}(\tau)\|^{2}_{L^{2}([0,T]\times\mathbb{R}^{2}_{x,v})}
\\&
\leq e^{6T}\|g_{0}\|^{2}_{L^{2}_{x,v}}+2C_{0}e^{6T}\|f_{N}(\tau)\|_{L^{\infty}([0,T]\times\mathbb{R}_{x};L^{2}_{v})}
\|\mathcal{H}^{\frac{s}{2}}g_{N}(\tau)\|^{2}_{L^{2}([0,T]\times\mathbb{R}^{2}_{x,v})}.
\end{split}
\end{equation*}
On the other hand, noticing that
\begin{align*}
&\|g_{\delta}\|^{2}_{L^{2}_{x,v}}=\frac{1}{2\pi}\sum_{n=0}^{+\infty}\int_{\mathbb{R}}\Big(1+\delta\sqrt{n+\frac{1}{2}}+\delta\langle\xi\rangle\Big)^{-2}
|\mathcal{F}_{x}\bar{g}_N(t,\xi)|^2d\xi,
\\&
\|\mathcal{H}^{\frac{s}{2}}g_{\delta}(\tau)\|^{2}_{L^{2}([0,T]\times\mathbb{R}^{2}_{x,v})}
\\&
=\frac{1}{2\pi}\int_{0}^{T}\sum_{n=0}^{+\infty}(n+\frac{1}{2})^{s}
\int_{\mathbb{R}}\Big(1+\delta\sqrt{n+\frac{1}{2}}+\delta\langle\xi\rangle\Big)^{-2}|\mathcal{F}_{x}\bar{g}_N(t,\xi)|^2d\xi dt,
\end{align*}
with $\bar{g}_N=(g_{N}(t,x,\cdot),e_{n})_{L^{2}(\mathbb{R}_{v})}$, where $\mathcal{F}_{x}$ denotes the partial Fourier transform in the position variable,
it follows from the monotone convergence theorem (passing to the limit $\delta\rightarrow0_{+}$) that
\begin{align*}
&\|g_{N}\|^{2}_{L^\infty_TL^{2}_{x,v}}+\frac{1}{C}\|\mathcal{H}^{\frac{s}{2}}g_{N}(\tau)\|^{2}_{L^{2}([0,T]\times\mathbb{R}^{2}_{x,v})}
\\&
\leq e^{6T}\|g_{0}\|^{2}_{L^{2}_{x,v}}+2C_{0}e^{6T}\|f_{N}(\tau)\|_{L^{\infty}([0,T]\times\mathbb{R}_{x};L^{2}_{v})}
\|\mathcal{H}^{\frac{s}{2}}g_{N}(\tau)\|^{2}_{L^{2}([0,T]\times\mathbb{R}^{2}_{x,v})}.
\end{align*}
Thanks to the smallness of $\|f_{N}(\tau)\|_{L^{\infty}([0,T]\times\mathbb{R}_{x};L^{2}_{v})}$ (taking $\|f_{N}(\tau)\|_{L^{\infty}([0,T]\times\mathbb{R}_{x};L^{2}_{v})}\leq\frac{1}{4CC_{0}e^{6T}}$), we arrive at
\begin{align*}
\|g_{N}\|^{2}_{L^\infty_TL^{2}_{x,v}}+\|\mathcal{H}^{\frac{s}{2}}g_{N}(\tau)\|^{2}_{L^{2}([0,T]\times\mathbb{R}^{2}_{x,v})}\leq 2(C+1)e^{6T}\|g_{0}\|^{2}_{L^{2}_{x,v}}.
\end{align*}
Hence, the proof of Proposition \ref{local existence-gA} is finished.
\end{proof}

\begin{remark}
Owing to the embedding $\widetilde{L}^{\infty}_T\widetilde{L}^2_{v}(B^{1/2}_{2,1})\hookrightarrow L^{\infty}([0,T]\times\mathbb{R}_{x};L^{2}(\mathbb{R}_{v}))$, we deduce that the norm $\|f\|_{L^{\infty}([0,T]\times\mathbb{R}_{x};L^{2}_{v})}$ is small, since $\|f\|_{\widetilde{L}^{\infty}_T\widetilde{L}^2_{v}(B^{1/2}_{2,1})}$ is sufficiently small. However, there is no regularity available in position variable $x$ for the weak solution $g_N$ according to Proposition \ref{local existence-gA}. In that case, we cannot attain desired solutions presented by Theorem \ref{local-regularity}.
\end{remark}

\subsection{Regularity of weak solution in position variable}

In what follow, we establish the regularity of $g_N$ with respect to $x$.

\begin{lemma}\label{local-11}
Let $0<\sigma\leq1/2$ and $0<T<+\infty$. For $N\in\mathbb{N}$, setting $f_N=\sum^{N-1}_{p\geq-1}\Delta_{p}f$. If $g_N$ satisfies
$$
g_N\in L^{\infty}([0,T];L^2(\mathbb{R}^{2}_{x,v})), \quad\quad \mathcal{H}^{\frac{s}{2}}g_N\in L^{2}([0,T]\times\mathbb{R}^{2}_{x,v}),
$$
then there exists a $\tilde{C}>0$ independent of $N$ such that for any $\kappa>0$
\begin{equation}\label{local-13}
\begin{split}
&\sum_{p\geq-1}\frac{2^{p\sigma}}{1+\kappa2^{2p\sigma}}\left(\int_{0}^{T}\left|\left(\Delta_{p}\Gamma(f_N,g_N),\Delta_{p}g_N\right)_{x,v}\right|dt\right)^{1/2}
\\&\quad
\leq \widetilde{C}\|f_N\|^{1/2}_{\widetilde{L}^{\infty}_T\widetilde{L}^2_{v}(B^{1/2}_{2,1})}\|\mathcal{H}^{\frac{s}{2}}g_N\|
_{\widetilde{L}^{2}_T\widetilde{L}^2_{v}(B^{\sigma,\kappa}_{2,1})}
\\&\quad\quad
+C_N\|f_N\|^{1/2}_{\widetilde{L}^{\infty}_T\widetilde{L}^2_{v}(B^{1/2}_{2,1})}\|\mathcal{H}^{\frac{s}{2}}g_N\|_{L^{2}_TL^2_{v}L^{2}_{x}},
\end{split}
\end{equation}
where $C_{N}>0$ is a constant depending only on $N$ and
\begin{align*}
\|g_N\|_{\widetilde{L}^{2}_T\widetilde{L}^2_{v}(B^{\sigma,\kappa}_{2,1})}=\sum_{p\geq-1}\frac{2^{p\sigma}}{1+\kappa2^{2p\sigma}}\|\Delta_{p}g_N\|_{L^{2}_TL^2_{v}L^{2}_{x}}.
\end{align*}
\end{lemma}

\begin{proof}
For $\sigma, \kappa>0$, we have
\begin{align*}
&\|\mathcal{H}^{\frac{s}{2}}g_N\|_{\widetilde{L}^{2}_T\widetilde{L}^2_{v}(B^{\sigma,\kappa}_{2,1})}
=\sum_{p\geq-1}\frac{2^{p\sigma}}{1+\kappa2^{2p\sigma}}\|\Delta_{p}\mathcal{H}^{\frac{s}{2}}g_N\|_{L^{2}_TL^2_{v}L^{2}_{x}}
\\&\leq C_{\kappa}\sum_{p\geq-1}2^{-p\sigma}\|\mathcal{H}^{\frac{s}{2}}g_N\|_{L^{2}_TL^2_{v}L^{2}_{x}}
\leq C_{\kappa}\|\mathcal{H}^{\frac{s}{2}}g_N\|_{L^{2}_TL^2_{v}L^{2}_{x}},
\end{align*}
where $C_{\kappa}>0$ is a constant depending only on $\kappa$. Hence, one has $\|\mathcal{H}^{\frac{s}{2}}g_N\|_{\widetilde{L}^{2}_T\widetilde{L}^2_{v}(B^{\sigma,\kappa}_{2,1})}<+\infty$ due to
$\mathcal{H}^{\frac{s}{2}}g_N\in L^{2}([0,T]\times\mathbb{R}^{2}_{x,v})$.

By using Bony's decomposition, we divide the inner product into three parts:
\begin{align*}
\left(\Delta_{p}\Gamma(f_N,g_N),\Delta_{p}g_N\right)=\left(\Delta_{p}(\Gamma^{1}(f_N,g_N)+\Gamma^{2}(f_N,g_N)+ \Gamma^{3}(f_N,g_N)),\Delta_{p}g_N\right),
\end{align*}
where $\Gamma^{1}(f_N,g_N)\triangleq\sum_{j}\Gamma(S_{j-1}f_N,\Delta_{j}g_N), \Gamma^{2}(f_N,g_N)\triangleq\sum_{j}\Gamma(\Delta_{j}f_N,S_{j-1}g_N)$
and $\Gamma^{3}(f_N,g_N)\triangleq\sum_{j}\sum_{|j-j'|\leq1}\Gamma(\Delta_{j'}f_N,\Delta_{j}g_N)$.
For $\Gamma^{1}(f_N,g_N)$, note that
$$
\Delta_{p}\sum_{j}(S_{j-1}f_N\Delta_{j}g_N)=\Delta_{p}\sum_{|j-p|\leq4}(S_{j-1}f_N\Delta_{j}g_N).
$$
It follows from Remark \ref{Tri} that
\begin{align*}
&\sum_{p\geq-1}\frac{2^{p\sigma}}{1+\kappa2^{2p\sigma}}\left(\int_{0}^{T}\left|\left(\Delta_{p}\Gamma^{1}(f_N,g_N),\Delta_{p}g_N\right)_{x,v}\right|dt\right)^{1/2}
\\&
\lesssim\sum_{p\geq-1}\frac{2^{p\sigma}}{1+\kappa2^{2p\sigma}}\left(\sum_{|j-p|\leq4}\int_{0}^{T}\left\|S_{j-1}f_N\right\|_{L^{2}_{v}L^{\infty}_{x}}
\left\|\mathcal{H}^{\frac{s}{2}}\Delta_jg_N\right\|_{L^2_{x,v}}
\left\|\mathcal{H}^{\frac{s}{2}}\Delta_pg_N\right\|_{L^2_{x,v}}dt\right)^{1/2}
\\&
\lesssim\left\|f_N\right\|^{1/2}_{\widetilde{L}^{\infty}_T\widetilde{L}^2_{v}(B^{1/2}_{2,1})}
\|\mathcal{H}^{\frac{s}{2}}g_N\|_{\widetilde{L}^{2}_T\widetilde{L}^2_{v}(B^{\sigma,\kappa}_{2,1})}
\left(\sum_{p\geq-1}\sum_{|j-p|\leq4}\frac{1+\kappa2^{2j\sigma}}{1+\kappa2^{2p\sigma}}2^{(p-j)s}c(j)\right)^{1/2}
\\&
\lesssim\left\|f_N\right\|^{1/2}_{\widetilde{L}^{\infty}_T\widetilde{L}^2_{v}(B^{1/2}_{2,1})}
\|\mathcal{H}^{\frac{s}{2}}g_N\|_{\widetilde{L}^{2}_T\widetilde{L}^2_{v}(B^{\sigma,\kappa}_{2,1})},
\end{align*}
where we used Lemmas \ref{J7-3}, \ref{J7-1} and \ref{J7-5} in the third line and the following sequence $\{c(j)\}$
\begin{align*}
&c(j):=\frac{\frac{2^{j\sigma}}{1+\kappa2^{2j\sigma}}
\|\mathcal{H}^{\frac{s}{2}}\Delta_jg_N\|_{L^{2}_{T}L^2_{x,v}}}
{\|\mathcal{H}^{\frac{s}{2}}g_N\|_{\widetilde{L}^{2}_T\widetilde{L}^2_{v}(B^{\sigma,\kappa}_{2,1})}}
\end{align*}
satisfying $\|c(j)\|_{\ell^{1}}\leq1$.

For $\Gamma^{2}(f_N,g_N)$, similarly, we get
\begin{align*}
&\sum_{p\geq-1}\frac{2^{p\sigma}}{1+\kappa2^{2p\sigma}}\left(\int_{0}^{T}\left|\left(\Delta_{p}\Gamma^{2}(f_N,g_N),\Delta_{p}g_N\right)_{x,v}\right|dt\right)^{1/2}
\\&
\lesssim\sum_{p\geq-1}\frac{2^{p\sigma}}{1+\kappa2^{2p\sigma}}\left(\int_{0}^{T}\sum_{\substack{|j-p|\leq4\\ j\leq N}}
\|\Delta_jf_N\|_{L^{2}_{v}L^{\infty}_{x}}\left\|\mathcal{H}^{\frac{s}{2}}S_{j-1}g_N\right\|_{L^2_{x,v}}
\left\|\mathcal{H}^{\frac{s}{2}}\Delta_pg_N\right\|_{L^2_{x,v}}dt\right)^{1/2}
\\&
\leq C_N\|f_N\|^{1/2}_{\widetilde{L}^{\infty}_T\widetilde{L}^2_{v}(B^{1/2}_{2,1})}\|\mathcal{H}^{\frac{s}{2}}g_N\|_{L^{2}_TL^2_{v}L^{2}_{x}},
\end{align*}
where $C_N=CN2^{N\sigma}$ with $N\in\mathbb{N}$. Owing to
\begin{align*}
&\Delta_{p}\left(\sum_{j}\sum_{|j-j'|\leq1}(\Delta_{j'}f_N\Delta_{j}g_N)\right)=\Delta_{p}\left(\sum_{\max{j,j'}\geq p-2}\sum_{|j-j'|\leq1}(\Delta_{j'}f_N\Delta_{j}g_N)\right)
\\&\qquad=0, \ \ \text{if} \ \ p\geq N+3,
\end{align*}
then $\Gamma^{3}(f_N,g_N)$ can be estimated as follows:
\begin{align*}
&\sum_{p\geq-1}\frac{2^{p\sigma}}{1+\kappa2^{2p\sigma}}\left(\int_{0}^{T}\left|\left(\Delta_{p}  \Gamma^{3}(f_N,g_N),\Delta_{p}g_N\right)_{x,v}\right|dt\right)^{1/2}
\\&
=\sum_{p\geq-1}^{N+2}\frac{2^{p\sigma}}{1+\kappa2^{2p\sigma}}\left(\int_{0}^{T}\left|(\Delta_{p}\Gamma^{3}(f_N,g_N),\Delta_{p}g_N)_{x,v}\right|dt\right)^{1/2}
\\&
\lesssim\sum_{p\geq-1}^{N+2}\sum_{j\leq N+1}\frac{2^{p\sigma}}{1+\kappa2^{2p\sigma}}
\left(\int_{0}^{T}\left\|\Delta_{j'}f_N\right\|_{L^{2}_{v}L^{\infty}_{x}}\left\|\mathcal{H}^{\frac{s}{2}}\Delta_jg_N\right\|_{L^2_{x,v}}
\left\|\mathcal{H}^{\frac{s}{2}}\Delta_{p}g_N\right\|_{L^2_{x,v}}dt\right)^{1/2}
\\&
\lesssim C_N\|f_N\|^{1/2}_{\widetilde{L}^{\infty}_T\widetilde{L}^2_{v}(B^{1/2}_{2,1})}
\|\mathcal{H}^{\frac{s}{2}}g_N\|_{L^{2}_TL^2_{v}L^{2}_{x}}.
\end{align*}
Together the above three inequalities, we can get \eqref{local-13}.
\end{proof}

Based on Proposition \ref{local existence-gA} and Lemma \ref{local-11}, we obtain the regularity of the weak solution $g_N$ to \eqref{local-111A}.

\begin{proposition}\label{local-regularity-g}
There exists a constant $\varepsilon_{0}>0$ such that for all $T>0, g_{0}\in\widetilde{L}^2_{v}(B^{1/2}_{2,1}), f\in\widetilde{L}^\infty_{T}\widetilde{L}^2_{v}(B^{1/2}_{2,1})$ fulfilling
$$
\|f\|_{\widetilde{L}^{\infty}_T\widetilde{L}^2_{v}(B^{1/2}_{2,1})}\leq\varepsilon_{0},
$$
then \eqref{local-111A}
admits a weak solution $g_N\in L^{\infty}([0,T]; L^{2}(\mathbb{R}^{2}_{x,v}))$ satisfying
\begin{equation}\label{local-2A}
\begin{split}
&\|g_N\|_{\widetilde{L}^{\infty}_T\widetilde{L}^2_{v}(B^{1/2}_{2,1})}+\frac{1}{\sqrt{2C}}
\|\mathcal{H}^{\frac{s}{2}}g_N\|_{\widetilde{L}^{2}_T\widetilde{L}^2_{v}(B^{1/2}_{2,1})}
\\&
\leq e^{T}\|g_{0}\|_{\widetilde{L}^2_{v}(B^{1/2}_{2,1})}+\sqrt{2}C_Ne^{T}
\|f_N\|^{1/2}_{\widetilde{L}^{\infty}_T\widetilde{L}^2_{v}(B^{1/2}_{2,1})}\|\mathcal{H}^{\frac{s}{2}}g_N\|_{L^{2}_TL^2_{v}L^{2}_{x}},
\end{split}
\end{equation}
where $C_N>0$ is some constant depending on $N$.
\end{proposition}

\begin{proof}
Applying $\Delta_{p}(p\geq-1)$ to \eqref{local-111A}, and then taking the inner product with $\Delta_{p}g_N$ over $\mathbb{R}_{x}\times \mathbb{R}_{v}$ gives
\begin{align*}
&\frac{1}{2}\frac{d}{dt}\left(\|\Delta_{p}g_N\|^{2}_{L^{2}_{x,v}}\right)+\frac{1}{C}\|\mathcal{H}^{\frac{s}{2}}\Delta_{p}g_N\|^{2}_{L^{2}_{x,v}}
\leq \|\Delta_{p}g_N\|^{2}_{L^{2}_{x,v}}+\left(\Delta_{p}\Gamma(f_N,g_N),\Delta_{p}g_N\right)_{L^{2}_{x,v}},
\end{align*}
where we used the coercivity estimate of $\mathcal{K}$. It follows that
\begin{align*}
\frac{d}{dt}\left(e^{-2t}\|\Delta_{p}g_N\|^{2}_{L^{2}_{x,v}}\right)+\frac{2}{C}e^{-2t}\|\mathcal{H}^{\frac{s}{2}}\Delta_{p}g_N\|^{2}_{L^{2}_{x,v}}
\leq2e^{-2t}|\left(\Delta_{p}\Gamma(f_N,g_N),\Delta_{p}g_N\right)_{L^{2}_{x,v}}|
\end{align*}
for $0\leq t\leq T$.

Integrating the above inequality with respect to the time variable over $[0,t]$ with $0\leq t\leq T$ and taking the square root of both sides of the resulting inequality, we get
\begin{align*}
&\|\Delta_{p}g_N\|_{L^{2}_{x,v}}+\sqrt{\frac{2}{C}}\Big(\int_{0}^te^{2(t-\tau)}\|\mathcal{H}^{\frac{s}{2}}\Delta_{p}g_N(\tau)\|^{2}_{L^{2}_{x,v}}d\tau\Big)^{1/2}
\\&\qquad\leq e^{t}\left\|\Delta_{p}g_0\right\|_{L^2_{v}L^2_{x}}
+\sqrt{2}\Big(\int_{0}^te^{2(t-\tau)}|\left(\Delta_{p}\Gamma(f_N,g_N),\Delta_{p}g_N\right)_{L^{2}_{x,v}}|d\tau\Big)^{1/2},
\end{align*}
then, taking supremum over $0\leq t\leq T$ on the left side and multiplying the resulting inequality by $\frac{2^{p/2}}{1+\kappa2^{p}}$, we obtain
\begin{align*}
&\frac{2^{p/2}}{1+\kappa2^{p}}\left\|\Delta_{p}g_N\right\|_{L^2_{v}L^2_{x}}
+\sqrt{\frac{2}{C}}\frac{2^{p/2}}{1+\kappa2^{p}}\left(\int_{0}^{t}\left\|\mathcal{H}^{\frac{s}{2}}\Delta_{p}g_N\right\|^{2}_{L^2_{v}L^2_{x}}dt\right)^{1/2}
\\&
\leq e^{T}\frac{2^{p/2}}{1+\kappa2^{p}}\|\Delta_{p}g_{0}\|_{L^2_{v}L^2_{x}}
+\sqrt{2}e^{T}\frac{2^{p/2}}{1+\kappa2^{p}}\left(\int_{0}^{T}\left|\left(\Delta_{p}\Gamma(f_N,g_N),\Delta_{p}g_N\right)_{L^{2}_{x,v}}\right|dt\right)^{1/2}.
\end{align*}
Further taking the summation over $p\geq-1$, the above inequality implies
\begin{align*}
&\|g_N\|_{\widetilde{L}^{\infty}_T\widetilde{L}^2_{v}(B^{1/2,\kappa}_{2,1})}+\sqrt{\frac{2}{C}}
\|\mathcal{H}^{\frac{s}{2}}g_N\|_{\widetilde{L}^{2}_T\widetilde{L}^2_{v}(B^{1/2,\kappa}_{2,1})}
\\&
\leq e^{T}\|g_{0}\|_{\widetilde{L}^2_{v}(B^{1/2}_{2,1})}+\sqrt{2}\widetilde{C}e^{T}\|f_N\|^{1/2}_{\widetilde{L}^{\infty}_T\widetilde{L}^2_{v}(B^{1/2}_{2,1})}
\|\mathcal{H}^{\frac{s}{2}}g_N\|_{\widetilde{L}^{2}_T\widetilde{L}^2_{v}(B^{1/2,\kappa}_{2,1})}
\\&\quad
+\sqrt{2}C_Ne^{T}\|f_N\|^{1/2}_{\widetilde{L}^{\infty}_T\widetilde{L}^2_{v}(B^{1/2}_{2,1})}\|\mathcal{H}^{\frac{s}{2}}g_N\|_{L^{2}_TL^2_{v}L^{2}_{x}},
\end{align*}
where we used the Proposition \ref{Cauchy-f} and Lemma \ref{local-11}.
Then, by taking $\|f\|_{\widetilde{L}^{\infty}_T\widetilde{L}^2_{v}(B^{1/2}_{2,1})}\leq\frac{1}{4e^{2T}C\widetilde{C}^{2}}$ and letting $\kappa\rightarrow0$, we obtain
\begin{align*}
&\|g_N\|_{\widetilde{L}^{\infty}_T\widetilde{L}^2_{v}(B^{1/2}_{2,1})}+\frac{1}{\sqrt{2C}}
\|\mathcal{H}^{\frac{s}{2}}g_N\|_{\widetilde{L}^{2}_T\widetilde{L}^2_{v}(B^{1/2}_{2,1})}
\\&
\leq e^{T}\|g_{0}\|_{\widetilde{L}^2_{v}(B^{1/2}_{2,1})}+\sqrt{2}C_Ne^{T}
\|f_N\|^{1/2}_{\widetilde{L}^{\infty}_T\widetilde{L}^2_{v}(B^{1/2}_{2,1})}\|\mathcal{H}^{\frac{s}{2}}g_N\|_{L^{2}_TL^2_{v}L^{2}_{x}},
\end{align*}
which ends the proof of Proposition \ref{local-regularity-g}.
\end{proof}

\subsection{Energy estimates in Besov space}

It follows from \eqref{local-2A} in Proposition \ref{local-regularity-g} that
$$
\|g_N\|_{\widetilde{L}^{\infty}_{T}\widetilde{L}^{2}_{v}(B^{1/2}_{2,1})}+\|\mathcal{H}^{\frac{s}{2}}g_N\|
_{\widetilde{L}^{2}_T\widetilde{L}^2_{v}(B^{1/2}_{2,1})}<+\infty.
$$
Then applying the Corollary \ref{trilinear-G-R} to $f_N$ and $ g_N$, we get the following inequality
\begin{equation}\label{local-13B}
\begin{split}
\sum_{p\geq-1}2^{\frac{p}{2}}\left(\int_{0}^{T}\left|\left(\Delta_{p}\Gamma(f_N,g_N),\Delta_{p}g_N\right)_{x,v}\right|dt\right)^{1/2}
\leq C_{1}\|f_N\|^{1/2}_{\widetilde{L}^{\infty}_T\widetilde{L}^2_{v}(B^{1/2}_{2,1})}
\|\mathcal{H}^{\frac{s}{2}}g_N\|_{\widetilde{L}^{2}_T\widetilde{L}^2_{v}(B^{1/2}_{2,1})},
\end{split}
\end{equation}
for some constant $C_{1}>0$ independent of $N$.

With aid of \eqref{local-13B}, one can obtain the further energy estimate, which is independent of $N$ for the weak solution $g_N$.

\begin{proposition}\label{local-regularity-gB}
There exists a constant $\varepsilon_{0}>0$ such that for $T>0, g_{0} \in\widetilde{L}^2_{v}(B^{1/2}_{2,1}), f\in\widetilde{L}^\infty_{T}\widetilde{L}^2_{v}(B^{1/2}_{2,1})$ fulfilling
$$
\|f\|_{\widetilde{L}^{\infty}_T\widetilde{L}^2_{v}(B^{1/2}_{2,1})}\leq\varepsilon_{0},
$$
then \eqref{local-111A} admits a weak solution $g_N\in L^{\infty}([0,T]; L^{2}(\mathbb{R}^{2}_{x,v}))$ satisfying
\begin{equation}\label{local-2A-B}
\begin{split}
&\|g_N\|_{\widetilde{L}^{\infty}_T\widetilde{L}^2_{v}(B^{1/2}_{2,1})}+\frac{1}{\sqrt{2C}}
\|\mathcal{H}^{\frac{s}{2}}g_N\|_{\widetilde{L}^{2}_T\widetilde{L}^2_{v}(B^{1/2}_{2,1})}
\leq e^{T}\|g_{0}\|_{\widetilde{L}^2_{v}(B^{1/2}_{2,1})},
\end{split}
\end{equation}
where $C>0$ is some constant independent of $N$.
\end{proposition}

\begin{proof}
Applying $2^{p}\Delta_{p}(p\geq-1)$ to \eqref{local-111A} and taking the inner product with $\Delta_{p}g_N$ over $\mathbb{R}_{x}\times \mathbb{R}_{v}$ give
\begin{align*}
&\frac{1}{2}\frac{d}{dt}\left(2^{p}\|\Delta_{p}g_N\|^{2}_{L^{2}_{x,v}}\right)+\frac{1}{C}2^{p}\|\mathcal{H}^{\frac{s}{2}}\Delta_{p}g_N\|^{2}_{L^{2}_{x,v}}
\\&\qquad
\leq 2^{p}\|\Delta_{p}g_N\|^{2}_{L^{2}_{x,v}}+2^{p}\left(\Delta_{p}\Gamma(f_N,g_N),\Delta_{p}g_N\right)_{L^{2}_{x,v}}.
\end{align*}
It follows that
\begin{align*}
\frac{d}{dt}\left(e^{-2t}2^{p}\|\Delta_{p}g_N\|^{2}_{L^{2}_{x,v}}\right)+\frac{2}{C}2^{p}e^{-2t}\|\mathcal{H}^{\frac{s}{2}}\Delta_{p}g_N\|^{2}_{L^{2}_{x,v}}
\leq2e^{-2t}2^{p}|\left(\Delta_{p}\Gamma(f_N,g_N),\Delta_{p}g_N\right)_{L^{2}_{x,v}}|
\end{align*}
for all $0\leq t\leq T$.

Integrating the above inequality with respect to the time variable over $[0,t]$ with $0\leq t\leq T$ and taking the square root, we obtain
\begin{align*}
&2^{\frac{p}{2}}\|\Delta_{p}g_N\|_{L^{2}_{x,v}}+\sqrt{\frac{2}{C}}2^{\frac{p}{2}}\Big(\int_{0}^te^{2(t-\tau)}
\|\mathcal{H}^{\frac{s}{2}}\Delta_{p}g_N(\tau)\|^{2}_{L^{2}_{x,v}}d\tau\Big)^{1/2}
\\&\qquad
\leq e^{t}2^{\frac{p}{2}}\left\|\Delta_{p}g_0\right\|_{L^2_{v}L^2_{x}}
+\sqrt{2}2^{\frac{p}{2}}\Big(\int_{0}^te^{2(t-\tau)}|\left(\Delta_{p}\Gamma(f_N,g_N),\Delta_{p}g_N\right)_{L^{2}_{x,v}}|d\tau\Big)^{1/2}.
\end{align*}
Taking supremum over $0\leq t\leq T$ on the left side and summing up over $p\geq-1$, we get
\begin{align*}
&\|g_N\|_{\widetilde{L}^{\infty}_T\widetilde{L}^2_{v}(B^{1/2}_{2,1})}+\sqrt{\frac{2}{C}}\|\mathcal{H}^{\frac{s}{2}}g_N\|_{\widetilde{L}^{2}_T\widetilde{L}^2_{v}(B^{1/2}_{2,1})}
\\&
\leq e^{T}\|g_{0}\|_{\widetilde{L}^2_{v}(B^{1/2}_{2,1})}
+\sqrt{2}e^{T}\sum_{p\geq-1}2^{\frac{p}{2}}\left(\int_{0}^{T}\left|\left(\Delta_{p}\Gamma(f_N,g_N),\Delta_{p}g_N\right)_{L^{2}_{x,v}}\right|dt\right)^{1/2}
\\&
\leq e^{T}\|g_{0}\|_{\widetilde{L}^2_{v}(B^{1/2}_{2,1})}
+\sqrt{2}e^{T}C_{1}\|f_N\|^{1/2}_{\widetilde{L}^{\infty}_T\widetilde{L}^2_{v}(B^{1/2}_{2,1})}
\|\mathcal{H}^{\frac{s}{2}}g_N\|_{\widetilde{L}^{2}_T\widetilde{L}^2_{v}(B^{1/2}_{2,1})},
\end{align*}
where we used Proposition \ref{Cauchy-f} and \eqref{local-13B}. It follows from the smallness of $\|f\|_{\widetilde{L}^{\infty}_T\widetilde{L}^2_{v}(B^{1/2}_{2,1})}$ (taking $\|f\|_{\widetilde{L}^{\infty}_T\widetilde{L}^2_{v}(B^{1/2}_{2,1})}\leq\frac{1}{4e^{2T}CC_{1}^{2}}$) that
\begin{align*}
&\|g_N\|_{\widetilde{L}^{\infty}_T\widetilde{L}^2_{v}(B^{1/2}_{2,1})}+\frac{1}{\sqrt{2C}}
\|\mathcal{H}^{\frac{s}{2}}g_N\|_{\widetilde{L}^{2}_T\widetilde{L}^2_{v}(B^{1/2}_{2,1})}
\leq e^{T}\|g_{0}\|_{\widetilde{L}^2_{v}(B^{1/2}_{2,1})},
\end{align*}
which indicates the desired inequality \eqref{local-2A-B}. The proof of Proposition \ref{local-regularity-gB} is completed.
\end{proof}

In the following, we prove Theorem \ref{local-regularity} with the help of Proposition \ref{local-regularity-gB}.  \\
{\bf The proof of the Theorem \ref{local-regularity}}
It suffices to show that the sequence $\{g_N,N\in\mathbb{N}\}$ is Cauchy in the space
$$X=\{g\in \widetilde{L}^{\infty}_{T}\widetilde{L}^{2}_{v}(B^{1/2}_{2,1})|\mathcal{H}^{\frac{s}{2}}g \in \widetilde{L}^{2}_{T}\widetilde{L}^{2}_{v}(B^{1/2}_{2,1})\}.$$

Set $w_{M,M'}=g_M-g_{M'}$ for $ M,M'\in\mathbb{N}$. Then it follows that \eqref{local-111A} that
\begin{align*}
\partial_{t}w_{M,M'}+v\partial_{x}w_{M,M'}+\mathcal{K}w_{M,M'}
=\Gamma(f_{M'},w_{M,M'})+\Gamma(f_M-f_{M'},g_N).
\end{align*}
Following from the proof procedure of Proposition \ref{local-regularity-gB}, we can obtain
\begin{equation*}
\begin{split}
&\left\|w_{M,M'}\right\|_{\widetilde{L}^{\infty}_{T}\widetilde{L}^{2}_{v}(B^{1/2}_{2,1})}+\sqrt{\frac{2}{C}}\left\|\mathcal{H}^{\frac{s}{2}}w_{M,M'}\right\|
_{\widetilde{L}^{2}_{T}\widetilde{L}^{2}_{v}(B^{1/2}_{2,1})}
\\&
\leq\sqrt{2}e^{T}\sum_{p\geq-1}2^{\frac{p}{2}}\Big(\int_{0}^{T}\Big|\left(\Delta_{p}\Gamma(f_{M'},w_{M,M'}),\Delta_{p}w_{M,M'}\right)_{L^{2}_{x,v}}\Big|dt\Big)^{1/2}
\\&
+\sqrt{2}e^{T}\sum_{p\geq-1}2^{\frac{p}{2}}\Big(\int_{0}^{T}\Big|\left(\Delta_{p}\Gamma(f_M-f_{M'},g_M),\Delta_{p}w_{M,M'}\right)_{L^{2}_{x,v}}\Big|dt\Big)^{1/2}
\\&
\leq\sqrt{2}C_{1}e^{T}\|f_{M'}\|^{1/2}_{\widetilde{L}^{\infty}_{T}\widetilde{L}^{2}_{v}(B^{1/2}_{2,1})}
\|\mathcal{H}^{\frac{s}{2}}w_{M,M'}\|_{\widetilde{L}^{2}_{T}\widetilde{L}^{2}_{v}(B^{1/2}_{2,1})}
\\&\quad
+\sqrt{2}C_{1}e^{T}\|(f_M-f_{M'})\|^{1/2}_{\widetilde{L}^{\infty}_{T}\widetilde{L}^{2}_{v}(B^{1/2}_{2,1})}
\|\mathcal{H}^{\frac{s}{2}}g_M\|^{1/2}_{\widetilde{L}^{2}_{T}\widetilde{L}^{2}_{v}(B^{1/2}_{2,1})}
\|\mathcal{H}^{\frac{s}{2}}w_{M,M'}\|^{1/2}_{\widetilde{L}^{2}_{T}\widetilde{L}^{2}_{v}(B^{1/2}_{2,1})}
\\&
\leq\sqrt{2C}C_{1}^{2}e^{2T}\|(f_M-f_{M'})\|_{\widetilde{L}^{\infty}_{T}\widetilde{L}^{2}_{v}(B^{1/2}_{2,1})}
\|\mathcal{H}^{\frac{s}{2}}g_M\|_{\widetilde{L}^{2}_{T}\widetilde{L}^{2}_{v}(B^{1/2}_{2,1})}
\\&\quad
+\frac{3}{4}\sqrt{\frac{2}{C}}\|\mathcal{H}^{\frac{s}{2}}w_{M,M'}\|
_{\widetilde{L}^{2}_{T}\widetilde{L}^{2}_{v}(B^{1/2}_{2,1})}.
\end{split}
\end{equation*}
The smallness of $\|g_{0}\|_{\widetilde{L}^2_{v}(B^{1/2}_{2,1})}$ (taking $\|g_{0}\|_{\widetilde{L}^2_{v}(B^{1/2}_{2,1})}\leq\frac{\tilde{\lambda}}{2CC_{1}^{2}e^{3T}}$) and Proposition \ref{local-regularity-gB} enables us to obtain
\begin{align*}
&\|w_{M,M'}\|_{\widetilde{L}^{\infty}_{T}\widetilde{L}^{2}_{v}(B^{1/2}_{2,1})}+\frac{1}{\sqrt{8C}}\|\mathcal{H}^{\frac{s}{2}}w_{M,M'}\|
_{\widetilde{L}^{2}_{T}\widetilde{L}^{2}_{v}(B^{1/2}_{2,1})}
\\&
\leq2CC_{1}^{2}e^{3T}\|g_{0}\|_{\widetilde{L}^2_{v}(B^{1/2}_{2,1})}\|(f_M-f_{M'})\|_{\widetilde{L}^{\infty}_{T}\widetilde{L}^{2}_{v}(B^{1/2}_{2,1})}
\\&
\leq\tilde{\lambda}\|(f_M-f_{M'})\|_{\widetilde{L}^{\infty}_{T}\widetilde{L}^{2}_{v}(B^{1/2}_{2,1})}
\end{align*}
for $0<\tilde{\lambda}<1$. It follows from  Proposition \ref{Cauchy-f} that $\{f_N\}$ is a Cauchy sequence in $\widetilde{L}^{\infty}_{T}\widetilde{L}^{2}_{v}(B^{1/2}_{2,1})$ which implies that $\{g_N\}$ is a Cauchy sequence in $X$.
Letting $g=\lim_{\substack{N\rightarrow\infty}}g_N$, we can get the desired result.

\section{Gelfand-Shilov and Gevrey regularizing effect}\label{S4}
In this section, we prove that the Cauchy problem \eqref{eq-1} enjoys the Gelfand-Shilov regularizing properties with respect to the velocity variable $v$ and Gevrey regularizing properties with respect to the position variable $x$.

\subsection{A priori estimates with exponential weights}
Firstly, it is shown that the sequence of approximate solutions $(\tilde{g}_n)_{n\geq0}$ (which defined by \eqref{equationA}) satisfies a priori estimate with exponential weights for sufficiently small initial data.

\begin{proposition}\label{Gelfand}
Let $T>0$. There exist some positive constants $C, \varepsilon_1>0, 0<c_0\leq1$ such that for all initial data $\|g_0\|_{\widetilde{L}^2_{v}(B^{1/2}_{2,1})}\leq\varepsilon_1$, the sequence of approximate solutions $(\tilde{g}_n)_{n\geq0}$ satisfies
\begin{equation}\label{G-S1}
\begin{split}
&\|G_{\kappa}(c t)\tilde{g}_n\|_{\widetilde{L}^{\infty}_T\widetilde{L}^2_{v}(B^{1/2}_{2,1})}
+\|\mathcal{H}^{\frac{s}{2}}G_{\kappa}(c t)\tilde{g}_n\|_{\widetilde{L}^{2}_T\widetilde{L}^2_{v}(B^{1/2}_{2,1})}
\\&
+\|\langle D_x\rangle^{\frac{s}{2s+1}}G_{\kappa}(c t)\tilde{g}_n\|_{\widetilde{L}^{2}_T\widetilde{L}^2_{v}(B^{1/2}_{2,1})}
\leq Ce^{CT}\|g_{0}\|_{\widetilde{L}^2_{v}(B^{1/2}_{2,1})},
\end{split}
\end{equation}
for $0<\kappa\leq1, 0<c\leq c_0, n\geq1$, where
$$
G_{\kappa}(t)=\frac{\exp(t(\mathcal{H}^{\frac{s+1}{2}}+\langle D_x\rangle^{\frac{3s+1}{2s+1}})^{\frac{2s}{3s+1}})}
{1+\kappa\exp(t(\mathcal{H}^{\frac{s+1}{2}}+\langle D_x\rangle^{\frac{3s+1}{2s+1}})^{\frac{2s}{3s+1}})}.
$$
\end{proposition}
To prove Proposition \ref{Gelfand}, we need some lemmas.
\begin{lemma}\label{HD}
There exists a constant $c_5>0$ such that for all $f\in\mathcal{S}(\mathbb{R}^2_{x,v})$,
\begin{align*}
\|(\mathcal{H}^{\frac{s+1}{2}}+\langle D_x\rangle^{\frac{3s+1}{2s+1}})^{\frac{s}{3s+1}}f\|_{L^2_{x,v}}\leq c_5\|\mathcal{H}^{\frac{(1+s)s}{2(3s+1)}}f\|_{L^2_{x,v}}
+c_5\|\langle D_x\rangle^{\frac{s}{2s+1}}f\|_{L^2_{x,v}}.
\end{align*}
\end{lemma}

\begin{proof}
We decompose $f$ into the Hermite basis in the velocity variable
\begin{align}\label{11}
f(x,v)=\sum_{n=0}^{+\infty}f_n(x)e_n(v), \ \text{with} \  f_n(x)=(f(x,\cdot),e_n)_{L^2(\mathbb{R}_v)}.
\end{align}
Since
$$
\forall0<\alpha<1, \ \ \forall a,b\geq0, \ \ \ (a+b)^{\alpha}\leq a^{\alpha}+b^{\alpha},
$$
one can verify that
\begin{align*}
\|(\mathcal{H}^{\frac{s+1}{2}}+\langle D_x\rangle^{\frac{3s+1}{2s+1}})^{\frac{s}{3s+1}}f\|_{L^2_{x,v}}&= \Big(\frac{1}{2\pi}\sum_{n=0}^{+\infty}\int_{\mathbb{R}}\Big((n+\frac{1}{2})^{\frac{s+1}{2}}+\langle\xi\rangle^{\frac{3s+1}{2s+1}}\Big)
^{\frac{2s}{3s+1}}|\widehat{f_n}(\xi)|^2d\xi
\Big)^{1/2}
\\&
\leq\Big(\frac{1}{2\pi}\sum_{n=0}^{+\infty}\int_{\mathbb{R}}\Big[(n+\frac{1}{2})^{\frac{s(1+s)}{3s+1}}+\langle\xi\rangle^{\frac{2s}{2s+1}}\Big]
|\widehat{f_n}(\xi)|^2d\xi\Big)^{1/2}
\\&
=\Big(\|\mathcal{H}^{\frac{s(s+1)}{2(3s+1)}}f\|^2_{L^2_{x,v}}+\|\langle D_x\rangle^{\frac{s}{2s+1}}f\|^2_{L^2_{x,v}}\Big)^{1/2}.
\end{align*}
\end{proof}

\begin{remark}
Since the indices
$$\frac{(1+s)s}{2(3s+1)}<\frac{s}{2},$$
we always use the following result
\begin{align*}
\|(\mathcal{H}^{\frac{s+1}{2}}+\langle D_x\rangle^{\frac{3s+1}{2s+1}})^{\frac{s}{3s+1}}f\|_{L^2_{x,v}}\leq c_5\|\mathcal{H}^{\frac{s}{2}}f\|_{L^2_{x,v}}
+c_5\|\langle D_x\rangle^{\frac{s}{2s+1}}f\|_{L^2_{x,v}}.
\end{align*}
\end{remark}

\begin{lemma}\label{HD-1}
For all $0\leq \alpha\leq1$, there exists a constant $\tilde{c}_{\alpha}>0$ such that for all $f\in\mathcal{S}(\mathbb{R}^2_{x,v})$,
\begin{align*}
\|\mathcal{H}^{\alpha}Q\Delta_pf\|_{L^2_{x,v}}\leq\tilde{c}_{\alpha}\|\mathcal{H}^{\alpha}\Delta_pf\|_{L^2_{x,v}}.
\end{align*}
\end{lemma}

\begin{proof}
We can deduce from \eqref{PH} that
\begin{equation}\label{QQ}
\begin{split}
\|\mathcal{H}^{\alpha}Q\Delta_pf\|_{L^2_{x,v}}&\lesssim\Big\|\mathrm{Op}^w\Big(\Big(1+\eta^2+\frac{v^2}{4}\Big)^\alpha\Big)Q\Delta_pf\Big\|_{L^2_{x,v}}
\\&
\lesssim\Big\|Q\mathrm{Op}^w\Big(\Big(1+\eta^2+\frac{v^2}{4}\Big)^\alpha\Big)\Delta_pf\Big\|_{L^2_{x,v}}
\\&\quad
+\Big\|\Big[\mathrm{Op}^w\Big(\Big(1+\eta^2+\frac{v^2}{4}\Big)^\alpha\Big),Q\Big]\Delta_pf\Big\|_{L^2_{x,v}},
\end{split}
\end{equation}
since $[Q,\Delta_p]=0$. Due to the fact that the multiplier is a bounded operator on $L^2(\mathbb{R}^2_{x,v})$ and \eqref{PH}, we can obtain
\begin{equation}\label{QQ-1}
\begin{split}
\Big\|Q\mathrm{Op}^w\Big(\Big(1+\eta^2+\frac{v^2}{4}\Big)^\alpha\Big)\Delta_pf\Big\|_{L^2_{x,v}}
&\lesssim\Big\|\mathrm{Op}^w\Big(\Big(1+\eta^2+\frac{v^2}{4}\Big)^\alpha\Big)\Delta_pf\Big\|_{L^2_{x,v}}
\\&
\lesssim\|\mathcal{H}^{\alpha}\Delta_pf\|_{L^2_{x,v}}.
\end{split}
\end{equation}
On the other hand, we deduce from \eqref{Q} and  (3.7), (3.8), Lemma 3.2 in \cite{LMPX2} that
$$
\Big[\mathrm{Op}^w\Big(\Big(1+\eta^2+\frac{v^2}{4}\Big)^\alpha\Big),Q\Big]\in\mathrm{Op}^w\Big(S(\langle(v,\eta)\rangle^{2\alpha-2},\Gamma_0)\Big)
\subset\mathrm{Op}^w(S(1,\Gamma_0)),
$$
uniformly with respect to the parameter $\xi\in\mathbb{R}$ because $0\leq \alpha\leq1$. It implies that
\begin{equation}\label{QQ-2}
\begin{split}
\left\|\Big[\mathrm{Op}^w\Big(\Big(1+\eta^2+\frac{v^2}{4}\Big)^\alpha\Big),Q\Big]\Delta_pf\right\|_{L^2_{x,v}}\lesssim\|\Delta_pf\|_{L^2_{x,v}}.
\end{split}
\end{equation}
Combining \eqref{QQ}, \eqref{QQ-1} and \eqref{QQ-2} gives
$$
\|\mathcal{H}^{\alpha}Q\Delta_pf\|_{L^2_{x,v}}\leq\tilde{c}_{\alpha}\|\mathcal{H}^{\alpha}\Delta_pf\|_{L^2_{x,v}}.
$$
\end{proof}

\begin{lemma}\label{HD-2}
There exists a constant $\tilde{c}_{1}>0$ such that for all $f\in\mathcal{S}(\mathbb{R}^2_{x,v})$, $0\leq c\leq1,0<\kappa\leq1,t\geq0$,
\begin{align*}
\left\|\mathcal{H}^{-\frac{s}{2}}\left[G_{\kappa}(c t)v(G_{\kappa}(c t))^{-1}-v\right]\Delta_p\partial_{x}f\right\|_{L^2_{x,v}}
\leq\tilde{c}_{1}c te^{\tilde{c}_{1}c t}\|\langle D_x\rangle^{\frac{s}{2s+1}}\Delta_pf\|_{L^2_{x,v}}.
\end{align*}
\end{lemma}

\begin{proof}
For all $f\in\mathcal{S}(\mathbb{R}^2_{x,v})$, and the decomposition \eqref{11}, by using the identities \eqref{H3}-\eqref{H4} satisfied by the creation and annihilation operators
\begin{align*}
&A_+e_n=(\frac{v}{2}-\partial_v)e_n=\sqrt{n+1}e_{n+1},
\\&
A_-e_n=(\frac{v}{2}+\partial_v)e_n=\sqrt{n}e_{n-1},
\end{align*}
we have immediately,
\begin{align*}
& ve_n=A_+e_n+A_-e_n=\sqrt{n+1}e_{n+1}+\sqrt{n}e_{n-1},
\\&
\mathcal{H}=\frac{1}{2}\left(A_+A_-+A_-A_+\right)e_n=\left(n+\frac{1}{2}\right)e_n.
\end{align*}
It follows that
\begin{equation*}
\begin{split}
&\frac{\exp(c t(\mathcal{H}^{\frac{1+s}{2}}+\langle\xi\rangle^{\frac{3s+1}{2s+1}})^{\frac{2s}{3s+1}})}{1+\kappa\exp(c t(\mathcal{H}^{\frac{1+s}{2}}+\langle\xi\rangle^{\frac{3s+1}{2s+1}})^{\frac{2s}{3s+1}})}
v\frac{1+\kappa\exp(c t(\mathcal{H}^{\frac{1+s}{2}}+\langle\xi\rangle^{\frac{3s+1}{2s+1}})^{\frac{2s}{3s+1}})}{\exp(c t(\mathcal{H}^{\frac{1+s}{2}}+\langle\xi\rangle^{\frac{3s+1}{2s+1}})^{\frac{2s}{3s+1}})}e_n
\\&
=\frac{\exp(c t((n+\frac{3}{2})^{\frac{1+s}{2}}+\langle\xi\rangle^{\frac{3s+1}{2s+1}})^{\frac{2s}{3s+1}})}{1+\kappa\exp(c t((n+\frac{3}{2})^{\frac{1+s}{2}}+\langle\xi\rangle^{\frac{3s+1}{2s+1}})^{\frac{2s}{3s+1}})}
\frac{1+\kappa\exp(c t((n+\frac{1}{2})^{\frac{1+s}{2}}+\langle\xi\rangle^{\frac{3s+1}{2s+1}})^{\frac{2s}{3s+1}})}{\exp(c t((n+\frac{1}{2})^{\frac{1+s}{2}}+\langle\xi\rangle^{\frac{3s+1}{2s+1}})^{\frac{2s}{3s+1}})}\sqrt{n+1}e_{n+1}
\\&
\quad+\frac{\exp(c t((n-\frac{1}{2})^{\frac{1+s}{2}}+\langle\xi\rangle^{\frac{3s+1}{2s+1}})^{\frac{2s}{3s+1}})}{1+\kappa\exp(c t((n-\frac{1}{2})^{\frac{1+s}{2}}+\langle\xi\rangle^{\frac{3s+1}{2s+1}})^{\frac{2s}{3s+1}})}
\frac{1+\kappa\exp(c t((n+\frac{1}{2})^{\frac{1+s}{2}}+\langle\xi\rangle^{\frac{3s+1}{2s+1}})^{\frac{2s}{3s+1}})}{\exp(c t((n+\frac{1}{2})^{\frac{1+s}{2}}+\langle\xi\rangle^{\frac{3s+1}{2s+1}})^{\frac{2s}{3s+1}})}\sqrt{n}e_{n-1}.
\end{split}
\end{equation*}
One can verify that
\begin{equation*}
\begin{split}
&\mathcal{F}_x\left(\mathcal{H}^{-\frac{s}{2}}\left[G_{\kappa}(c t)v(G_{\kappa}(c t))^{-1}-v\right]\Delta_p\partial_{x}f\right)
\\&
=\sum_{n=0}^{+\infty}i\xi\widehat{\Delta_pf_n}(\xi)\sqrt{n+1}(n+\frac{3}{2})^{-\frac{s}{2}}A^+_{n,\langle\xi\rangle}e_{n+1}\\
&\quad
+\sum_{n=1}^{+\infty}i\xi\widehat{\Delta_pf_n}(\xi)\sqrt{n}(n-\frac{1}{2})^{-\frac{s}{2}}A^-_{n,\langle\xi\rangle}e_{n-1},
\end{split}
\end{equation*}
where $\mathcal{F}_x$ stands for the partial Fourier transform with respect to the position variable $x$ and
\begin{equation}\label{Note-A}
\begin{split}
A^+_{n,\langle\xi\rangle}=\frac{\exp\left(c t\left((n+\frac{3}{2})^{\frac{s+1}{2}}+\langle\xi\rangle^{\frac{3s+1}{2s+1}}\right)^{\frac{2s}{3s+1}}-c t\left((n+\frac{1}{2})^{\frac{s+1}{2}}+\langle\xi\rangle^{\frac{3s+1}{2s+1}}\right)^{\frac{2s}{3s+1}}\right)-1}{\left(1+\kappa\exp\left(c t\left((n+\frac{3}{2})^{\frac{s+1}{2}}+\langle\xi\rangle^{\frac{3s+1}{2s+1}}\right)^{\frac{2s}{3s+1}}\right)\right)},\\
A^-_{n,\langle\xi\rangle}=\frac{\exp\left(c t\left((n-\frac{1}{2})^{\frac{s+1}{2}}+\langle\xi\rangle^{\frac{3s+1}{2s+1}}\right)^{\frac{2s}{3s+1}}-c t\left((n+\frac{1}{2})^{\frac{s+1}{2}}+\langle\xi\rangle^{\frac{3s+1}{2s+1}}\right)^{\frac{2s}{3s+1}}\right)-1}{\left(1+\kappa\exp\left(c t\left((n-\frac{1}{2})^{\frac{s+1}{2}}+\langle\xi\rangle^{\frac{3s+1}{2s+1}}\right)^{\frac{2s}{3s+1}}\right)\right)}.
\end{split}
\end{equation}
Then by the Plancherel theorem and Cauchy-Schwarz inequality, we have
\begin{equation}\label{Plancherel}
\begin{split}
&\sqrt{2\pi}\left\|\mathcal{H}^{-\frac{s}{2}}\left[G_{\kappa}(c t)v(G_{\kappa}(c t))^{-1}-v\right]\Delta_p\partial_{x}f\right\|_{L^2_{x,v}}\\
=&\left\|\mathcal{F}_x\left(\mathcal{H}^{-\frac{s}{2}}\left[G_{\kappa}(c t)v(G_{\kappa}(c t))^{-1}-v\right]\Delta_p\partial_{x}f\right)\right\|_{L^2_{\xi,v}}\\
\lesssim&\left(\sum_{n=0}^{+\infty}\|\xi\widehat{\Delta_pf_n}(\xi)A^+_{n,\langle\xi\rangle}\|^2_{L^2_{\xi}}\langle\,n\rangle^{1-s} \right)^{\frac{1}{2}}+\left(\sum_{n=1}^{+\infty}\|\xi\widehat{\Delta_pf_n}(\xi)A^-_{n,\langle\xi\rangle}\|^2_{L^2_{\xi}}\langle\,n\rangle^{1-s} \right)^{\frac{1}{2}}
\end{split}
\end{equation}
with $A^+_{n,\langle\xi\rangle},A^-_{n,\langle\xi\rangle}$ defined in \eqref{Note-A}.  Now we come to estimate $|A^+_{n,\langle\xi\rangle}|, |A^-_{n,\langle\xi\rangle}|$. It follows from the mean value theorem that
\begin{align*}
&\left((n+\frac{3}{2})^{\frac{s+1}{2}}+\langle\xi\rangle^{\frac{3s+1}{2s+1}}\right)^{\frac{2s}{3s+1}}-\left((n+\frac{1}{2})^{\frac{s+1}{2}}+\langle\xi\rangle^{\frac{3s+1}{2s+1}}\right)^{\frac{2s}{3s+1}}\\
&=\frac{c t s(1+s)}{3s+1}\left((n+\frac{1}{2}+\theta)^{\frac{s+1}{2}}+\langle\xi\rangle^{\frac{3s+1}{2s+1}}\right)^{-\frac{1+s}{3s+1}}(n+\frac{1}{2}+\theta)^{\frac{s-1}{2}}\quad(0<\theta<1),
\end{align*}
which leads to
\begin{align*}
0\leq A^+_{n,\langle\xi\rangle}\leq& \exp\left(\frac{2c t s(1+s)}{3s+1}\right)\frac{c t s(1+s)}{3s+1}\left((n+\frac{1}{2})^{\frac{s+1}{2}}+\langle\xi\rangle^{\frac{3s+1}{2s+1}}\right)^{-\frac{1+s}{3s+1}}(n+\frac{1}{2})^{\frac{s-1}{2}}\\
\leq& \exp\left(\frac{2c t s(1+s)}{3s+1}\right)\frac{2c t s(1+s)}{3s+1}\langle\xi\rangle^{-\frac{1+s}{2s+1}}\langle\,n\rangle^{\frac{s-1}{2}}.
\end{align*}
This shows that, for all $n\geq0$,
\begin{equation}\label{estimate1}
|A^+_{n,\langle\xi\rangle}|\leq \exp\left(\frac{2c t s(1+s)}{3s+1}\right)\frac{2c t s(1+s)}{3s+1}\langle\xi\rangle^{-\frac{1+s}{2s+1}}\langle\,n\rangle^{\frac{s-1}{2}}.
\end{equation}
On the other hand, we use the mean value theorem again,
\begin{align*}
&0\geq A^-_{n,\langle\xi\rangle}
\geq\exp\left(c t\left((n-\frac{1}{2})^{\frac{s+1}{2}}+\langle\xi\rangle^{\frac{3s+1}{2s+1}}\right)^{\frac{2s}{3s+1}}-c t\left((n+\frac{1}{2})^{\frac{s+1}{2}}+\langle\xi\rangle^{\frac{3s+1}{2s+1}}\right)^{\frac{2s}{3s+1}}\right)-1\\
&=\exp\left(-\frac{c t s(1+s)}{3s+1}\left((n-\frac{1}{2}+\theta)^{\frac{s+1}{2}}+\langle\xi\rangle^{\frac{3s+1}{2s+1}}\right)^{-\frac{1+s}{3s+1}}(n-\frac{1}{2}+\theta)^{\frac{s-1}{2}}\right)-1\quad(0<\theta<1)\\
&=\exp\left(-\frac{c \tau t s(1+s)}{3s+1}\left((n-\frac{1}{2}+\theta)^{\frac{s+1}{2}}+\langle\xi\rangle^{\frac{3s+1}{2s+1}}\right)^{-\frac{1+s}{3s+1}}(n-\frac{1}{2}+\theta)^{\frac{s-1}{2}}\right)\\
&\qquad\times\frac{c t s(1+s)}{3s+1}\left((n-\frac{1}{2}+\theta)^{\frac{s+1}{2}}+\langle\xi\rangle^{\frac{3s+1}{2s+1}}\right)^{-\frac{1+s}{3s+1}}(n-\frac{1}{2}+\theta)^{\frac{s-1}{2}}  \quad(0<\theta<1).
\end{align*}
Then for all $n\geq1$, we have
\begin{equation}\label{estimate2}
|A^-_{n,\langle\xi\rangle}|\leq \frac{2c t s(1+s)}{3s+1}\langle\xi\rangle^{-\frac{1+s}{2s+1}}\langle\,n\rangle^{\frac{s-1}{2}}.
\end{equation}
Substituting the results \eqref{estimate1} and \eqref{estimate2} into \eqref{Plancherel}, we conclude that
\begin{align*}
&\left\|\mathcal{H}^{-\frac{s}{2}}\left[G_{\kappa}(c t)v(G_{\kappa}(c t))^{-1}-v\right]\Delta_p\partial_{x}f\right\|_{L^2_{x,v}}\\
&\leq \exp\left(\frac{2c t s(1+s)}{3s+1}\right)\frac{2c t s(1+s)}{3s+1}\frac{1}{\sqrt{2\pi}}
\\&\quad
\times\left(\left(\sum_{n=0}^{+\infty}\|\langle\xi\rangle^{\frac{s}{2s+1}}\widehat{\Delta_pf_n}(\xi)\|^2_{L^2_{\xi}}\right)^{\frac{1}{2}}
+\left(\sum_{n=1}^{+\infty}\|\langle\xi\rangle^{\frac{s}{2s+1}}\widehat{\Delta_pf_n}(\xi)\|^2_{L^2_{\xi}} \right)^{\frac{1}{2}}\right)
\\&
\leq\tilde{c}_{1}c te^{\tilde{c}_{1}c t}\|\langle D_x\rangle^{\frac{s}{2s+1}}\Delta_pf\|_{L^2_{x,v}}.
\end{align*}
\end{proof}

\begin{proof}[The proof of Proposition \ref{Gelfand}]
Let $0\leq c\leq1$ and $0<\kappa\leq1$. Define
\begin{align}\label{H-1}
h_{n,c,\kappa}=G_{\kappa}(c t)\tilde{g}_n, \ \ \  n\geq0.
\end{align}
The function $h_{n,c,\kappa}$ depends on the parameters $0\leq c\leq1$ and $0<\kappa\leq1$. Here, we write $h_n$ for $h_{n,c,\kappa}$ for simplicity. Notice that
\begin{align*}
h_{0}(t)=(1+\kappa\exp(t(\mathcal{H}^{\frac{s+1}{2}}+\langle D_x\rangle^{\frac{3s+1}{2s+1}})^{\frac{2s}{3s+1}}))^{-1}g_0, \ \ 0\leq t\leq T,
\end{align*}
satisfies
\begin{align}\label{H-2}
\|h_{0}\|_{\widetilde{L}^\infty_{T}\widetilde{L}^2_{v}(B^{1/2}_{2,1})}\leq\|g_{0}\|_{\widetilde{L}^2_{v}(B^{1/2}_{2,1})}.
\end{align}
By using \eqref{H-1} that
\begin{align*}
\tilde{g}_n=(G_{\kappa}(c t))^{-1}h_n=(\kappa+\exp(-c t(\mathcal{H}^{\frac{s+1}{2}}+\langle D_x\rangle^{\frac{3s+1}{2s+1}})^{\frac{2s}{3s+1}}))h_n,
\end{align*}
then the equation
$$
\partial_t\tilde{g}_{n+1}+v\partial_x\tilde{g}_{n+1}+\mathcal{K}\tilde{g}_{n+1}=\Gamma(\tilde{g}_n, \tilde{g}_{n+1})
$$
can be rewritten as
\begin{align*}
&(G_{\kappa}(c t))^{-1}\partial_th_{n+1}+v\partial_{x}(G_{\kappa}(c t))^{-1}h_{n+1}+(G_{\kappa}(c t))^{-1}\mathcal{K}h_{n+1}
\\&\qquad
-c(\mathcal{H}^{\frac{s+1}{2}}+\langle D_x\rangle^{\frac{3s+1}{2s+1}})^{\frac{2s}{3s+1}}
\exp(-c t(\mathcal{H}^{\frac{s+1}{2}}+\langle D_x\rangle^{\frac{3s+1}{2s+1}})^{\frac{2s}{3s+1}})h_{n+1}
\\
=&\Gamma((G_{\kappa}(c t))^{-1}h_n,(G_{\kappa}(c t))^{-1}h_{n+1}).
\end{align*}
Due to \eqref{A-2}, the linearized Kac operator $\mathcal{K}=f(\mathcal{H})$ is a function of the harmonic oscillator acting on the velocity variable $v$, which can commute with the exponential weight $(G_{\kappa}(c t))^{-1}$.   Applying $\Delta_p(p\geq-1)$ to the resulting equality, we have
\begin{equation}\label{H-3}
\begin{split}
&\partial_t\Delta_ph_{n+1}+G_{\kappa}(c t)v(G_{\kappa}(c t))^{-1}\partial_{x}\Delta_ph_{n+1}+\mathcal{K}\Delta_ph_{n+1}
\\&\qquad
-\frac{c(\mathcal{H}^{\frac{s+1}{2}}+\langle D_x\rangle^{\frac{3s+1}{2s+1}})^{\frac{2s}{3s+1}}}{1+\kappa\exp(c t(\mathcal{H}^{\frac{s+1}{2}}+\langle D_x\rangle^{\frac{3s+1}{2s+1}})^{\frac{2s}{3s+1}})}\Delta_ph_{n+1}
\\&
=G_{\kappa}(c t)\Delta_p\Gamma((G_{\kappa}(c t))^{-1}h_n,(G_{\kappa}(c t))^{-1}h_{n+1}).
\end{split}
\end{equation}
According to Lemma \ref{Linear} and \eqref{G} in Section \ref{S5}, we choose the positive parameter $0<\varepsilon\leq\varepsilon_0$  in order to ensure that the multiplier
\begin{align}\label{Q}
Q=Q(v,D_v,D_x)=1-\varepsilon m^w(v,D_v,D_x)
\end{align}
is a positive bounded isomorphism on $L^2(\mathbb{R}^2_{x,v})$.

By integrating with respect to the $\xi-$variable and using the multiplier $Q\Delta_ph_{n+1}$ in $L^{2}(\mathbb{R}^2_{x,v})$, we deduce that from \eqref{H-3}
\begin{equation}\label{H-4}
\begin{split}
&\frac{1}{2}\frac{d}{dt}\|Q^{1/2}\Delta_ph_{n+1}\|^{2}_{L^2_{x,v}}+\mathrm{Re}(\mathcal{K}\Delta_ph_{n+1}, Q\Delta_ph_{n+1})_{L^2(\mathbb{R}^2_{x,v})}
\\&
\qquad+\mathrm{Re}(G_{\kappa}(c t)v(G_{\kappa}(c t))^{-1}\partial_{x}\Delta_ph_{n+1}, Q\Delta_ph_{n+1})_{L^2(\mathbb{R}^2_{x,v})}
\\&\qquad
-\mathrm{Re}\Big(\frac{c(\mathcal{H}^{\frac{s+1}{2}}+\langle D_x\rangle^{\frac{3s+1}{2s+1}})^{\frac{2s}{3s+1}}}{\exp(-c t(\mathcal{H}^{\frac{s+1}{2}}+\langle D_x\rangle^{\frac{3s+1}{2s+1}})^{\frac{2s}{3s+1}})}\Delta_ph_{n+1}, Q\Delta_ph_{n+1}\Big)_{L^2(\mathbb{R}^2_{x,v})}
\\&
=\mathrm{Re}(G_{\kappa}(c t)\Delta_p\Gamma((G_{\kappa}(c t))^{-1}h_n,(G_{\kappa}(c t))^{-1}h_{n+1}), Q\Delta_ph_{n+1})_{L^2(\mathbb{R}^2_{x,v})},
\end{split}
\end{equation}
which leads to
\begin{equation}\label{H-4}
\begin{split}
&\frac{1}{2}\frac{d}{dt}\|Q^{1/2}\Delta_ph_{n+1}\|^{2}_{L^2_{x,v}}+\mathrm{Re}((v\partial_x+\mathcal{K})\Delta_ph_{n+1}, Q\Delta_ph_{n+1})_{L^2(\mathbb{R}^2_{x,v})}
\\&
\quad+\mathrm{Re}([G_{\kappa}(c t)v(G_{\kappa}(c t))^{-1}-v]\partial_{x}\Delta_ph_{n+1}, Q\Delta_ph_{n+1})_{L^2(\mathbb{R}^2_{x,v})}
\\&
\leq c\|(\mathcal{H}^{\frac{s+1}{2}}+\langle D_x\rangle^{\frac{3s+1}{2s+1}})^{\frac{s}{3s+1}}\Delta_ph_{n+1}\|_{L^2_{x,v}}
\|(\mathcal{H}^{\frac{s+1}{2}}+\langle D_x\rangle^{\frac{3s+1}{2s+1}})^{\frac{s}{3s+1}}Q\Delta_ph_{n+1}\|_{L^2_{x,v}}
\\&\quad
+|(G_{\kappa}(c t)\Delta_p\Gamma((G_{\kappa}(c t))^{-1}h_n,(G_{\kappa}(c t))^{-1}h_{n+1}), Q\Delta_ph_{n+1})_{L^2(\mathbb{R}^2_{x,v})}|.
\end{split}
\end{equation}
It follows from Lemma \ref{Linear-B} that
\begin{equation}\label{H-5}
\begin{split}
&\mathrm{Re}((v\partial_x+\mathcal{K})\Delta_pu, Q\Delta_pu)_{L^2(\mathbb{R}^2_{x,v})}
\\&
\geq c_3\|\mathcal{H}^{\frac{s}{2}}\Delta_pu\|^2_{L^2(\mathbb{R}^2_{x,v})}+c_3\|\langle D_x\rangle^{\frac{s}{2s+1}}\Delta_pu\|^2_{L^2(\mathbb{R}^2_{x,v})}
-c_4\|\Delta_pu\|^2_{L^2(\mathbb{R}^2_{x,v})}
\end{split}
\end{equation}
for some constants $c_3,c_4>0$. We deduce from \eqref{H-4} and \eqref{H-5} that
\begin{equation}\label{H-6}
\begin{split}
&\frac{1}{2}\frac{d}{dt}\|Q^{1/2}\Delta_ph_{n+1}\|^{2}_{L^2_{x,v}}+c_3\|\mathcal{H}^{\frac{s}{2}}\Delta_ph_{n+1}\|^2_{L^2(\mathbb{R}^2_{x,v})}+c_3\|\langle D_x\rangle^{\frac{s}{2s+1}}\Delta_ph_{n+1}\|^2_{L^2(\mathbb{R}^2_{x,v})}
\\&
\leq c\|(\mathcal{H}^{\frac{s+1}{2}}+\langle D_x\rangle^{\frac{3s+1}{2s+1}})^{\frac{s}{3s+1}}\Delta_ph_{n+1}\|_{L^2_{x,v}}
\|(\mathcal{H}^{\frac{s+1}{2}}+\langle D_x\rangle^{\frac{3s+1}{2s+1}})^{\frac{s}{3s+1}}Q\Delta_ph_{n+1}\|_{L^2_{x,v}}
\\&
\quad+|([G_{\kappa}(c t)v(G_{\kappa}(c t))^{-1}-v]\partial_{x}\Delta_ph_{n+1}, Q\Delta_ph_{n+1})_{L^2(\mathbb{R}^2_{x,v})}|
\\&
\quad+|(G_{\kappa}(c t)\Delta_p\Gamma((G_{\kappa}(c t))^{-1}h_n,(G_{\kappa}(c t))^{-1}h_{n+1}), Q\Delta_ph_{n+1})_{L^2(\mathbb{R}^2_{x,v})}|\\
&\quad+c_4\|\Delta_ph_{n+1}\|^2_{L^2(\mathbb{R}^2_{x,v})}.
\end{split}
\end{equation}

By using Lemma \ref{HD} for $f$ being replaced by $\Delta_ph_{n+1}$ and $Q\Delta_ph_{n+1}$ and Lemma \ref{HD-1}, we have
\begin{align*}
&\Big\|\left(\mathcal{H}^{\frac{s+1}{2}}+\langle D_x\rangle^{\frac{3s+1}{2s+1}}\right)^{\frac{s}{3s+1}}\Delta_ph_{n+1}\Big\|_{L^2_{x,v}}
\Big\|\left(\mathcal{H}^{\frac{s+1}{2}}+\langle D_x\rangle^{\frac{3s+1}{2s+1}}\right)^{\frac{s}{3s+1}}Q\Delta_ph_{n+1}\Big\|_{L^2_{x,v}}
\\&
\lesssim\Big(\left\|\mathcal{H}^{\frac{s}{2}}\Delta_ph_{n+1}\right\|_{L^2_{x,v}}
+\left\|\langle D_x\rangle^{\frac{s}{2s+1}}\Delta_ph_{n+1}\right\|_{L^2_{x,v}}\Big)
\\&\qquad
\times\Big(\left\|\mathcal{H}^{\frac{s}{2}}Q\Delta_ph_{n+1}\right\|_{L^2_{x,v}}
+\left\|\langle D_x\rangle^{\frac{s}{2s+1}}Q\Delta_ph_{n+1}\right\|_{L^2_{x,v}}\Big)
\\&
\lesssim\Big\|\mathcal{H}^{\frac{s}{2}}\Delta_ph_{n+1}\Big\|^{2}_{L^2_{x,v}}
+\Big\|\langle D_x\rangle^{\frac{s}{2s+1}}\Delta_ph_{n+1}\Big\|^{2}_{L^2_{x,v}}.
\end{align*}
Based on Lemma \ref{HD-1} and Lemma \ref{HD-2}, we obtain
\begin{align*}
&\left|\left(\left[G_{\kappa}(c t)v(G_{\kappa}(c t))^{-1}-v\right]\partial_{x}\Delta_ph_{n+1}, Q\Delta_ph_{n+1}\right)_{L^2(\mathbb{R}^2_{x,v})}\right|
\\&=
\left|\left(\mathcal{H}^{-\frac{s}{2}}\left[G_{\kappa}(c t)v(G_{\kappa}(c t))^{-1}-v\right]\partial_{x}\Delta_ph_{n+1}, \mathcal{H}^{\frac{s}{2}}Q\Delta_ph_{n+1}\right)_{L^2(\mathbb{R}^2_{x,v})}\right|
\\&
\leq\left\|\mathcal{H}^{-\frac{s}{2}}\left[G_{\kappa}(c t)v(G_{\kappa}(c t))^{-1}-v\right]\partial_{x}\Delta_ph_{n+1}\right\|_{L^2_{x,v}}
\left\|\mathcal{H}^{\frac{s}{2}}Q\Delta_ph_{n+1}\right\|_{L^2_{x,v}}
\\&
\leq c_8c te^{c_8c t}\left\|\langle D_x\rangle^{\frac{s}{2s+1}}\Delta_ph_{n+1}\right\|_{L^2_{x,v}}\left\|\mathcal{H}^{\frac{s}{2}}\Delta_ph_{n+1}\right\|_{L^2_{x,v}},
\end{align*}
since $Q$ is commuting with any function of the operator $D_x$. Then, it follows from \eqref{H-6} that there exists some positive constants $0<c_0\leq1, c_9>0$ such that for $0\leq c\leq c_0, 0<\kappa\leq1,0\leq t\leq T$,
\begin{equation*}
\begin{split}
&\frac{1}{2}\frac{d}{dt}\big\|Q^{1/2}\Delta_ph_{n+1}\big\|^{2}_{L^2_{x,v}}+\left(c_9-c_8c Te^{c_8T}\right)(\left\|\mathcal{H}^{\frac{s}{2}}\Delta_ph_{n+1}\right\|^2_{L^2(\mathbb{R}^2_{x,v})}
\\&
+\left\|\langle D_x\rangle^{\frac{s}{2s+1}}\Delta_ph_{n+1}\right\|^2_{L^2(\mathbb{R}^2_{x,v})})
\leq c_4\left\|\Delta_ph_{n+1}\right\|^2_{L^2(\mathbb{R}^2_{x,v})}
\\&\quad
+\left|\left(G_{\kappa}(c t)\Delta_p\Gamma((G_{\kappa}(c t))^{-1}h_n,(G_{\kappa}(c t))^{-1}h_{n+1}), Q\Delta_ph_{n+1}\right)_{L^2(\mathbb{R}^2_{x,v})}\right|.
\end{split}
\end{equation*}
Here, the constant $0<c_0\leq1$ is chosen sufficiently small so that
$$
c_8c Te^{c_8T}\leq\frac{c_9}{2},
$$
then we obtain that for all $0\leq c\leq c_0, 0<\kappa\leq1,0\leq t\leq T$,
\begin{equation*}
\begin{split}
&\frac{d}{dt}\big\|Q^{1/2}\Delta_ph_{n+1}\big\|^{2}_{L^2_{x,v}}+c_9\left\|\mathcal{H}^{\frac{s}{2}}\Delta_ph_{n+1}\right\|^2_{L^2_{x,v}}
+c_9\left\|\langle D_x\rangle^{\frac{s}{2s+1}}\Delta_ph_{n+1}\right\|^2_{L^2_{x,v}}
\\&
\leq 2c_{10}\big\|Q^{1/2}\Delta_ph_{n+1}\big\|^2_{L^2_{x,v}}
+2\left|\left(G_{\kappa}(c t)\Delta_p\Gamma((G_{\kappa}(c t))^{-1}h_n,(G_{\kappa}(c t))^{-1}h_{n+1}), Q\Delta_ph_{n+1}\right)_{L^2_{x,v}}\right|
\end{split}
\end{equation*}
with $c_{10}=c_4\|(Q^{1/2})^{-1}\|_{\mathcal{L}(L^2)}>0$. Following from \eqref{equationA}, \eqref{H-1} and \eqref{H-2}, for all $0\leq c\leq c_0, 0<\kappa\leq1,0\leq t\leq T$ we are led to
\begin{equation*}
\begin{split}
&\left\|Q^{1/2}\Delta_ph_{n+1}\right\|^{2}_{L^2_{x,v}}
\\&
+c_9\int_0^te^{2c_{10}(t-\tau)}\left(\left\|\mathcal{H}^{\frac{s}{2}}\Delta_ph_{n+1}(\tau)\right\|^2_{L^2_{x,v}}
+\left\|\langle D_x\rangle^{\frac{s}{2s+1}}\Delta_ph_{n+1}(\tau)\right\|^2_{L^2_{x,v}}\right)d\tau
\\&
\leq e^{2c_{10}t}\left\|Q^{1/2}\Delta_pg_{0}\right\|^{2}_{L^2_{x,v}}+2\int_0^te^{2c_{10}(t-\tau)}
\\&
\times\left|\left(G_{\kappa}(c\tau)\Delta_p\Gamma((G_{\kappa}(c\tau))^{-1}h_n(\tau),(G_{\kappa}(c\tau))^{-1}h_{n+1}(\tau)), Q\Delta_ph_{n+1}(\tau)\right)_{L^2(\mathbb{R}^2_{x,v})}\right|d\tau
\\&
\leq e^{2c_{10}t}\big\|Q^{1/2}\big\|^2_{\mathcal{L}(L^2)}\left\|\Delta_pg_{0}\right\|^{2}_{L^2_{x,v}}
\\&
+2\int_0^te^{2c_{10}(t-\tau)}
\\&
\times\left|\left(G_{\kappa}(c\tau)\Delta_p\Gamma((G_{\kappa}(c\tau))^{-1}h_n(\tau),(G_{\kappa}(c\tau))^{-1}h_{n+1}(\tau)), Q\Delta_ph_{n+1}(\tau)\right)_{L^2(\mathbb{R}^2_{x,v})}\right|d\tau.
\end{split}
\end{equation*}
Taking the square root of the above inequality and taking the supremum over $0\leq t\leq T$ give
\begin{equation*}
\begin{split}
&\big\|Q^{1/2}\Delta_ph_{n+1}\big\|_{L^{\infty}_TL^2_{x,v}}+\sqrt{c_9}\left\|\mathcal{H}^{\frac{s}{2}}\Delta_ph_{n+1}\right\|_{L^{2}_TL^2_{x,v}}
+\sqrt{c_9}\left\|\langle D_x\rangle^{\frac{s}{2s+1}}\Delta_ph_{n+1}\right\|_{L^{2}_TL^2_{x,v}}
\\&
\leq e^{c_{10}T}\|Q^{1/2}\|_{\mathcal{L}(L^2)}\left\|\Delta_pg_{0}\right\|_{L^2_{x,v}}
\\&
+\sqrt{2}e^{c_{10}T}\left(\int_0^T\left|\left(G_{\kappa}(c t)\Delta_p\Gamma((G_{\kappa}(c t))^{-1}h_n,(G_{\kappa}(c t))^{-1}h_{n+1}), Q\Delta_ph_{n+1}\right)_{L^2_{x,v}}\right|dt\right)^{1/2}.
\end{split}
\end{equation*}
Multiply the above inequality by $2^{\frac{p}{2}}$ and take the summation over $p\geq-1$, we have
\begin{equation*}
\begin{split}
&\|Q^{1/2}h_{n+1}\|_{\widetilde{L}^{\infty}_T\widetilde{L}^2_{v}(B^{1/2}_{2,1})}+\sqrt{c_9}\|\mathcal{H}^{\frac{s}{2}}h_{n+1}\|
_{\widetilde{L}^{2}_T\widetilde{L}^2_{v}(B^{1/2}_{2,1})}+\sqrt{c_9}\|\langle D_x\rangle^{\frac{s}{2s+1}}h_{n+1}\|
_{\widetilde{L}^{2}_T\widetilde{L}^2_{v}(B^{1/2}_{2,1})}
\\&
\leq e^{c_{10}T}\|Q^{1/2}\|_{\mathcal{L}(L^2)}\|g_{0}\|_{\widetilde{L}^2_{v}(B^{1/2}_{2,1})}+\sqrt{2}e^{c_{10}T}\sum_{p\geq-1}2^{\frac{p}{2}}
\\&
\times\left(\int_0^T\left|\left(G_{\kappa}(c t)\Delta_p\Gamma((G_{\kappa}(c t))^{-1}h_n,(G_{\kappa}(c t))^{-1}h_{n+1}), Q\Delta_ph_{n+1})_{L^2(\mathbb{R}^2_{x,v})}\right)\right|dt\right)^{1/2}.
\end{split}
\end{equation*}
It follows from Lemma \ref{trilinear-G} that
\begin{equation}\label{h-n11}
\begin{split}
&\|h_{n+1}\|_{\widetilde{L}^{\infty}_T\widetilde{L}^2_{v}(B^{1/2}_{2,1})}+\sqrt{c_9}\|\mathcal{H}^{\frac{s}{2}}h_{n+1}\|
_{\widetilde{L}^{2}_T\widetilde{L}^2_{v}(B^{1/2}_{2,1})}+\sqrt{c_9}\|\langle D_x\rangle^{\frac{s}{2s+1}}h_{n+1}\|
_{\widetilde{L}^{2}_T\widetilde{L}^2_{v}(B^{1/2}_{2,1})}
\\&
\leq e^{c_{10}T}\|Q^{1/2}\|_{\mathcal{L}(L^2)}\|(Q^{1/2})^{-1}\|_{\mathcal{L}(L^2)}\|g_{0}\|_{\widetilde{L}^2_{v}(B^{1/2}_{2,1})}
\\&\qquad
+\sqrt{2}e^{c_{10}T}C_1\|h_{n}\|^{1/2}_{\widetilde{L}^{\infty}_T\widetilde{L}^2_{v}(B^{1/2}_{2,1})}
\|\mathcal{H}^{\frac{s}{2}}h_{n+1}\|_{\widetilde{L}^{2}_T\widetilde{L}^2_{v}(B^{1/2}_{2,1})}
\\&
\leq e^{c_{10}T}\|Q^{1/2}\|_{\mathcal{L}(L^2)}\|(Q^{1/2})^{-1}\|_{\mathcal{L}(L^2)}\|g_{0}\|_{\widetilde{L}^2_{v}(B^{1/2}_{2,1})}
\\&\qquad
+e^{c_{10}T}c_{11}\|h_{n}\|^{1/2}_{\widetilde{L}^{\infty}_T\widetilde{L}^2_{v}(B^{1/2}_{2,1})}
\|\mathcal{H}^{\frac{s}{2}}h_{n+1}\|_{\widetilde{L}^{2}_T\widetilde{L}^2_{v}(B^{1/2}_{2,1})}.
\end{split}
\end{equation}
Next, we use the mathematical induction argument to show that
\begin{align}\label{h-n}
\left\|h_{n}\right\|_{\widetilde{L}^{\infty}_T\widetilde{L}^2_{v}(B^{1/2}_{2,1})}\leq\frac{c_9}{2c_{11}^2e^{2c_{10}T}}
\end{align}
for $n\geq0$. In the case of $n=0$, owing to the assumption
\begin{align}\label{H-10}
\|g_{0}\|_{\widetilde{L}^2_{v}(B^{1/2}_{2,1})}\leq\varepsilon_1, \ \text{with} \ 0<\varepsilon_1=\inf(\tilde{\varepsilon}_0,\frac{c_9}{2c_{11}^2e^{2c_{10}T}},\frac{c_9}{2c_{11}^2e^{3c_{10}T}\|Q^{1/2}\|_{\mathcal{L}(L^2)}})
\leq\tilde{\varepsilon}_0
\end{align}
where the positive parameter $\tilde{\varepsilon}_0>0$ is defined in \eqref{g}, we deduce from \eqref{H-2} that
$$
\|h_{0}\|_{\widetilde{L}^{\infty}_T\widetilde{L}^2_{v}(B^{1/2}_{2,1})}\leq\|g_{0}\|_{\widetilde{L}^2_{v}(B^{1/2}_{2,1})}\leq\frac{c_9}{2c_{11}^2e^{2c_{10}T}}.
$$
In the case of $n\geq1$, if we assume that
$$
\left\|h_{n}\right\|_{\widetilde{L}^{\infty}_T\widetilde{L}^2_{v}(B^{1/2}_{2,1})}\leq\frac{c_9}{2c_{11}^2e^{2c_{10}T}},
$$
then it follows from \eqref{h-n11} that
\begin{equation}\label{H-9}
\begin{split}
&\|h_{n+1}\|_{\widetilde{L}^{\infty}_T\widetilde{L}^2_{v}(B^{1/2}_{2,1})}+\sqrt{\frac{c_9}{2}}\|\mathcal{H}^{\frac{s}{2}}h_{n+1}\|
_{\widetilde{L}^{2}_T\widetilde{L}^2_{v}(B^{1/2}_{2,1})}+\sqrt{c_9}\|\langle D_x\rangle^{\frac{s}{2s+1}}h_{n+1}\|
_{\widetilde{L}^{2}_T\widetilde{L}^2_{v}(B^{1/2}_{2,1})}
\\&
\leq e^{c_{10}T}\|Q^{1/2}\|_{\mathcal{L}(L^2)}\|(Q^{1/2})^{-1}\|_{\mathcal{L}(L^2)}\|g_{0}\|_{\widetilde{L}^2_{v}(B^{1/2}_{2,1})}.
\end{split}
\end{equation}
Together with \eqref{H-10} and \eqref{H-9}, we deduce that
$$
\|h_{n+1}\|_{\widetilde{L}^{\infty}_T\widetilde{L}^2_{v}(B^{1/2}_{2,1})}\leq\frac{c_9}{2c_{11}^2e^{2c_{10}T}}.
$$
Hence, it follows from \eqref{H-9} that for $0\leq c\leq c_0, 0<\kappa\leq1,n\geq1$,
\begin{equation*}
\begin{split}
&\|h_{n}\|_{\widetilde{L}^{\infty}_T\widetilde{L}^2_{v}(B^{1/2}_{2,1})}+\sqrt{\frac{c_9}{2}}\|\mathcal{H}^{\frac{s}{2}}h_{n}\|
_{\widetilde{L}^{2}_T\widetilde{L}^2_{v}(B^{1/2}_{2,1})}+\sqrt{c_9}\|\langle D_x\rangle^{\frac{s}{2s+1}}h_{n}\|
_{\widetilde{L}^{2}_T\widetilde{L}^2_{v}(B^{1/2}_{2,1})}
\\&
\leq e^{c_{10}T}\|Q^{1/2}\|_{\mathcal{L}(L^2)}\|(Q^{1/2})^{-1}\|_{\mathcal{L}(L^2)}\|g_{0}\|_{\widetilde{L}^2_{v}(B^{1/2}_{2,1})}.
\end{split}
\end{equation*}
This ends the proof of Proposition \ref{Gelfand}.
\end{proof}

Based on Proposition \ref{Gelfand}, by passing the limit when $\kappa\rightarrow0_{+}$ in the estimate \eqref{G-S1}, it follows from the monotone convergence theorem that the following lemma:

\begin{lemma}\label{Gelfand-1}
Let $T>0$. Then, there exist some constants $C, \varepsilon_1>0, 0<c_0\leq1$ such that for all initial data $\|g_0\|_{\widetilde{L}^2_{v}(B^{1/2}_{2,1})}\leq\varepsilon_1$, the sequence of approximate solutions $(\tilde{g}_n)_{n\geq0}$ satisfies
\begin{equation*}
\begin{split}
&\|G_{0}(c t)\tilde{g}_n\|_{\widetilde{L}^{\infty}_T\widetilde{L}^2_{v}(B^{1/2}_{2,1})}
+\|\mathcal{H}^{\frac{s}{2}}G_{0}(c t)\tilde{g}_n\|_{\widetilde{L}^{2}_T\widetilde{L}^2_{v}(B^{1/2}_{2,1})}
\\&
+\|\langle D_x\rangle^{\frac{s}{2s+1}}G_{0}(c t)\tilde{g}_n\|_{\widetilde{L}^{2}_T\widetilde{L}^2_{v}(B^{1/2}_{2,1})}
\leq Ce^{CT}\|g_{0}\|_{\widetilde{L}^2_{v}(B^{1/2}_{2,1})}
\end{split}
\end{equation*}
for all $0<c\leq c_0, n\geq1$, where
$$
G_{0}(t)=\exp(t(\mathcal{H}^{\frac{s+1}{2}}+\langle D_x\rangle^{\frac{3s+1}{2s+1}})^{\frac{2s}{3s+1}}).
$$
\end{lemma}

\subsection{Gelfand-Shilov and Gevrey regularities}

It follows from the Cauchy-Schwarz inequality  that for all $0<c\leq c_0,0<\kappa\leq1$,
\begin{align*}
\left\|G_{\kappa}(c t)\Delta_pf\right\|^{2}_{L^2_{x,v}}
\leq\left\|G_{0}(2c t)\Delta_pf\right\|_{L^2_{x,v}}\left\|\Delta_pf\right\|_{L^2_{x,v}}.
\end{align*}
By passing to the limit $\kappa\rightarrow0_{+}$ in the above inequality, it follows from the monotone convergence theorem that for all $0<c\leq c_0$,
$$
\|G_0(c t)\Delta_pf\|^{2}_{L^2_{x,v}}\leq\|G_{0}(2c t)\Delta_pf\|_{L^2_{x,v}}\|\Delta_pf\|_{L^2_{x,v}}.
$$
It implies that for all $0<c\leq c_0$,
\begin{equation}\label{M-f}
\begin{split}
\|G_{0}(c t)f\|_{\widetilde{L}^{\infty}_T\widetilde{L}^2_{v}(B^{1/2}_{2,1})}&
\leq\Big(\sum_{p\geq-1}2^{\frac{p}{2}}\|G_{0}(2c t)\Delta_pf\|_{L^\infty_TL^2_{x,v}}\Big)^{1/2}
\Big(\sum_{p\geq-1}2^{\frac{p}{2}}\|\Delta_pf\|_{L^\infty_TL^2_{x,v}}\Big)^{1/2}
\\&
\leq\|G_{0}(2c t)f\|^{1/2}_{\widetilde{L}^{\infty}_T\widetilde{L}^2_{v}(B^{1/2}_{2,1})}
\|f\|^{1/2}_{\widetilde{L}^{\infty}_T\widetilde{L}^2_{v}(B^{1/2}_{2,1})}.
\end{split}
\end{equation}
For the solutions $(\tilde{g}_n)_{n\geq0}$ defined in \eqref{equationA}, by using Lemma \ref{Gelfand-1} and \eqref{M-f}, we can obtain that for $0\leq c\leq\frac{c_0}{2}$,
\begin{align*}
&\|G_0(c t)\tilde{g}_{n+p}-G_0(c t)\tilde{g}_n\|_{\widetilde{L}^{\infty}_T\widetilde{L}^2_{v}(B^{1/2}_{2,1})}
\\&
\leq2\sqrt{C}e^{\frac{C T}{2}}\|g_{0}\|^{1/2}_{\widetilde{L}^2_{v}(B^{1/2}_{2,1})}
\|\tilde{g}_{n+p}-\tilde{g}_n\|^{1/2}_{\widetilde{L}^{\infty}_T\widetilde{L}^2_{v}(B^{1/2}_{2,1})},
\end{align*}
which implies that $(G_0(c t)\tilde{g}_n)_{n\geq1}$ is a Cauchy sequence in $\widetilde{L}^{\infty}_T\widetilde{L}^2_{v}(B^{1/2}_{2,1})$.
Let $h$ be the limit of the Cauchy sequence $(G_0(\frac{c_0}{2} t)\tilde{g}_n)_{n\geq1}$ in the space $\widetilde{L}^{\infty}_T\widetilde{L}^2_{v}(B^{1/2}_{2,1})$. Notice that
\begin{align*}
\Big\|\tilde{g}_n-\Big(G_0\Big(\frac{c_0}{2} t\Big)\Big)^{-1}h\Big\|_{\widetilde{L}^{\infty}_T\widetilde{L}^2_{v}(B^{1/2}_{2,1})}
\leq\Big\|G_0\Big(\frac{c_0}{2} t\Big)\tilde{g}_n-h\Big\|_{\widetilde{L}^{\infty}_T\widetilde{L}^2_{v}(B^{1/2}_{2,1})},
\end{align*}
then following from the convergence of the sequences $\{\tilde{g}_n\}$ in $\widetilde{L}^{\infty}_T\widetilde{L}^2_{v}(B^{1/2}_{2,1})$ and the uniqueness of the solution to the Cauchy problem \eqref{eq-1}, we have
\begin{align*}
g=\Big(G_0\Big(\frac{c_0}{2} t\Big)\Big)^{-1}h=\exp\left(-\frac{c_0}{2}t\left(\mathcal{H}^{\frac{s+1}{2}}+\langle D_x\rangle^{\frac{3s+1}{2s+1}}\right)^{\frac{2s}{3s+1}}\right)h.
\end{align*}
On the other hand, we can deduce from Lemma \ref{Gelfand-1} that
\begin{align}\label{g-limit2}
\Big\|G_0(\frac{c_0}{2} t)\tilde{g}_n\Big\|_{\widetilde{L}^{\infty}_T\widetilde{L}^2_{v}(B^{1/2}_{2,1})}
\leq Ce^{CT}\|g_{0}\|_{\widetilde{L}^2_{v}(B^{1/2}_{2,1})}.
\end{align}
Passing to the limit in the above estimate \eqref{g-limit2} when $n\rightarrow+\infty$, we obtain
\begin{equation}\label{g-limit3}
\big\|\exp\big(\frac{c_0 t}{2}\big(\mathcal{H}^{\frac{s+1}{2}}+\langle D_x\rangle^{\frac{3s+1}{2s+1}}\big)^{\frac{2s}{3s+1}}\big)g\big\|_{\widetilde{L}^{\infty}_T\widetilde{L}^2_{v}(B^{1/2}_{2,1})}
\leq Ce^{CT}\|g_{0}\|_{\widetilde{L}^2_{v}(B^{1/2}_{2,1})}.
\end{equation}
By using the following elementary inequality,
\begin{align*}
\forall x,c>0, \ \ \ x^k\exp\left(-\frac{3s+1}{s}cx^{\frac{s}{3s+1}}\right)\leq\frac{(k!)^{\frac{3s+1}{s}}}{c^{\frac{3s+1}{s}k}},
\end{align*}
we can deduce that for $f\in\mathcal{S}(\mathbb{R}^{2}_{x,v})$, for all $k\geq0$,
\begin{equation}\label{gelfand-Shilov-1}
\begin{split}
&\left\|\left(\mathcal{H}^{\frac{s+1}{2}}+\langle D_x\rangle^{\frac{3s+1}{2s+1}}\right)^{k}\exp\left(-\frac{c_0t}{2}\left(\mathcal{H}^{\frac{s+1}{2}}+\langle D_x\rangle^{\frac{3s+1}{2s+1}}\right)^{\frac{2s}{3s+1}}\right)\Delta_pf\right\|_{L^2_{x,v}}
\\&
\leq\Big(\frac{3s+1}{sc_0}\Big)^{\frac{3s+1}{2s}k}\frac{(k!)^{\frac{3s+1}{2s}}}{t^{\frac{3s+1}{2s}k}}\|\Delta_pf\|_{L^2_{x,v}}.
\end{split}
\end{equation}
Then it follows from \eqref{g-limit3} and \eqref{gelfand-Shilov-1} that the solution to the Cauchy problem \eqref{eq-1} satisfies for all $0<t\leq T, k\geq0$,
\begin{equation}\label{gelfand-Shilov-2}
\begin{split}
&\|\left(\mathcal{H}^{\frac{s+1}{2}}+\langle D_x\rangle^{\frac{3s+1}{2s+1}}\right)^{k}g\|_{\widetilde{L}^2_{v}(B^{1/2}_{2,1})}
\\&
\leq\Big(\frac{3s+1}{sc_0}\Big)^{\frac{3s+1}{2s}k}\frac{(k!)^{\frac{3s+1}{2s}}}{t^{\frac{3s+1}{2s}k}}
\big\|\exp\big(\frac{c_0 t}{2}\big(\mathcal{H}^{\frac{s+1}{2}}+\langle D_x\rangle^{\frac{3s+1}{2s+1}}\big)^{\frac{2s}{3s+1}}\big)g\big\|_{\widetilde{L}^{\infty}_T\widetilde{L}^2_{v}(B^{1/2}_{2,1})}
\\&
\leq\Big(\frac{3s+1}{sc_0}\Big)^{\frac{3s+1}{2s}k}\frac{(k!)^{\frac{3s+1}{2s}}}{t^{\frac{3s+1}{2s}k}}Ce^{CT}\|g_0\|_{\widetilde{L}^2_{v}(B^{1/2}_{2,1})}.
\end{split}
\end{equation}
By \eqref{gelfand-Shilov-2}, we obtain that there exists a positive constant $C>1$ such that $\forall0<t\leq T, k\geq0$,
\begin{equation*}
\left\|\left(\mathcal{H}^{\frac{s+1}{2}}+\langle D_x\rangle^{\frac{3s+1}{2s+1}}\right)^{k}g(t)\right\|_{\widetilde{L}^2_{v}(B^{1/2}_{2,1})}
\leq \frac{C^{k+1}}{t^{\frac{3s+1}{2s}k}}(k!)^{\frac{3s+1}{2s}}\left\|g_0\right\|_{\widetilde{L}^2_{v}(B^{1/2}_{2,1})}.
\end{equation*}
This proves the Gelfand-Shilov property in Theorem \ref{Main}.

On the other hand, we have for $p\geq-1,q\geq0$,
\begin{align}\label{g-besov1}
&\partial_x^q\Delta_pg(x,v)=\sum_{n=0}^{+\infty}\partial_x^q\Delta_pg_n(x)e_n(x), \ \  \text{with} \ \ g_n(x)=(g(x,\cdot),e_n)_{L^2(\mathbb{R}_v)}
\end{align}
and
\begin{align}\label{g-besov2}
\partial_x^q\Delta_pg_n(x)=(\partial_x^q\Delta_pg(x,\cdot),e_n)_{L^2(\mathbb{R}_v)}.
\end{align}

We deduce from \eqref{g-besov1}-\eqref{g-besov2} and Lemma \ref{J10-5} with $r=\frac{3s+1}{2s(s+1)}$ that there exist some constants $C_1, C_2>0$ such that for for all $k,l,q\geq0, \varepsilon>0$,
\begin{equation}\label{Gevrey-1}
\begin{split}
&\|v^k\partial_v^l\partial_x^qg(t)\|_{\widetilde{L}^2_{v}(B^{1/2}_{2,1})}=\sum_{p\geq-1}2^{\frac{p}{2}}\|v^k\partial_v^l\partial_x^q\Delta_pg(t)\|_{L^2_{x,v}}
\\&
\leq\sum_{p\geq-1}\sum_{n=0}^{+\infty}2^{\frac{p}{2}}\|\partial_x^q\Delta_pg_n(t)\|_{L^2_{x}}\|v^k\partial_v^le_n\|_{L^2_{v}}
\\&
\leq C_1\Big(\frac{C_2}{\inf(\varepsilon^{\frac{3s+1}{2s(s+1)}},1)}\Big)^{k+l}(k!)^{\frac{3s+1}{2s(s+1)}}(l!)^{\frac{3s+1}{2s(s+1)}}
\sum_{n=0}^{+\infty}\sum_{p\geq-1}2^{\frac{p}{2}}\|\partial_x^q\Delta_pg_n(t)\|_{L^2_{x}}
\\&\qquad
\times\Big((1-\delta_{n,0})\exp\Big(\varepsilon\frac{3s+1}{2s(s+1)}n^{\frac{s(s+1)}{3s+1}}\Big)+\delta_{n,0}\Big),
\end{split}
\end{equation}
where $\delta_{n,0}$ stands for the Kronecker delta, i.e., $\delta_{n,0}=1$ if $n=0, \delta_{n,0}=0$ if $n\neq0$. It follows from \eqref{g-limit3} that for all $0\leq t\leq T$,
\begin{align*}
&\Big\|\exp\Big(\frac{c_0t}{2}(\mathcal{H}^{\frac{s+1}{2}}+\langle D_x\rangle^{\frac{3s+1}{2s+1}})^{\frac{2s}{3s+1}}\Big)g\Big\|_{\widetilde{L}^2_{v}(B^{1/2}_{2,1})}
\\&
=\sum_{p\geq-1}2^{\frac{p}{2}}\Big(\sum_{n=0}^{+\infty}\Big\|\exp\Big(\frac{c_0t}{2}((n+\frac{1}{2})^{\frac{s+1}{2}}+\langle D_x\rangle^{\frac{3s+1}{2s+1}})^{\frac{2s}{3s+1}}\Big)\Delta_pg_n(t)\Big\|^2_{L^2_{x}}\Big)^{1/2}
\\&
\leq Ce^{CT}\|g_{0}\|_{\widetilde{L}^2_{v}(B^{1/2}_{2,1})},
\end{align*}
which implies that for all $0\leq t\leq T$,
\begin{equation}\label{Gevrey-2}
\sum_{p\geq-1}2^{\frac{p}{2}}\sup_{n\geq0}\Big\|\exp\Big(\frac{c_0t}{2}((n+\frac{1}{2})^{\frac{s+1}{2}}+\langle D_x\rangle^{\frac{3s+1}{2s+1}})^{\frac{2s}{3s+1}}\Big)\Delta_pg_n(t)\Big\|_{L^2_{x}}
\leq Ce^{CT}\|g_{0}\|_{\widetilde{L}^2_{v}(B^{1/2}_{2,1})}.
\end{equation}
Then we obtain that for all $n,q\geq0, p\geq-1$,
\begin{equation}\label{Gevrey-3}
\begin{split}
\|\partial_x^q\Delta_pg_n(t)\|_{L^2_{x}}&
=\Big(\frac{1}{2\pi}\int_{\mathbb{R}}|\xi|^{2q}\exp\Big(-c_0t((n+\frac{1}{2})^{\frac{s+1}{2}}+\langle\xi\rangle^{\frac{3s+1}{2s+1}})^{\frac{2s}{3s+1}}\Big)
\\&\quad
\times\Big|\exp\Big(\frac{c_0t}{2}\Big((n+\frac{1}{2})^{\frac{s+1}{2}}+\langle\xi\rangle^{\frac{3s+1}{2s+1}}\Big)^{\frac{2s}{2s+1}}\Big)
\widehat{\Delta_pg_n}(\xi)\Big|^2d\xi\Big)^{1/2}
\\&
\leq\Big(\frac{2s+1}{sc_0t}\Big)^{\frac{2s+1}{2s}q}(q!)^{\frac{2s+1}{2s}}\exp\Big(-\frac{c_0t}{2}\Big(n+\frac{1}{2}\Big)^{\frac{s(s+1)}{3s+1}}\Big)
\\&\quad
\times\Big\|\exp\Big(\frac{c_0t}{2}\Big((n+\frac{1}{2})^{\frac{s+1}{2}}+\langle D_x\rangle^{\frac{3s+1}{2s+1}}\Big)^{\frac{2s}{2s+1}}\Big)\Delta_pg_n(t)\Big\|_{L^2_{x}},
\end{split}
\end{equation}
where we used the following inequality that, $\forall0<t\leq T, \ \ \forall q\geq0, \ \ \forall\xi\in\mathbb{R},$
\begin{align*}
\langle\xi\rangle^{2q}e^{-c_0t\langle\xi\rangle^{\frac{2s}{2s+1}}}\leq\Big(\frac{2s+1}{sc_0t}\Big)^{\frac{2s+1}{s}q}(q!)^{\frac{2s+1}{s}}.
\end{align*}
Then it follows from \eqref{Gevrey-1}, \eqref{Gevrey-2} and \eqref{Gevrey-3} that for all $0\leq t\leq T, k,l,q\geq0$,
\begin{equation*}
\begin{split}
&\|v^k\partial_v^l\partial_x^qg(t)\|_{\widetilde{L}^2_{v}(B^{1/2}_{2,1})}
\\&
\leq C_1\Big(\frac{C_2}{\inf(\varepsilon^{\frac{3s+1}{2s(s+1)}},1)}\Big)^{k+l}\Big(\frac{2s+1}{sc_0t}\Big)^{\frac{2s+1}{2s}q}
(k!)^{\frac{3s+1}{2s(s+1)}}(l!)^{\frac{3s+1}{2s(s+1)}}(q!)^{\frac{2s+1}{2s}}
\\&\quad
\times\sum_{n=0}^{+\infty}\sum_{p\geq-1}2^{\frac{p}{2}}\Big\|\exp\Big(\frac{c_0t}{2}\Big((n+\frac{1}{2})^{\frac{s+1}{2}}+\langle D_x\rangle^{\frac{3s+1}{2s+1}}\Big)^{\frac{2s}{2s+1}}\Big)\Delta_pg_n(t)\Big\|_{L^2_{x}}
\\&\quad
\times\Big((1-\delta_{n,0})\exp\Big(\varepsilon\frac{3s+1}{2s(s+1)}n^{\frac{s(s+1)}{3s+1}}-\frac{c_0}{4}t\Big(n+\frac{1}{2}\Big)^{\frac{s(s+1)}{3s+1}}\Big)
\\&\qquad\quad
+\delta_{n,0}\exp\Big(-\frac{c_0}{4}t\Big(n+\frac{1}{2}\Big)^{\frac{s(s+1)}{3s+1}}\Big)\Big),
\\&
\leq Ce^{CT}\|g_{0}\|_{\widetilde{L}^2_{v}(B^{1/2}_{2,1})}C_1\Big(\frac{C_2}{\inf(\varepsilon^{\frac{3s+1}{2s(s+1)}},1)}\Big)^{k+l}\Big(\frac{2s+1}{sc_0t}\Big)^{\frac{2s+1}{2s}q}
(k!)^{\frac{3s+1}{2s(s+1)}}(l!)^{\frac{3s+1}{2s(s+1)}}(q!)^{\frac{2s+1}{2s}}
\\&\quad
\times\sum^{+\infty}_{n=0}\Big((1-\delta_{n,0})\exp\Big(\varepsilon\frac{3s+1}{2s(s+1)}n^{\frac{s(s+1)}{3s+1}}-\frac{c_0}{4}t\Big(n+\frac{1}{2}\Big)^{\frac{s(s+1)}{3s+1}}\Big)
\\&\quad
\qquad\qquad+\delta_{n,0}\exp\Big(-\frac{c_0}{4}t\Big(n+\frac{1}{2}\Big)^{\frac{s(s+1)}{3s+1}}\Big)\Big).
\end{split}
\end{equation*}
If we choose
$$
\varepsilon=\frac{s(s+1)c_0t}{12s+4}>0,
$$
then there exist some constants $C_3,C_4>0$ such that for all
\begin{equation}\label{Gevrey-5}
\begin{split}
&\|v^k\partial_v^l\partial_x^qg(t)\|_{\widetilde{L}^2_{v}(B^{1/2}_{2,1})}
\\&
\leq C_3C_4^{k+l+q}\frac{F(t)}{t^{\frac{3s+1}{s(s+1)}(k+l)+\frac{2s+1}{2s}q}}
(k!)^{\frac{3s+1}{2s(s+1)}}(l!)^{\frac{3s+1}{2s(s+1)}}(q!)^{\frac{2s+1}{2s}}\|g_{0}\|_{\widetilde{L}^2_{v}(B^{1/2}_{2,1})},
\end{split}
\end{equation}
where
\begin{align*}
F(x)=\sum_{n=0}^{+\infty}\exp\Big(-\frac{c_0x}{8}\Big(n+\frac{1}{2}\Big)^{\frac{s(s+1)}{3s+1}}\Big), \ \ x>0.
\end{align*}
For any $x>0$ we obtain that
\begin{align*}
F(x)
&=x^{-(\frac{3s+1}{s(s+1)}+\alpha)}\sum_{n=0}^{+\infty}\Big(\frac{c_0x}{8}\Big(n+\frac{1}{2}\Big)^{\frac{s(s+1)}{3s+1}}\Big)^{\frac{3s+1}{s(s+1)}+\alpha}
\exp\Big(-\frac{c_0x}{8}\Big(n+\frac{1}{2}\Big)^{\frac{s(s+1)}{3s+1}}\Big)
\\&\quad
\times\frac{1}{(\frac{c_0}{8})^{\frac{3s+1}{s(s+1)}+\alpha}(n+\frac{1}{2})^{1+\frac{\alpha s(s+1)}{3s+1}}}
\\&
\lesssim x^{-(\frac{3s+1}{s(s+1)}+\alpha)}\|z^{\frac{3s+1}{s(s+1)}+\alpha}e^{-z}\|_{L^\infty([0,+\infty))}
\sum_{n=0}^{+\infty}\frac{1}{(\frac{c_0}{8})^{\frac{3s+1}{s(s+1)}+\alpha}(n+\frac{1}{2})^{1+\frac{\alpha s(s+1)}{3s+1}}}
\\&
\lesssim x^{-(\frac{3s+1}{s(s+1)}+\alpha)}.
\end{align*}
with a positive parameter $\alpha>0$. We deduce from \eqref{Gevrey-5} that for $\forall\alpha>0$, there exist some constants $C_5, C_6>0$ such that for all $0\leq t\leq T, k,l,q\geq0$,
\begin{equation*}
\begin{split}
\|v^k\partial_v^l\partial_x^qg(t)\|_{\widetilde{L}^2_{v}(B^{1/2}_{2,1})}\leq\frac{C_5C_6^{k+l+q}}{t^{\frac{3s+1}{2s(s+1)}(k+l+2)+\frac{2s+1}{2s}q+\alpha}}
(k!)^{\frac{3s+1}{2s(s+1)}}(l!)^{\frac{3s+1}{2s(s+1)}}(q!)^{\frac{2s+1}{2s}}\|g_{0}\|_{\widetilde{L}^2_{v}(B^{1/2}_{2,1})}.
\end{split}
\end{equation*}
This proves the Gevrey smoothing property in Theorem \ref{Main}.

\section{Appendix}\label{S5}

\subsection{Hermite functions}\label{S5-1}

The standard Hermite functions $(\phi_{n})_{n\in \mathbb{N}}$ are defined for $v\in\mathbb{R}$,
\begin{align*}
\phi_n(v)=\frac{(-1)^n}{\sqrt{2^nn!\sqrt{\pi}}}e^{\frac{v^2}{2}}\frac{d^n}{dv^n}(e^{-\frac{v^2}{2}})
=-\frac{1}{\sqrt{2^nn!\sqrt{\pi}}}(v-\frac{d}{dv})^n(e^{-\frac{v^2}{2}})=\frac{a_+^n\phi_{0}}{\sqrt{n!}},
\end{align*}
where $a_+$ is the creation operator
$$
a_+=\frac{1}{\sqrt{2}}\Big(v-\frac{d}{dv}\Big).
$$
The family $(\phi_n)_{n\in \mathbb{N}}$ is an orthonormal basis of $L^2(\mathbb{R})$. We set for $n\in\mathbb{N},v\in\mathbb{R}$,
\begin{align*}
e_n(v)=2^{-1/4}\phi_n(2^{-1/2}v), \ \ \ e_n=\frac{1}{\sqrt{n!}}(\frac{v}{2}-\frac{d}{dv})^ne_0.
\end{align*}
The family $(e_{n})_{n\in \mathbb{N}}$ is an orthonormal basis of $L^2(\mathbb{R})$ composed by the eigenfunctions of the harmonic oscillator
$$
\mathcal{H}=-\Delta_{v}+\frac{v^2}{4}=\sum_{n\geq0}(n+\frac{1}{2})\mathbb{P}_n, \ \ 1=\sum_{n\geq0}\mathbb{P}_n,
$$
where $\mathbb{P}_n$ stands for the orthogonal projection
$$
\mathbb{P}_nf=(f,e_n)_{L^2(\mathbb{R}_v)}e_n.
$$
It satisfies the identities
\begin{align}\label{H3}
A_+e_n=\sqrt{n+1}e_{n+1}, \ \ \ A_-e_n=\sqrt{n}e_{n-1},
\end{align}
where
\begin{align}\label{H4}
A_{\pm}=\frac{v}{2}\mp\frac{d}{dv}.
\end{align}

\subsection{The Kac collision operator}\label{S5-2}

For $\varphi$ a function defined on $\mathbb{R}$, we denote its even part by
$$
\breve{\varphi}(\theta)=\frac{1}{2}(\varphi(\theta)+\varphi(-\theta)).
$$
The following lemma is given by \cite{LMPX1} (Lemma A.1):
\begin{lemma}\label{Kac}
Let $\nu\in L^1_{\mathrm{loc}}(\mathbb{R}^{\ast})$ be an even function such that $\theta^2\nu(\theta)\in L^1(\mathbb{R})$. Then, the mapping
\begin{align*}
\varphi\in C^2_c(\mathbb{R})\mapsto\lim_{\varepsilon\rightarrow0_+}\int_{|\theta|\geq\varepsilon}\nu(\theta)(\varphi(\theta)-\varphi(0))d\theta
=\int_0^1\int_{\mathbb{R}}(1-t)\theta^2\nu(\theta)\varphi''(t\theta)d\theta dt,
\end{align*}
defines a distribution of order 2 denoted $\mathrm{fp}(\nu)$. The linear form $\mathrm{fp}(\nu)$ can be extended to $C^{1,1}$ functions ($C^1$ functions whose
second derivative is $L^{\infty}$). For $\varphi\in C^{1,1}$ satisfying $\varphi(0)=0$, the function $\nu\breve{\varphi}$ belongs to $L^1(\mathbb{R})$ and
$$
\langle\mathrm{fp}(\nu),\varphi\rangle=\int\nu(\theta)\breve{\varphi}(\theta)d\theta.
$$
\end{lemma}

Let $f,g\in\mathcal{S}(\mathbb{R})$ be Schwartz functions. We define
\begin{align*}
F_{f,g}(\underbrace{v,v_{*}}_w)=f(v)g(v_{*}), \ \ \varphi_{f,g}(\theta,v)=\int_{\mathbb{R}}(F_{f,g}(R_{\theta}w)-F_{f,g}(w))dv_{*},
\end{align*}
where $R_{\theta}$ stands for the rotation of angle $\theta$ in $\mathbb{R}^2$,
\begin{equation*}
R_{\theta}=
\left(
\begin{aligned}
&\cos\theta \ \ \ -\sin\theta \\
&\sin\theta \ \ \ \cos\theta
\end{aligned} \right)=\exp(\theta J), \ \ J=R_{\frac{\pi}{2}}.
\end{equation*}
The second derivative with respect to $\theta$ of the function $\varphi_{f,g}$ is in $\mathcal{S}(\mathbb{R})$ uniformly with respect to $\theta$. We define the non-cutoff Kac operator as
$$
K(g,f)(v)=\langle\mathrm{fp}(\mathbb{I}_{(-\frac{\pi}{4},\frac{\pi}{4})}\beta),\varphi_{f,g}(\cdot,v)\rangle,
$$
when $\beta$ is a function satisfying \eqref{A-0}. Since $\Phi_{f,g}(0,v)\equiv0$, Lemma \ref{Kac} allows to replace the finite part by the absolutely converging integral
$$
K(g,f)(v)=\int_{|\theta|\leq\frac{\pi}{4}}\beta(\theta)\Big(\int_{\mathbb{R}}((\breve{g})'_\ast f'-(\breve{g})_\ast f)dv_{*}\Big)d\theta
=K(\breve{g},f)(v).
$$
It was established in \cite{LMPX1} (Lemma A.2) that $K(g,f)\in\mathcal{S}(\mathbb{R})$, when $g,f\in\mathcal{S}(\mathbb{R})$. We also recall the Bobylev formula \cite{Boby} providing an explicit formula for the Fourier transform of the Kac operator
$$
\widehat{K(g,f)}(\xi)=\int_{|\theta|\leq\frac{\pi}{4}}\beta(\theta)\Big[\widehat{\breve{g}}(\xi\sin\theta)\widehat{f}(\xi\cos\theta)
-\widehat{g}(0)\widehat{f}(\xi)\Big]d\theta,
$$
when $f,g\in\mathcal{S}(\mathbb{R})$. The proof of this formula may be found in \cite{LMPX1} (Lemma A.4).

\subsection{Linear inhomogeneous Kac operator}\label{S5-3}

We recall some spectral analysis for the linear inhomogeneous Kac operator that are given in \cite{LMPX1,LMPX2}. Consider the operator acting in the velocity variable
\begin{align}\label{P}
P=iv\xi+a_0^{w}(v,D_v),
\end{align}
with parameter $\xi\in\mathbb{R}$, where the operator $A=a_0^{w}(v,D_v)$ stands for the pseudo-differential operator
$$
a_0^{w}(v,D_v)u=\frac{1}{2\pi}\int_{\mathbb{R}^2}e^{i(v-w)\eta}a_0(\frac{v+w}{2},\eta)u(w)dwd\eta
$$
defined by the Weyl quantization of the symbol
\begin{align*}
a_0(v,\eta)=c_0(1+\eta^2+\frac{v^2}{4})^{s}
\end{align*}
with some constants $c_0>0,0<s<1$. This operator corresponds to the principle part of the linear inhomogeneous Kac operator
$$
v\partial_x+\mathcal{K}
$$
on the Fourier side in the position variable.

Let $\psi$ be a $C^{\infty}_0(\mathbb{R},[0,1])$ function satisfying
$$
\psi=1 \ \ \text{on} \ \ [-1,1], \ \ \mathrm{supp} \  \psi\subset[-2,2].
$$
We define the real-valued symbol
\begin{align}\label{G}
m=-\frac{\xi\eta}{\lambda^{\frac{2s+2}{2s+1}}}\psi(\frac{\eta^2+v^{2}}{\lambda^{\frac{2}{2s+1}}})
\end{align}
with
\begin{align*}
\lambda=(1+v^2+\eta^2+\xi^2)^{\frac{1}{2}}.
\end{align*}

It holds that the following equivalence of norms
\begin{align}\label{PH}
\forall r\in\mathbb{R}, \exists C_r>0, \frac{1}{C_r}\|\mathcal{H}^{r}u\|_{L^2}\leq\Big\|\mathrm{Op}^w\Big(\Big(1+\eta^2+\frac{v^2}{4}\Big)^r\Big)u\Big\|_{L^2}\leq C_r\|\mathcal{H}^{r}u\|_{L^2},
\end{align}
where $\mathcal{H}=-\Delta_v+\frac{v^2}{4}$ stands for the harmonic oscillator.
%
%

\subsection{Fundamental inequalities}\label{S5-4}
We recall some estimates for the Kac collision operator along with the Hermite basis, see \cite{LMPX2} for details.
\begin{lemma}\label{Linear}
Let $P$ be the operator defined in \eqref{P} and $M=m^{w}$ the self-adjoint operator defined by the Weyl quantization of the symbol \eqref{G}. Then, the
operator $M$ is uniformly bounded on $L^2(\mathbb{R}_v)$ with respect to the parameter $\xi\in\mathbb{R}$, and there exist some positive constants $0<\varepsilon_0\leq1,c_1,c_2>0$ such that for all $0<\varepsilon\leq\varepsilon_0,u\in\mathcal{S}(\mathbb{R}_v),\xi\in\mathbb{R}$,
\begin{align*}
\mathrm{Re}(Pu,(1-\varepsilon M)u)\geq c_1\|\mathcal{H}^{\frac{s}{2}}u\|^{2}_{L^2(\mathbb{R}_v)}+c_1\varepsilon\langle\xi\rangle^{\frac{2s}{2s+1}}
\|u\|^{2}_{L^2(\mathbb{R}_v)}-c_2\|u\|^{2}_{L^2(\mathbb{R}_v)},
\end{align*}
where $\mathcal{H}=-\Delta_{v}+\frac{v^2}{4}$ stands for the harmonic oscillator.
\end{lemma}

Since the operator $\Delta_p$ acts on the position variable $x$ only, we obtain the following conclusion based on Lemma \ref{Linear}.

\begin{lemma}\label{Linear-B}
Let $P$ be the operator defined in \eqref{P} and $M=m^{w}$ the self-adjoint operator defined by the Weyl quantization of the symbol \eqref{G}. Then, the
operator $M$ is uniformly bounded on $L^2(\mathbb{R}_v)$ with respect to the parameter $\xi\in\mathbb{R}$, and there exist some positive constants $0<\varepsilon_0\leq1,c_3,c_4>0$ such that for all $0<\varepsilon\leq\varepsilon_0,u\in\mathcal{S}(\mathbb{R}_v),\xi\in\mathbb{R},p\geq-1$,
\begin{equation*}
\begin{split}
&\mathrm{Re}(P\Delta_pu,(1-\varepsilon M)\Delta_pu)_{L^2(\mathbb{R}^2_{v})}
\\&
\geq c_3\|\mathcal{H}^{\frac{s}{2}}\Delta_pu\|^2_{L^2(\mathbb{R}^2_{v})}+c_3\langle\xi\rangle^{\frac{2s}{2s+1}}\|\Delta_pu\|^2_{L^2(\mathbb{R}^2_{v})}
-c_4\|\Delta_pu\|^2_{L^2(\mathbb{R}^2_{v})},
\end{split}
\end{equation*}
where $\mathcal{H}=-\Delta_{v}+\frac{v^2}{4}$ stands for the harmonic oscillator.
\end{lemma}

\begin{lemma}\label{J10-1}
Let $(e_n)_{n\geq0}$ be the Hermite basis of $L^2(\mathbb{R})$ describes in Section \ref{S5-1}. We have
\begin{align*}
\Gamma(e_{k},e_{l})=\alpha_{k,l}e_{k+l}, \ \ \ k,l\geq0,
\end{align*}
with
\begin{align*}
&\alpha_{2n,m}=\sqrt{C^{2n}_{2n+m}}\int_{-\frac{\pi}{4}}^{\frac{\pi}{4}}\beta(\theta)(\sin\theta)^{2n}(\cos\theta)^md\theta, \ \ \ n\geq1, \ m\geq0,
\\&\alpha_{0,m}=\int_{-\frac{\pi}{4}}^{\frac{\pi}{4}}\beta(\theta)((\cos\theta)^m-1)d\theta, \ \ m\geq1; \ \ \alpha_{0,0}=\alpha_{2n+1,m}=0, \ \ n,m\geq0,
\end{align*}
where $C_{n}^{k}=\frac{n!}{k!(n-k)!}$ stands for the binomial coefficients.
\end{lemma}

\begin{lemma}\label{J10-2}
We assume that the cross section satisfies \eqref{A-5} with $0<s<1$. Then, there exists a positive constant $C>0$ such that for all $n\geq1, m\geq0$,
\begin{align*}
0\leq\alpha_{2n,m}=\sqrt{C^{2n}_{2n+m}}\int_{-\frac{\pi}{4}}^{\frac{\pi}{4}}\beta(\theta)(\sin\theta)^{2n}(\cos\theta)^md\theta\leq\frac{C}{n^{\frac{3}{4}}}\tilde{\mu}_{n,m},
\end{align*}
where $\tilde{\mu}_{n,m}=(1+\frac{m}{n})^{s}(1+\frac{n}{m+1})^{\frac{1}{4}}$.
\end{lemma}

In \cite{LMPX2}, the authors showed a key estimate on the Hermite functions.

\begin{lemma}\label{J10-5}
It holds that
\begin{align*}
&\forall n,k,l\geq0, \ \ \|v^k\partial_v^le_n\|_{L^2(\mathbb{R})}\leq2^k\sqrt{\frac{(k+l+n)!}{n!}},
\\&\forall r\geq\frac{1}{2}, \forall\varepsilon>0, \forall n,k,l\geq0,
\\&
\|v^k\partial_v^le_n\|_{L^2(\mathbb{R})}\leq\sqrt{2}((1-\delta_{n,0})\exp(\varepsilon rn^{\frac{1}{2r}})+\delta_{n,0})
\Big(\frac{2^{\frac{3}{2}+r}e^r}{\inf(\varepsilon^r,1)}\Big)^{k+l}(k!)^r(l!)^r,
\end{align*}
where $\delta_{n,0}$ stands for the Kronecker delta i.e., $\delta_{n,0}=1$ if $n=0,\delta_{n,0}=0$ if $n\neq0$.
\end{lemma}

\subsection{Gelfand-Shilov regularity}\label{S5-5}

We refer the reader to the works \cite{GS,GPR,NR,TKNN} and the references herein for extensive expositions of the Gelfand-Shilov regularity theory. The Gelfand-Shilov spaces $\mathcal{S}^{\mu}_{\nu}(\mathbb{R})$ may also be characterized as the spaces of Schwartz functions belonging to the Gevrey space $G^{\mu}(\mathbb{R})$, whose Fourier transforms belong to the Gevrey space $G^{\nu}(\mathbb{R})$. That is, $f\in \mathcal{S}(\mathbb{R})$
satisfying
\begin{align*}
\exists C\geq0, \varepsilon>0, \ \ |f(v)|\leq Ce^{-\varepsilon|v|^{1/\nu}},\ \  v\in\mathbb{R}, \ \ \  |\widehat{f}(\xi)|\leq Ce^{-\varepsilon|\xi|^{1/\mu}}, \ \ \xi\in\mathbb{R}.
\end{align*}
In particular, we notice that Hermite functions belong to the symmetric Gelfand-Shilov spaces $\mathcal{S}^{1/2}_{1/2}(\mathbb{R})$. More generally, the symmetric Gelfand-Shilov spaces $\mathcal{S}^{\mu}_{\mu}(\mathbb{R})$, with $\mu\geq1/2$, can be characterized through the decomposition into the Hermite basis
$(e_n)_{n\geq0}$ see e.g. \cite{TKNN} (Proposition 1.2)
\begin{align*}
f\in\mathcal{S}^{\mu}_{\nu}(\mathbb{R})&\Leftrightarrow f\in L^2(\mathbb{R}), \  \exists t_0>0, \ \|((f,e_n)_{L^2}\exp(t_0n^{\frac{1}{2\mu}}))_{n\geq0}\|_{l^2(\mathbb{N})}
\\&\Leftrightarrow f\in L^2(\mathbb{R}), \  \exists t_0>0, \ \|e^{t_0\mathcal{H}^{1/2\mu}}f \|_{L^2},
\end{align*}
where $\mathcal{H}=-\Delta_v+\frac{v^2}{4}$ is the harmonic oscillator and $(e_n)_{n\geq0}$ is Hermite basis given by Section \ref{S5-1}.

\subsection{Fundamental properties in Besov space}\label{S5-6}

For convenience of reader, we recall some fundamental properties in Besov space which are frequently used in this paper. The Littlewood-Paley decomposition is ``almost" orthogonal in the following sense.

\begin{lemma}
For any $u\in\mathcal{S}'(\mathbb{R}^{d})$ and $v\in\mathcal{S}'(\mathbb{R}^{d})$, the following properties hold:
\begin{align*}
&\quad\Delta_{q}\Delta_{q}u\equiv0, \quad\quad \text{if} \quad |p-q|\geq2,
\\&
\Delta_{q}(S_{p-1}u\Delta_{p}v)\equiv0, \quad\quad \text{if} \quad |p-q|\geq5.
\end{align*}
\end{lemma}

Additionally, the standard Young's inequality for convolution products implies that

\begin{lemma}\label{J7-3}
Let $1\leq p\leq\infty$ and $u\in L^{p}_{x}$, then there exists a constant $C>0$ independent of $p,q$ and $u$ such that
\begin{align*}
\|\Delta_{q}u\|_{L^{p}_{x}}\leq C\|u\|_{L^{p}_{x}}, \ \ \ \  \|S_{q}u\|_{L^{p}_{x}}\leq C\|u\|_{L^{p}_{x}}.
\end{align*}
\end{lemma}

The following embedding properties in Besov spaces have been used several times.

\begin{lemma}\label{J7-1}
Let $s\in \mathbb{R}$.

(1) If $\widetilde{s}<s$, then $B^{s}_{2,1}\hookrightarrow B^{\widetilde{s}}_{2,1}$;
(2) $B^{1/2}_{2,1}(\mathbb{R})\hookrightarrow L^\infty(\mathbb{R})$ and $\dot{B}^{1/2}_{2,1}(\mathbb{R})\hookrightarrow L^\infty(\mathbb{R})$.
\end{lemma}

According to \cite{XK}, we have the following topology between homogeneous Chemin-Lerner spaces and nonhomogeneous Chemin-Lerner spaces.

\begin{lemma}
Let $1\leq \varrho,q,r\leq\infty$ and $s>0$. Then we have
\begin{align*}
\|\nabla_{x}\cdot\|_{\widetilde{L}^{\varrho}_{T}(\dot{B}^{s}_{p,r})}\sim\|\cdot\|_{\widetilde{L}^{\varrho}_{T}(\dot{B}^{s+1}_{p,r})}, \ \ \ \ \|\cdot\|_{\widetilde{L}^{\varrho}_{T}(\dot{B}^{s}_{p,r})}\lesssim\|\cdot\|_{\widetilde{L}^{\varrho}_{T}(B^{s}_{p,r})}.
\end{align*}
\end{lemma}





Finally, it follows from \cite{DLX-2016} that

\begin{lemma}\label{J7-5}
Let $s\in\mathbb{R}$ and $1\leq\varrho_{1},\varrho_{2},p,r\leq\infty$.

(1) If $r\leq\min\{\varrho_{1},\varrho_{2}\}$, then it holds that
\begin{align*}
\|u\|_{L^{\varrho_{1}}_{T}L^{\varrho_{2}}_{v}(B^{s}_{p,r})}\leq\|u\|_{\widetilde{L}^{\varrho_{1}}_{T}\widetilde{L}^{\varrho_{2}}_{v}(B^{s}_{p,r})}.
\end{align*}

(2) If $r\geq\max\{\varrho_{1},\varrho_{2}\}$, then it holds that
\begin{align*}
\|u\|_{L^{\varrho_{1}}_{T}L^{\varrho_{2}}_{v}(B^{s}_{p,r})}\geq\|u\|_{\widetilde{L}^{\varrho_{1}}_{T}\widetilde{L}^{\varrho_{2}}_{v}(B^{s}_{p,r})}.
\end{align*}
\end{lemma}

\bigskip

\noindent{\bf Acknowledgements.}
The first author is supported by the China Scholarship Council(CSC).
The second author is supported by the National Natural Science Foundation of China (11701578).
The research of C.-J. Xu is supported by``The Fundamental Research Funds for Central Universities of China".
The research of J. Xu is partially supported by the National Natural Science Foundation of China (11871274)
and ``The Fundamental Research Funds for the Central Universities of China" (NE2015005).

\end{document}